\documentclass{amsart}
\usepackage{etex}
\usepackage[style=alphabetic,doi=false,isbn=false,url=false,maxbibnames=99]{biblatex}
\usepackage{fixltx2e}

\bibliography{emertongee}

\RequirePackage{amsmath}
\RequirePackage{amssymb}
\RequirePackage{amsxtra}
\RequirePackage{amsfonts}
\RequirePackage{latexsym}
\RequirePackage{euscript}
\RequirePackage{amscd}
\RequirePackage{amsthm}
\RequirePackage{xypic}
\xyoption{all}
\usepackage{dsfont}

\newcommand{\bbmu}{\mu}
 
\newcommand{\Kbar}{{\overline{K}}}
\newcommand{\Wgen}{W_{\operatorname{gen}}}

\newcommand{\xbar}{\overline{x}}
\newcommand{\zbar}{\overline{z}}

\def\A{\mathbb A}

\def\C{\mathbb C}
\def\F{\mathbb F}

\def\Q{\mathbb{Q}}
\def\R{\mathbb{R}}
\def\T{\mathbb{T}}

\def\Z{\mathbb{Z}}
\def\Fbar{\overline{\F}}

\def\Qbar{{\overline{\Q}}}

\def\Zbar{\overline{\Z}}

\def\m{\mathfrak m}
\newcommand{\st}{\mathrm{st}}

\def\chibar{\overline{\chi}}

\def\unif{\varpi}

\def\et{\mathrm{\acute{e}t}}

\def\lalg{\mathrm{l.alg}}

\def\ab{\mathrm{ab}}

\def\SL{\mathrm{SL}}

\def\GL{\mathrm{GL}}

\def\Gal{\mathrm{Gal}}

\def\End{\mathrm{End}}

\def\Hom{\mathop{\mathrm{Hom}}\nolimits}

\def\Spec{\mathop{\mathrm{Spec}}\nolimits}
\def\Spf{\mathop{\mathrm{Spf}}\nolimits}
\def\Frob{\mathop{\mathrm{Frob}}\nolimits}

\def\Ind{\mathop{\mathrm{Ind}}\nolimits}

\def\rhobar{\overline{\rho}}

\def\pst{\mathrm{pst}}

\def\W{\mathrm{W}}
\def\WD{\mathrm{WD}}

\def\m{\mathfrak{m}}

\def\iso{\buildrel \sim \over \longrightarrow}

\def\OT{\mathop{\mathrm{OT}}\nolimits}

\swapnumbers

\newcommand{\onto}{\twoheadrightarrow}

\newcommand{\into}{\hookrightarrow}

\newcommand{\To}{\longrightarrow}
\newcommand{\isoto}{\stackrel{\sim}{\To}}

\newcommand{\FBrModdd}[1][r]{\text{$k_E$-$\operatorname{BrMod}_{\mathrm {dd}}^{#1}$}}

\newcommand{\textT}{\mathrm{T}}

\newcommand{\Tst}{\textT_{\st}}

\newlength{\ownl}

\newcommand{\chara}{{\operatorname{char}\,}}

\newcommand{\CHom}{\mathop{Hom}\nolimits}

\newcommand{\tr}{{\operatorname{tr}\,}}

\newcommand{\wt}[1]{\widetilde{#1}}

\newcommand{\Gm}{{\mathbb{G}_m}}

\newcommand{\cris}{{\operatorname{cris}}}

\newcommand{\op}{{\operatorname{op}}}

\newcommand{\G}{{\mathbb{G}}}

\newcommand{\M}{{\mathcal{M}}}

\newcommand{\CO}{{\mathcal{O}}}

\newcommand{\cA}{\mathcal{A}}

\newcommand{\cF}{\mathcal{F}}
\newcommand{\cG}{\mathcal{G}}

\newcommand{\cM}{\mathcal{M}}
\newcommand{\cN}{\mathcal{N}}
\newcommand{\cO}{\mathcal{O}}

\newcommand{\cU}{\mathcal{U}}

\newcommand{\cX}{\mathcal{X}}
\newcommand{\cY}{\mathcal{Y}}

\newcommand{\gp}{{\mathfrak{p}}}

\newcommand{\tD}{\widetilde{{D}}}

\newcommand{\tJ}{\widetilde{{J}}}

\newcommand{\tS}{\widetilde{{S}}}

\newcommand{\tV}{\widetilde{{V}}}

\newcommand{\tj}{\widetilde{{j}}}

\newcommand{\etabar         }{\overline{\eta}}

 \newcommand{\omegat   }{\widetilde{\omega}}

 \newcommand{\p}{\mathfrak{p}}

\newcommand{\rbar}{\bar{r}}

 \newcommand{\Qp}{{\Q_p}}
\newcommand{\Qpn}{{\Q_{p^n}}}
\newcommand{\GQp}{{G_{\Q_p}}}
\newcommand{\IQp}{{I_{\Q_p}}}
\newcommand{\Zp}{{\Z_p}}
\newcommand{\Zptimes}{{\Z_p^\times}}
\newcommand{\Qptimes}{{\Q_p^\times}}
\newcommand{\Zl}{{\Z_l}}
\newcommand{\Ql}{\Q_l} 
\newcommand{\Qpbar}{{\overline{\Q}_p}}
\newcommand{\Zpbar}{{\overline{\Z}_p}}

\newcommand{\Fpbar}{{\overline{\F}_p}}
\newcommand{\Fptimes}{{{\F}_p^\times}}
\newcommand{\Fpbartimes}{{\overline{\F}^\times_p}}

\newcommand{\Fp}{{\F_p}}

\newcommand{\cshO}[1]{\widehat{\cO_{#1}^{\mathrm{\, sh}}}}

\usepackage{hyperref}

\newtheorem{theorem}[subsubsection]{Theorem}
\newtheorem{thm}[subsubsection]{Theorem}
\newtheorem{lemma}[subsubsection]{Lemma}
\newtheorem{lem}[subsubsection]{Lemma}

\newtheorem{defn}[subsubsection]{Definition}
\newtheorem{cor}[subsubsection]{Corollary}

\newtheorem{prop}[subsubsection]{Proposition}
\newtheorem{remark}[subsubsection]{Remark}
\newtheorem{rem}[subsubsection]{Remark}

\newtheorem{hyp}[subsubsection]{Hypothesis}

\newtheorem{ithm}{Theorem}

\def\numequation{\addtocounter{subsubsection}{1}\begin{equation}}
\def\nummultline{\addtocounter{subsubsection}{1}\begin{multline}}
\def\anumequation{\addtocounter{subsection}{1}\begin{equation}}

\renewcommand{\theequation}{\arabic{section}.\arabic{subsection}.\arabic{subsubsection}}

\title[$p$-adic Hodge-theoretic properties of mod $p$ \'etale cohomology]
{$p$-adic Hodge-theoretic properties of \'etale cohomology with mod   $p$ coefficients,
and the cohomology of Shimura varieties}
\author{Matthew Emerton and Toby Gee}
\thanks{The first author was supported in part by NSF grant
  DMS-1003339, and the second author was supported in part by NSF
  grant DMS-0841491 and EPSRC Mathematics Platform grant EP/I019111/1}
\address[Matthew Emerton]{Mathematics Department, University of Chicago}
\address[Toby Gee]{Mathematics Department, Imperial College London \newline \quad} 
\email[Matthew Emerton]{emerton@math.uchicago.edu}
\email[Toby Gee]{toby.gee@imperial.ac.uk}

\begin{document}

\maketitle

\begin{abstract}
  We prove vanishing results for the cohomology of unitary Shimura
  varieties with integral coefficients at arbitrary level, and deduce applications to the
  weight part of Serre's conjecture. In order to do this, we show that the mod~$p$ cohomology of a smooth projective variety
  with semistable reduction over $K$, a finite extension of $\Qp$,
  embeds into the reduction modulo $p$ of a semistable Galois
  representation with Hodge--Tate weights in the expected range (at least
  after semisimplifying, in the case of the cohomological degree greater than~$1$).
  \end{abstract}

The aim of this paper is to establish vanishing results for the
cohomology of certain unitary similitude groups. 
 For
example, we prove the following result.

\begin{ithm}
  \label{thm: main vanishing for U(2,1) intro}
  Let $X$ be a projective $U(2,1)$-Shimura variety of some sufficiently small level, and let $\mathcal F$ be a canonical local system of $\Fbar_p$-vector spaces
  on $X$.
  Let $\m$ be a maximal
  ideal of the Hecke algebra acting on the cohomology $H^{\bullet}(X,\mathcal F)$,
  and suppose that there is a Galois
  representation $\rho_{\mathfrak m}:G_F\to\GL_3(\Fpbar)$ associated
  to $\m$. If we suppose further that we have $\SL_3(k)\subset\rho_{\mathfrak{m}}(G_F)\subset\Fpbartimes\SL_3(k)$
  for some finite extension $k/\Fp$, and that $\rho_{\mathfrak m} |_ {G_{\mathbb Q_p}}$ is
  $1$-regular and irreducible, then the localisations $H^i_{\et}(X_\Qbar, \mathcal F)_{\mathfrak m}$
  vanish in degrees $i\ne 2$.
\end{ithm}
(See Corollary \ref{cor: main vanishing for U(2,1)} and Lemma
\ref{lem:big} below,
and see Sections \ref{sec:Breuil} and \ref{sec:Shimura} for the
precise definitions that we are using; for simplicity we work with
$U(2,1)$-Shimura varieties\footnote{These Shimura varieties might more
 properly be called $GU(2,1)$-Shimura varieties; see Section
 \ref{sec:Shimura} for their definition.} over a quadratic imaginary field $F$. Note that ``sufficiently small
level'' means that the compact open subgroup defining the level is
sufficiently small. We say that a Galois representation is associated to a
maximal ideal of a Hecke algebra if there is the usual relation between Hecke
polynomials and characteristic polynomials of Frobeneii at unramified places;
see~\S\ref{subsec:
vanishing theorems for U(n-1,1)} for a precise definition.)

In fact, we prove a version of this result for
$U(n-1,1)$-Shimura varieties 
under weaker assumptions on $\rho_\m$; however, in general we can only prove
vanishing in degrees outside of the range $[n/2,(3n-4)/2]$. 

We also prove the following result, which makes no explicit reference
to a maximal ideal in the Hecke algebra.

\begin{ithm}
\label{thm:main two intro}
Let $X$ and $\mathcal F$ be as in the statement of Theorem~{\em \ref{thm: main vanishing for U(2,1) intro}}.
  If $\rho$ is a three-dimensional irreducible sub-$G_F$-representation 
of the \'etale cohomology group $H^1_{\et}(X_{\Qbar},\mathcal F)$,
then either every irreducible subquotient
of $\rho_{| G_{\mathbb Q_p}}$ is one-dimensional, or else 
$\rho_{| G_{\mathbb Q_p}}$ is not $1$-regular, or else $\rho(G_F)$ is not generated
by its subset of regular elements.
\end{ithm}

Note that in neither theorem do we make any assumption on the level of the
Shimura variety at $p$.

A Galois representation $\rho_\m$ as in the statement of
Theorem \ref{thm: main vanishing for U(2,1) intro} is known to exist if $\m$
corresponds to a system of Hecke eigenvalues arising
from the reduction mod $p$ of the Hecke eigenvalues attached to some automorphic Hecke
eigenform.
Furthermore, recent work of Scholze \cite{1306.2070}
(which appeared after the first version of this paper was written) 
implies that such a representation exists for {\em any} maximal ideal~$\mathfrak m$.

It seems reasonable to believe that any irreducible sub-$G_F$-representation
of any of the \'etale cohomology groups $H^i_{\et}(X_{\Qbar},\mathcal F)$
for any of the Shimura varieties under consideration should in fact
be a constituent of $\rho_{\mathfrak m}$ for some maximal ideal
$\mathfrak m$ of the Hecke algebra.  However, this doesn't seem to be known,
and relating the ``abstract'' $G_F$-representations $\rho_{\mathfrak m}$
to the ``physical'' $G_F$-representations appearing on \'etale cohomology
is one of the problems we have to deal with in proving our results.

\smallskip
{\bf Application to Serre-type conjectures.} We are able to combine our
results with those of \cite{EGH} so as to establish cases of the weight
part of the Serre-type conjecture of \cite{bib:herzig-thesis} for $U(2,1)$.
More precisely, we have the following result
(where the assertion that $\rhobar$ is modular means that the
corresponding system of Hecke eigenvalues occurs in the mod $p$
cohomology of some $U(2,1)$-Shimura variety;
see Theorem \ref{thm:main EGH} and Lemma \ref{lem:big}.)

\begin{ithm}
  \label{thm:main EGH intro}Suppose that $\rho:G_F\to\GL_3(\Fpbar)$
  satisfies $\SL_3(k)\subset\rho(G_F)\subset\Fpbartimes\SL_3(k)$
  for some finite extension $k/\Fp$, that $\rho|_{\GQp}$ is
  irreducible and $1$-regular, and that $\rho$ is modular of some
  strongly generic weight. Then the set of generic weights for which
  $\rho$ is modular is exactly the set predicted by the recipe of
  \cite{bib:herzig-thesis}.
\end{ithm}

\smallskip

{\bf Relationship with a mod $p$ analogue of Arthur's conjecture.}
Arthur has made a quite precise conjecture regarding the systems of Hecke
eigenvalues that appear in the $L^2$-automorphic spectrum of any reductive group
over a number field \cite{Arthur}, which has consequences for the nature of the
Hecke eigenvalues appearing in the cohomology of Shimura varieties \cite[\S 9]{Arthur}.
For our purposes it suffices to describe a qualitative version of these consequences:
namely, Arthur's conjecture 
implies that if $\lambda$ is a system of Hecke eigenvalues appearing in
the degree $i$ cohomology, where $i$ is less than the middle dimension,
then $\lambda$ is attached (in the sense of e.g.\ \cite{BG}, \cite{1206.6730})
 to a reducible Galois representation (i.e.\ one which
factors through a parabolic subgroup of the $L$-group).  

The fragmentary evidence available suggests that a similar statement will be
true for the mod $p$ cohomology of Shimura varieties.  Our
Theorems~\ref{thm: main vanishing for U(2,1) intro} 
and~\ref{thm:main two intro}
give further evidence in this direction.

\smallskip
{\bf $p$-adic Hodge theory.} In order to prove these theorems,
we establish some new results about the $p$-adic Hodge theoretic
properties of the \'etale cohomology of varieties over a number field or $p$-adic field
with coefficients in a field of characteristic $p$. In the first section we establish results about the mod $p$ \'etale cohomology 
of varieties over number fields or $p$-adic fields which, although weaker in
their conclusions, are substantially broader in the scope of their application
than previously known mod $p$ comparison theorems. For example, we
prove the following result (see Theorem \ref{thm:arb cohom} below).

\begin{ithm} 
\label{thm:arb cohom introduction version}
Let $K$ be a finite extension of $\Qp$, and write $G_K$ for
the absolute Galois group of $K$. If $X$ is a
smooth projective variety over $K$ which has semistable reduction,
and if $\rho$ is an irreducible subquotient of the $G_K$-representation
$H^i_{\et}(X_{\overline{K}},\Fpbar)$, then $\rho$
also embeds as a
subquotient of a $G_K$-representation
over $\Fpbar$
which is the reduction
modulo the maximal ideal of a $G_K$-invariant
$\Zpbar$-lattice in a $G_K$-representation
over $\Qpbar$
which is semistable
with Hodge--Tate weights contained in the interval $[-i,0]$.
\end{ithm}

Both the hypotheses and the conclusions of our theorems are rather
precisely tailored to maximise (as far as we are able) their utility
in applications to the analysis of Galois representations occurring in
the cohomology of Shimura varieties, which we give in the third section.

The remaining two sections of the paper are devoted respectively to
using integral $p$-adic Hodge theory (Breuil modules with descent
data) to establish a result related to the reductions of tamely
potentially semistable $p$-adic representations of $G_{\mathbb Q_p}$
(Section~\ref{sec:Breuil}) and to proving some technical results about
group representations (Section~\ref{sec:group}).  The result of
Section~\ref{sec:Breuil} is an essential ingredient in the arguments
of Section~\ref{sec:Shimura}, while the results of
Section~\ref{sec:group} provide sufficient conditions for the various
representation-theoretic hypotheses appearing in the results of
Section~\ref{sec:Shimura} to be satisfied.

\smallskip

{\bf Remark on related papers.} Very general vanishing theorems for
the mod $p$ cohomology of Shimura varieties have been proved by
Lan and Suh (cf. \cite{LanSuh}); however, their results apply only in
situations of good reduction and for coefficients corresponding to
small Serre weights, which makes them unsuitable for the kinds of
applications we have in mind, e.g.\ to the weight part of Serre-type
conjectures. In the ordinary case there is the work of
Mokrane--Tilouine in the Siegel case \cite[\S 9]{MR1944174} and
Dimitrov in the case of Hilbert modular varieties \cite[\S
6.4]{MR2172950}. Finally, in the recent preprint \cite{Shinmodp}, Shin
proves a general vanishing result for cohomology outside of middle
degree for the part of the mod $p$ cohomology which is supercuspidal
at some prime $l\ne p$, by completely different methods to those of
this paper. It seems plausible that via the mod $p$ local Langlands
correspondence for $\GL_n(\Ql)$, Shin's hypothesis could be
interpreted as a condition on the restriction to a decomposition
group at $l$ of the relevant mod $p$ Galois representations, whereas
our conditions involve the restriction to a decomposition group at
$p$, so our results appear to be complementary.

\smallskip

{\bf Acknowledgements.}
We would like to thank Brian Conrad, Florian Herzig, Mark Kisin, Kai-Wen Lan, Madhav Nori,
Bjorn Poonen, Michael Rapoport, Junecue Suh, and Teruyoshi Yoshida for helpful
correspondence and conversations on the subject of this article.
We also thank an anonymous referee for their critical reading of the
paper, which led to the filling of some gaps in the arguments of Subsection~\ref{subsec:equivariant}.

\smallskip

{\bf Conventions.} 
For any field $K$ we let $G_K$ denote a choice of absolute Galois group of $K$.

If $K$ is a finite field then by a Frobenius element in $G_K$ we will always
mean a geometric Frobenius element.  We extend this convention in an evident
way to Frobenius elements at primes in Galois groups of number fields, 
and to Frobenius elements in Galois groups of local fields.

If $K$ is a local field, then we let $\mathcal O_K$ denote the ring of integers
of $K$, we let $I_K$ denote the inertia subgroup of $G_K$,
we let $\W_K$ denote the Weil group of $K$
(the subgroup of $G_K$ consisting of elements whose reduction modulo $I_K$ is
an integral power of Frobenius),
and we let $\WD_K$ denote the Weil--Deligne group of $K$.

If $K$ is a number field, and $v$ is a finite place of $K$, then we
will write $K_v$ for the completion of $K$ at $v$, and $\cO_{K_v}$ for
its ring of integers. We will write $\cO_{K,(v)}$ for the localisation
of $\cO_K$ at the prime ideal $v$.

We will write $\Zpbar$ for the ring of integers in $\Qpbar$ (a fixed
algebraic closure of $\Qp$), and $\m_\Zpbar$ for the maximal ideal of $\Zpbar$.

We let $\omega$ denote the mod $p$ cyclotomic character. We will
denote a Teichm\"uller lift with a tilde, so that for example
$\omegat$ is the Teichm\"uller lift of $\omega$.

We use the traditional normalisation of Hodge--Tate weights,
with respect to which the cyclotomic character has Hodge--Tate weight~$1$.

By a {\em closed geometric point} $\overline{x}$ of a scheme $X$, we mean
a morphism of schemes $\overline{x}: \Spec \Omega \to X$ for a separably closed
field $\Omega$,
whose image is a closed point $x$ of $X$,
and such that the induced embedding $\kappa(x) \hookrightarrow \Omega$
(where $\kappa(x)$ denotes the residue field of $x$) identifies
$\Omega$ with a separable closure of $\kappa(x)$.
If $\overline{x}$ is a closed geometric point of a Noetherian scheme $X$,
then we let $\cO_{X,\overline{x}}$ denote the local ring of $X$ at $\overline{x}$,
i.e.\ the stalk, in the \'etale topology on $X$, of the structure sheaf
of $X$ at $\overline{x}$,
we let $\widehat{\cO_{X,\overline{x}}}$ denote the completion of $\cO_{X,\overline{x}}$,
and we write $\widehat{X_{\overline{x}}} := \Spf \widehat{\cO_{X,\overline{x}}}$, and refer
to $\widehat{X_{\overline{x}}}$ as the formal completion of $X$ along the closed geometric
point $\overline{x}$.

The symbol $G$ will always denote a group; in Section~\ref{sec:Shimura} it will
be a certain algebraic group, and in Section~\ref{sec:group} it will be a finite group.
\section{$p$-adic Hodge theoretic properties of mod $p$ cohomology}
\label{sec:cohom}

\subsection{Introduction}We now describe in more detail our results on the integral
$p$-adic Hodge theory of the \'etale cohomology of projective
varieties, which are perhaps the most novel part of this paper.

It is well-known that integral $p$-adic Hodge theory is less
robust than the corresponding theory with rational coefficients; for
example, the comparison theorems for integral and mod $p$ \'etale
cohomology due to Fontaine--Messing \cite{FM} and Faltings \cite{Fal}
involve restrictions both on the degrees of cohomology and the
dimensions of the varieties considered, and they also require that the
field $K$ be absolutely unramified and that the variety under
consideration be of good reduction.  More recently, Caruso
\cite{Caruso-thesis} has proved an integral comparison theorem in the
case of semistable reduction for possibly ramified fields $K$, but
there are still restrictions: his result requires that $e i < p - 1$,
where $e$ is the absolute ramification index of $K$, and $i$ is the
degree of cohomology under consideration.

These restrictions are unfortunate, since mod   $p$ and integral $p$-adic Hodge theory 
are among the most powerful local tools available for the analysis of Galois representations
occurring in the mod   $p$ \'etale cohomology of varieties.
The premise that underlies the present work is that frequently in such applications,
one does not need a precise comparison theorem relating the mod   $p$ \'etale cohomology
to an analogous structure involving mod   $p$ de Rham or crystalline cohomology.  Rather,
one often uses the comparison theorem merely to draw much less specific conclusions,
such as that the Galois representations occurring in certain mod   $p$
\'etale cohomology spaces are in the essential image of the Fontaine--Laffaille functor, 
applied to Fontaine--Laffaille modules whose Fontaine--Laffaille numbers lie in
some prescribed range.
Our aim is to establish results of the latter type in more general contexts than
they have previously been proved.

The precise direction of
our work is informed to a significant extent by the fairly recent development of a rich
{\em internal} integral $p$-adic Hodge theory, by Breuil \cite{Br}, Kisin \cite{Ki1}, 
Liu \cite{Liu} and others.  What we mean here by the word ``internal'' is that these
developments have been directed not so much at applications to comparison theorems, but 
rather at the purely Galois-theoretic problem of giving a $p$-adic Hodge-theoretic
description of Galois-invariant lattices in crystalline or semistable Galois
representations, and of the mod   $p$ Galois representations that appear in the
reductions of such lattices.   These tools, especially the theory of {\em Breuil modules}
\cite{Br}, which provides the desired description of the mod $p$ representations
arising as reductions of such lattices, have proved very useful in arithmetic applications.
Because of the availability of these tools, it has become both
possible and worthwhile to move beyond the Fontaine--Laffaille context
in integral $p$-adic Hodge theory.
While Caruso's work mentioned above is a significant step in this direction,
an important aspect of the present work will be the consideration of situations 
in which
the bound $e i < p-1$, required for the validity of the comparison theorem
of \cite{Caruso-thesis}, does not hold.

Our goal, then, is to establish in various situations that a Galois representation
appearing in the mod   $p$ \'etale cohomology of a variety
can be embedded in the reduction of a Galois-invariant lattice contained in a
crystalline or semistable Galois representation, with Hodge--Tate weights lying
in some specified range (namely, the range that one would expect given the degree
of the cohomology space under consideration).
Since, in arithmetic situations, one frequently has to make a ramified base-change 
in order to obtain good or semistable reduction, and since the resulting descent data
on the associated Breuil module typically then play an important role
in whatever analysis has to be undertaken, we also prove 
results in certain cases of potentially semistable reduction which allow
us to gain some control over these descent data.  

The idea underlying our approach is very simple.  Suppose that $X$ is a variety
over a $p$-adic field $K$.  If $i$ is some degree of cohomology, then we have a
short exact sequence
$$0 \to H^i_{\et}(X_{\overline{K}},\mathbb Z_p)/p H^i_{\et}(X_{\overline{K}},\mathbb Z_p)
\to H^i_{\et}(X_{\overline{K}},\mathbb F_p) \to H^{i+1}_{\et}(X_{\overline{K}},\mathbb Z_p)
[p] \to 0,$$
as well as an isomorphism
$$\mathbb Q_p\otimes_{\mathbb Z_p} H^i_{\et}(X_{\overline{K}},\mathbb Z_p)
\iso H^i_{\et}(X_{\overline{K}},\mathbb Q_p).$$
Thus, if both $H^i_{\et}(X_{\overline{K}},\mathbb Z_p)$ and $H^{i+1}_{\et}(X_{\overline{K}},\mathbb Z_p)$
are torsion-free,
then we see that $H^i_{\et}(X_{\overline{K}},\mathbb F_p)$ is the reduction
mod   $p$ of $H^i_{\et}(X_{\overline{K}},\mathbb Z_p)$, which 
{\em is} a Galois-invariant lattice in
$H^i_{\et}(X_{\overline{K}},\mathbb Q_p)$.
Furthermore, the usual comparison theorems
of {\em rational} $p$-adic Hodge theory \cite{Fal, Tsu} can be applied to conclude that
this latter representation
is e.g.\ crystalline (if $X$ is proper with good reduction)
or semistable (if $X$ is proper with semistable reduction).

The obstruction to implementing this idea is that we have no reason to believe
in general that
$H^{i}_{\et}(X_{\overline{K}},\mathbb Z_p)$
and
$H^{i+1}_{\et}(X_{\overline{K}},\mathbb Z_p)$
will be torsion-free.
To get around this difficulty, we engage in various d\'evissages using the weak Lefschetz 
theorem.   To explain these, first consider the case when $X$ is a
projective curve and $i = 1$.  In this case all the cohomology with $\mathbb Z_p$-coefficients
is certainly torsion-free,
and so $H^1_{\et}(X_{\overline{K}},\mathbb F_p)$ {\em is} the reduction of a
Galois-invariant lattice in $H^1_{\et}(X_{\overline{K}},\mathbb Q_p)$. 
Now a simple induction using the weak Lefschetz theorem shows that
for {\em any} smooth projective variety $X$ over $K$, there is an embedding
$$H^1_{\et}(X_{\overline{K}},\mathbb F_p) \hookrightarrow
H^1_{\et}(C_{\overline{K}},\mathbb F_p),$$
where $C$ is a smooth projective curve.  Furthermore, if $X$ has good
(respectively  semistable)
reduction, we can ensure that the same is true of $C$.  This gives the desired
result in the case of $H^1$ (ignoring for a moment the problem of obtaining a refinement
dealing with descent data in the potentially semistable case).

For higher degrees of cohomology, a more elaborate d\'evissage is required.  
The key point, again established via the weak Lefschetz theorem,
is that if $X$ is smooth and projective of dimension $d$, and if $Y$ and $Z$
are sufficiently generic hyperplane sections of $X$, then the cohomology
of the pair
$\bigl((X\setminus Y)_{\overline{K}},(Z\setminus Y)_{\overline{K}}\bigr)$, with either
$\mathbb Z_p$ or $\mathbb F_p$ coefficients,
vanishes in degrees other than $d$ (see Subsection~\ref{subsec:vanishing outside middle}
of the appendix; note that $X\setminus Y$ is affine), so that
$H^d_{\et}\bigl((X\setminus Y)_{\overline{K}},(Z\setminus Y)_{\overline{K}}, \mathbb Z_p)$
is torsion-free and is thus a Galois-invariant lattice in
$H^d_{\et}\bigl((X\setminus Y)_{\overline{K}},(Z\setminus
Y)_{\overline{K}}, \mathbb Q_p)$, which is potentially semistable by
\cite{Go20111127}, and
whose reduction is equal to 
$H^d_{\et}\bigl((X\setminus Y)_{\overline{K}},(Z\setminus Y)_{\overline{K}}, \mathbb F_p)$.
Such relative cohomology spaces are the essential ingredient of the {\em basic lemma} of 
Be{\v\i}linson \cite{Bei}, and we learned the
idea of using them as building blocks for the cohomology of varieties from Nori \cite{Nori},
who has used the basic lemma as the foundation of his approach to the construction
of motives.  Indeed, our present approach to integral $p$-adic Hodge theory
was inspired by Be{\v\i}linson's and Nori's work.

\smallskip

\subsection{Bertini-type theorems} We begin by giving a straightforward
generalisation of some of the results of \cite{JS}, which
build on the results of \cite{Poonen} to prove Bertini-type theorems for
varieties with semistable reduction over a discrete valuation ring. It
will be convenient to allow $K$ to denote either a number field, or a
field of characteristic zero that is complete with respect to a
discrete valuation with perfect residue field $k_K$ of
characteristic~$p$.  We abbreviate these two situations as ``the
global case'' and ``the local case'' respectively, and in the former
case we will let $v$ denote a place of $K$ dividing
$p$.

We recall the following definition:

\begin{defn}
{\em
Suppose firstly that we are in the local case.
We then say that a projective $\cO_K$-scheme $\cX$ is {\em semistable}
if it is regular and flat over $\Spec\cO_K$, and if the special fibre $\cX_s$ is
reduced and is a normal crossings divisor; 
equivalently,
a finite-type $\cO_K$-scheme $\cX$ is semistable 
if at each closed geometric point $\overline{x}$ of~$\cX_s$,
there is an isomorphism of 
complete local rings
\begin{equation*}
\widehat{\cO_{\overline{x},\cX}} \cong
\cshO{K}[[x_1,\dots,x_n]]/(x_1\cdots x_m - \unif_K),
\end{equation*}
where $\cshO{K}$ is the completion of the strict Henselisation of $\cO_K$ (equivalently,
the completion of the ring of integers in the maximal unramified algebraic extension of~$K$),
the element $\unif_K$ is a uniformiser of $\cshO{K}$,
and $1 \leq m \leq n$. We say that $\cX$ is strictly semistable if it
is semistable, and if the special fibre $\cX_s$ is a strict normal
crossings divisor.

Again in the local case, we say that a smooth projective $K$-scheme has
{\em good reduction}
if it admits a smooth projective model over $\mathcal O_K$, and that
it has {\em {\em (}strictly{\em )} semistable
reduction} if it admits an extension to a projective $\mathcal O_K$-scheme
which is (strictly) semistable in the sense of the preceding definition.

In the global case, we say that a smooth projective $K$-scheme has
good reduction at $v$ if it admits a smooth projective model over $\cO_{K,(v)}$,
and that it has (strictly) semistable reduction
at $v$ if it admits a (strictly) semistable projective model over $\cO_{K,(v)}$, i.e.\ a projective
model over $\cO_{K,(v)}$ whose base-change over $\CO_{K_v}$ is
(strictly) semistable 
in the sense of the preceding definition. }
\end{defn}

\begin{remark}
\label{rem:semistable}
{\em Note that our definition of a semistable $\cO_K$-scheme
(putting ourselves in the local case) includes the requirement
that the scheme be regular.  This is the definition that is frequently
adopted in the theory of semistable reduction,
and it is well-suited to our intended applications.
Recall that, with this definition, semistability is {\em not} preserved
under the base-change to $\cO_L$, if $L$ is a finite extension of $K$,
unless $L/K$ is unramified or the original scheme is in fact smooth
over $\cO_K$; see also Remark~\ref{rem:rigid} below.
}
\end{remark}

\begin{prop}\label{prop: Bertini with two hypersurfaces}
Let $X$ be a smooth projective variety
over $K$ with strictly
semistable {\em (}respectively good{\em )} reduction {\em (}at $v$, in the global case{\em )}.
Then there are smooth hypersurface sections $Y$ and $Z$ of $X$
{\em (}with respect to an appropriately
chosen embedding of $X$ into some projective space{\em )}
such that $Y$ and $Z$
intersect transversely, and all of $Y$, $Z$, $Y\cap Z$ have strictly
semistable {\em (}respectively good{\em )} reduction {\em (}at $v$, in the global case{\em )}.

\end{prop}
\begin{proof}We firstly handle the local case. Choose an extension $\cX$
of $X$ to an $\cO_K$-scheme that is projective and smooth (in the good
reduction case) or strictly semistable (in the strictly
semistable reduction case), and fix
an embedding of $\cX$ into some projective space over $\cO_K$.  By Corollary 0 and
  Corollary 1 of \cite{JS} (or, perhaps more precisely, by their proofs)
  we can find a hypersurface section
  $\cY$ of $\cX$ such that $\cY$ is again either smooth or strictly semistable 
  over $\cO_K$. We take $Y$ to be the generic fibre of $\cY$.
  By Remark~0~(ii) of~\cite{JS}, together with Lemma~1 and
  the remark immediately before Corollary~1 of {\em op.\ cit.}, we see
  that in order to find $Z$ it is enough to check that given a finite
  collection $X_1,\dots,X_n$ of smooth projective schemes in
  $\mathbb{P}^N_{/k_K}$, there is a common hypersurface section
  meeting each of them transversely. This is an immediate consequence
  of Theorem 1.3 of \cite{Poonen}, taking the set $U_P$ there to be
  the subset of $\hat{\cO}_P$ consisting of the $f$ such that $f=0$ is transverse to
  each $X_i$ at $P$. (Since this set contains all the $f$ which do not
  vanish at~$P$, and in particular contains all the $f$ congruent modulo the
  maximal ideal to
  particular choice of $f$, it has positive Haar measure.)

  We now pass to the global case. Let $\cX$ be a smooth (in the good reduction case) or strictly
  semistable (in the strictly semistable
  reduction case) projective model of $X$ over $\cO_{K,(v)}$.  
  Let $P^*_d$ denote the projective space (over $\cO_{K,(v)}$) of degree $d$ hypersurfaces in
  the ambient projective space containing $\cX$.
  Applying the argument in the local case to the base-change
  $\cX_{/\cO_{K_v}}$, we see that for some $d \geq 1$, 
  there is a $K_v$-valued point of $P^*_d$ corresponding to a hypersurface section
  of $\cX_{/\cO_{K_v}}$ having either smooth or semistable intersection (depending
  on the case we are in) with $\cX_{/\cO_{K_v}}$. 
  Furthermore, this point lies in an
  affinoid open subset of $P^*_{d /K_v}$ (the preimage of an open set in
  the special fibre of $P^*_d$), all of  whose points correspond to hyperplane sections of
  $\cX_{/\cO_{K_v}}$ with either smooth or strictly semistable intersection.
  (See Remark~0~(i)
  of~\cite{JS}, as well as the proofs of Theorems~0 and~1 of the same reference.)
The set of $K_v$-points of this affinoid open is a non-empty open subset of~$P^*_{d}(K_v)$.  Since $K$ is dense in $K_v$, we
  see that this intersection also contains a $K$-point of~$P^*_d$, which gives the required hypersurface
  section~$Y$. We find the hypersurface section $Z$ by applying the same argument.
\end{proof}

\subsection{Cohomology in degree $1$}
Our arguments in degree $1$ are rather simpler than in general degree,
so we warm up with this case. (In fact, our result in this case is
slightly stronger than our result in general degree, as we do not need
to semisimplify the representation, so this result is not completely
subsumed by our later results in general degree.) Fix a prime $p$.  Let $K$ denote a field
of characteristic zero, complete with respect to a discrete valuation,
with ring of integers $\mathcal O_K$ and residue field~$k$, assumed to
be perfect of characteristic $p$.  Let $\overline{K}$ denote an
algebraic closure of $K$, and set $G_K := \Gal(\overline{K}/K)$.

\begin{theorem}
\label{thm:first cohom}
If $X$ is a smooth projective variety over $K$ which
has good {\em (}respectively  strictly semistable{\em )} reduction,
then $H^1_{\et}(X_{\overline{K}},\mathbb F_p)$ embeds $G_K$-equivariantly
into the reduction
modulo $p$ of a $G_K$-invariant lattice in a crystalline {\em (}respectively 
semistable{\em )} $p$-adic representation
of $G_K$ whose Hodge--Tate weights are contained in $[-1,0]$.
\end{theorem}
\begin{proof}
We proceed by induction on the dimension $d$ of $X$.
If $d \leq 1$ then $H^1_{\et}(X_{\overline{K}},\mathbb F_p)$
is isomorphic to the reduction modulo $p$
of
$H^1_{\et}(X_{\overline{K}},\mathbb Z_p)$,
and the latter space (being torsion-free, by virtue of our assumption
on $d$) is in turn a lattice in
$H^1_{\et}(X_{\overline{K}},\mathbb Q_p)$,
which is crystalline (respectively  semistable) with Hodge--Tate weights lying in $[-1,0],$
by the main result of \cite{Tsu}.

Suppose now that $d > 1$.  It follows from Corollary~0 (respectively  Corollary 1)
of \cite{JS} that if $X$ has good reduction (respectively  strictly semistable reduction)
then we may choose a smooth hypersurface section $Y$ of $X$ defined over $K$
which has good (respectively strictly semistable) reduction.
Our induction hypothesis applies to show that
$H^1_{\et}(Y_{\overline{K}},\mathbb F_p)$
embeds as a subobject
of a $G_K$-representation over $\mathbb F_p$ which is the reduction
modulo $p$ of a $G_K$-invariant lattice in a crystalline (respectively  semistable)
$p$-adic representation
of $G_K$ whose Hodge--Tate weights are contained in $[-1,0]$.
On the other hand, the weak Lefschetz Theorem with $\mathbb
F_p$-coefficients
(\cite[Exp.\ XIV Cor.\ 3.3]{MR0354654}) implies that the natural (restriction) map
$H^1_{\et}(X_{\overline{K}},\mathbb F_p) \rightarrow
H^1_{\et}(Y_{\overline{K}},\mathbb F_p)$
is an embedding (because $1 \leq d-1$ by assumption).
This completes the proof.
\end{proof}

\subsection{Cohomology of arbitrary degree}
As always we fix a prime $p$.  The necessary d\'evissages in this
subsection will be more elaborate than in the previous one, and so, to
maximise the utility of our results for later applications, it will be
convenient to again allow $K$ to denote either a number field or a
field of characteristic zero that is complete with respect to a
discrete valuation with perfect residue field of characteristic~$p$.
In applications it will also be useful to have flexibility in the
choice of coefficients in the various cohomology spaces that we
consider, and to this end we fix an algebraic extension $E$ of $\mathbb
Q_p$, with ring of integers $\mathcal O_E$ and residue field
$k_E$. (In applications, $E$ will typically either be a finite
extension of $\Qp$, or else will be $\Qpbar$.) 

We now recall some consequences of the weak Lefschetz Theorem.  Among
other notions, we will use the \'etale cohomology of a pair consisting
of a variety and a closed subvariety; a precise definition of this
cohomology, and a verification of its basic properties (such as those
recalled in the next paragraph), is included in Appendix \ref{appendix: cohomology results on
  pairs etc}.

Let $X$ be a smooth projective variety of dimension $d$ over $K$, and
suppose that $Y$ and $Z$ are two smooth hypersurface sections of $X$,
chosen so that $Y\cap Z$ is also smooth.  Let $A$ denote either $E$,
$\mathcal O_E$, or $k_E$.  In either the first or last case, the
spaces $H^i_{\et}\bigl((X\setminus Y)_{\overline{K}}, (Z\setminus
Y)_{\overline{K}} , A\bigr)$ and $H^{2 d - i}_{\et}\bigl((X\setminus
Z)_{\overline{K}}, (Y\setminus Z)_{\overline{K}} , A\bigr)(d)$ are
naturally dual to one another, for each integer $i$. The weak
Lefschetz Theorem implies that the former space vanishes when $i > d$,
and that the latter space vanishes when $2 d - i > d$, i.e.\ when $i <
d$. Thus in fact both spaces vanish unless $i = d$.  It then follows
that both spaces vanish unless $i = d$ in the case when $A$ is taken
to be $\mathcal O_E$ as well, and hence that, when $i = d$, both
spaces are torsion-free.

Let $\overline{K}$ denote an
algebraic closure of $K$, and set $G_K := \Gal(\overline{K}/K)$.
Now let $\rho:G_K \to \GL_n(k_E)$ be irreducible and continuous.
In the global case, we fix a place $v$ of $K$ lying over $p$, 
and a decomposition group $D_v \subset G_K$ for $v$.

\begin{theorem} 
\label{thm:arb cohom}
If $X$ is a smooth projective variety over $K$ which
has strictly semistable {\em (}respectively good{\em )} reduction {\em (}at $v$,
if we are in the global case{\em )},
and if $\rho$ embeds as a subquotient of
$H^i_{\et}(X_{\overline{K}},k_E)$, 
then $\rho$
also embeds as a
subquotient of a $G_K$-representation
over $k_E$
which is the reduction
modulo the uniformiser of a $G_K$-invariant
$\mathcal O_E$-lattice in a $G_K$-representation
which is semistable {\em (}respectively crystalline{\em )} {\em (}at $v$, in the global case{\em )}
with Hodge--Tate weights contained in the interval $[-i,0]$.
\end{theorem}
\begin{proof}
We proceed by induction on the dimension of $X$. Suppose initially
that we are in the strictly semistable reduction case.
By Proposition~\ref{prop: Bertini with two hypersurfaces}
we can and do choose smooth hypersurface sections $Y$ and $Z$, having
smooth intersection, and such
that $Y$, $Z$, and $Y\cap Z$ all have strictly semistable reduction.

We then consider the long exact sequences (cf.~\cite[Chap.\ III Prop.\ 1.25]{MR559531} for the first two, which are local cohomology long exact sequences,
 and Appendix \ref{appendix: cohomology results on
  pairs etc} for the third, which is the long exact sequence of the pair 
$(X\setminus Y, Z\setminus Y)$)
$$\cdots \to H^i_{Y_{\overline{K}},\et}(X_{\overline{K}}, A) \to
H^i_{\et}(X_{\overline{K}}, A) \to H^i_{\et}\bigl((X\setminus Y)_{\overline{K}}, A\bigr) \to \cdots,$$
$$\cdots \to H^i_{(Y\cap Z)_{\overline{K}},\et}(Z_{\overline{K}}, A) \to
H^i_{\et}(Z_{\overline{K}}, A) \to H^i_{\et}\bigl((Z\setminus Y)_{\overline{K}}, A\bigr) \to \cdots,$$
and
$$\cdots \to H^i_{\et}\bigl((X\setminus Y)_{\overline{K}}, (Z\setminus Y)_{\overline{K}}, A\bigr)
\to H^i_{\et}\bigl((X\setminus Y)_{\overline{K}}, A\bigr) \to H^i_{\et}\bigl((Z\setminus Y)_{\overline{K}},
A\bigr)
\to \cdots,$$
with $A$ taken to be either $E$ or $k_E$.
We also recall (cf.\ \cite[Exp.\ XIV \S 3]{MR0354654}) that there are canonical isomorphisms
$H^{i-2}_{\et}(Y_{\overline{K}},A)(-1)
\iso H^i_{Y_{\overline{K}},\et}(X_{\overline{K}},A)$
and
$H^{i-2}_{\et}\bigl((Y\cap Z)_{\overline{K}},A\bigr)(-1) \iso
H^i_{(Y\cap Z)_{\overline{K}},\et}(Z_{\overline{K}},A)$.

When $A = E$, all the cohomology spaces that appear are potentially
semistable \cite{Go20111127}. Since $H^i_{\et}(X_{\overline{K}},E)$,
$H^i_{\et}(Y_{\overline{K}},E)$, and $H^i_{\et}(Z_{\overline{K}},E)$
are semistable with Hodge--Tate weights lying in $[-i,0]$, we see
that $H^i_{\et}\bigl((X\setminus Y)_{\overline{K}},E\bigr)$,
$H^i_{\et}\bigl((Z\setminus Y)_{\overline{K}},E\bigr)$, and
$H^i_{\et}\bigl((X\setminus Y)_{\overline{K}},(Z\setminus Y)_{\overline{K}},E\bigr)$
are semistable, with Hodge--Tate weights lying in $[-i,0]$.

Now taking $A = k_E$, we see that since $\rho$ is irreducible, it
appears either as a subquotient of
$H^{i-2}_{\et}(Y_{\overline{K}},k_E)(-1),$ as a subquotient of
$H^i_{\et}(Z_{\overline{K}},k_E)$, as a subquotient of $H^{i-1}(Y\cap
Z)(-1)$, or else as a subquotient of $H^i_{\et}\bigl((X\setminus
Y)_{\overline{K}},(Z\setminus Y)_{\overline{K}},k_E\bigr).$ In the first three cases
the theorem follows by induction on the dimension.  In the final case,
the conclusion follows from the vanishing theorem noted above; namely,
$H^i_{\et}\bigl((X\setminus Y)_{\overline{K}},(Z\setminus Y)_{\overline{K}},E\bigr)$
is the desired semistable representation, with invariant lattice
$H^i_{\et}\bigl((X\setminus Y)_{\overline{K}},(Z\setminus Y)_{\overline{K}},\mathcal
O_E\bigr)$, whose reduction $H^i_{\et}\bigl((X\setminus
Y)_{\overline{K}},(Z\setminus Y)_{\overline{K}},k_E\bigr)$ contains $\rho$.

Finally, suppose that we are in the good reduction case. Again, by
Proposition~\ref{prop: Bertini with two hypersurfaces}, we can and do
choose smooth hypersurface sections $Y$ and $Z$, having smooth
intersection, and such that $Y$, $Z$, and $Y\cap Z$ all have
good reduction. Applying the same argument as in the previous paragraph, we see
by induction on the dimension of $X$ that it is
enough to check that $H^i_{\et}\bigl((X\setminus
Y)_{\overline{K}},(Z\setminus Y)_{\overline{K}},E\bigr)$ is crystalline;
but this follows immediately from Theorem 1.2 of~\cite{Go20111127}. (Note that
if in the notation of~\cite{Go20111127} we take $D^1=Z$, and $D^2=Y$, then by
(\ref{eqn:Faltings}) below we see that $H^i_{\et}\bigl((X\setminus
Y)_{\overline{K}},(Z\setminus Y)_{\overline{K}},E\bigr)$ appears on the left hand side of the Hyodo--Kato
isomorphism in the statement of Theorem 1.2 of~\cite{Go20111127}, and since we
have already shown that $H^i_{\et}\bigl((X\setminus
Y)_{\overline{K}},(Z\setminus Y)_{\overline{K}},E\bigr)$ is semistable, it is enough to show
that the monodromy operator $N$ vanishes on the right hand side of the Hyodo--Kato
isomorphism. This follows easily from the definition of this operator as a
boundary map, as all objects concerned arise from base change from objects with
trivial log structures.)
\end{proof}

\subsection{Equivariant versions}
\label{subsec:equivariant}
In practice, we will need equivariant analogues of the preceding results.
As in the preceding section, we let $K$ denote either a number field (``the global
case'') or a field
of characteristic zero that is complete with respect to a discrete valuation with
perfect residue field of characteristic $p$ (``the local case'').
We let $\overline{K}$ denote an
algebraic closure of $K$, and set $G_K := \Gal(\overline{K}/K)$.
In the global case, we fix a place $v$ of $K$ lying over $p$, 
and a decomposition group $D_v \subset G_K$ for~$v$.

We now put ourselves in the following (somewhat elaborate) situation,
which we call a {\em tamely ramified semistable context}, or a
{\em tame semistable context} for short.

We suppose that $X_0$ and $X_1$ are 
smooth projective varieties over $K$, that $G$ is a finite group
which acts on $X_1$, and that $\pi: X_1 \to X_0$ is a finite \'etale 
morphism which intertwines the given $G$-action on  $X_1$
with the trivial $G$-action on $X_0$, making $X_1$ an \'etale $G$-torsor 
over $X_0$.   

We suppose further that $X_0$ admits a semistable projective model $\cX_0$ over
$\mathcal O_K$ (in the local case) or 
over  $\mathcal O_{K,(v)}$
(in the global case).  We also suppose that there is a finite Galois extension $L$ of~$K$,
and (in the global case) a prime $w$ of $L$ lying over $v$, such
that $(X_1)_{/L}$ admits a semistable projective model $\mathcal X_1$ over
$\mathcal O_L$ (in the local case) or over $\mathcal O_{L,(w)}$
(in the global case) to which the $G$-action extends, such that $\pi$ extends
to a morphism  $\mathcal X_1
\to (\mathcal X_0)_{/\mathcal O_L}$ which intertwines the $G$-action
on its source with the trivial $G$-action on its target, and such that
the action of the (opposite group of) the inertia group $I(L/K)^{\op}$ (respectively  $I(L_w/K_v)^{\op}$ in the global case) on
$(X_1)_{/L}$ extends to an action on $\mathcal X_1$.\footnote{Note that the tameness condition
that we are going to require below ensures that $L/K$ is in fact tamely ramified,
and hence that $I(L/K)$ is abelian.  Thus passing to the opposite group is not
actually necessary here when passing from the action on rings
to the action on their $\Spec$s, but we will keep the superscript $\op$ in the notation 
for the sake of conceptual clarity.}

Finally (and most importantly) we assume that the composite morphism
\numequation
\label{eqn:tamely ramified composite}
\mathcal X_1 \to (\mathcal X_0)_{/\mathcal O_L} \to \mathcal X_0
\end{equation}
(the first being the extension of $\pi$, and the second being 
the natural map) is tamely ramified along the special fibre
$(\cX_0)_s$, in the sense of \cite[Def.~2.2.2]{Groth-Murre}.

In fact, in our applications we will consider the case that $\cX_0$ is
furthermore strictly semistable, in which case we will say that we are
in a \emph{tame strictly semistable context}.

\begin{remark}
\label{rem:rigid}
{\em The notion of a tame semistable context is somewhat rigid, 
as we will see in the following lemma,
and would perhaps not be of much interest if it did not occur naturally
in the Shimura variety context (as we will see
Subsection~\ref{subsec:semistable reduction}).
As one example of this rigidity, note that if $G = 1$, i.e.\ if $X_0$ and $X_1$
coincide,
then the only way to achieve a tame semistable context is if $L/K$ is
unramified,
or if $\cX_0$ is smooth over $\cO_K$.
(Indeed, since a tamely ramified morphism is finite, 
and since the base-change $\cX_{0/\cO_L}$ over the semistable $\cO_K$-scheme
$\cX_0$ is normal, we see that if $X_0$ and $X_1$ coincide then the
morphism $\cX_1 \to \cX_{0/\cO_L}$ is necessarily an isomorphism.
This implies that the semistable
$\cO_K$-scheme $\cX_0$ has a semi-stable base-change over $\cO_L$,
which, as we noted in Remark~\ref{rem:semistable},
is possible only if $L/K$ is unramified or $\cX_0$ is smooth
over $\cO_K$.  Another point of view on this case is as follows:
if $\cX_0$
is semistable but not smooth, then in order to construct a semistable
model $\cX_1$ of  $\cX_{0/\cO_L}$,
we must perform some non-trivial blow-ups, and the resulting morphism
$\cX_1 \to \cX_0$ is not finite, and in particular not tamely ramified.)
}
\end{remark}

The following lemma gives a more concrete interpretation of the 
stipulation that~(\ref{eqn:tamely ramified composite})  be tamely ramified 
along $(\cX_0)_s$.

\begin{lemma}
\label{lem:tame}
In the above setting,
the morphism~{\em (\ref{eqn:tamely ramified composite})} is tamely
ramified along $(\cX_0)_s$
if and only if the following two conditions hold:
\begin{enumerate}
\item
$L$ {\em (}respectively
$L_w$ in the global case{\em )} is tamely ramified over $K$ {\em (}respectively $K_v$
in the global case{\em )},  of ramification degree $e$, say;
\item
For each closed geometric point $\overline{x}_1$ of the special fibre
$(\mathcal X_1)_s$, with image $\overline{x}_0$ in $(\mathcal X_0)_s$,
and for some choice of isomorphism
\numequation
\label{eqn:semistable model}
\widehat{\cO_{\overline{x}_0,\cX_0}} \cong \cshO{K}[[x_1,\ldots,x_n]]/(x_1\cdots x_m
- \unif_K),
\end{equation}
where $\unif_K$ is a uniformiser of $\cshO{K}$ and $1 \leq m \leq n$,
there is a corresponding isomorphism
$$
\widehat{\cO_{\overline{x}_1,\cX_1}} \cong \cshO{L}[[y_1,\ldots,y_n]]/(y_1\cdots y_m
- \unif_L),$$
where $\unif_L$ is a uniformiser of $\cshO{L}$,
such that the induced morphism
$$\widehat{(\cX_1)_{\overline{x}_1}} \to
\widehat{(\cX_0)_{\overline{x}_0}}$$
is defined by the formula $x_j = y_j^e$ for $1 \leq j \leq m,$
and $x_j = y_j$ for $m < j \leq n$.
\end{enumerate}
Furthermore, if these equivalent conditions hold, then condition~{\em (2)} holds
for {\em every} choice of isomorphism~{\em (\ref{eqn:semistable model})}.
\end{lemma}
\begin{proof}
We first note that if we are in the global case, than relabelling
$K_v$ as $K$ and $L_w$ as $L$, we may reduce ourselves to 
proving the lemma in the local case.  Thus we assume that we in the local case
from now on.

If conditions~(1) and (2) (for some choice of isomorphism~(\ref{eqn:semistable model}))
hold, then the morphism~(\ref{eqn:tamely ramified composite}) is certainly
tamely ramified along $(\cX_0)_s$.  (This amounts to the claim
that we can verify tame ramification by passing to formal completions of closed
geometric points, which is indeed the case, as follows from \cite[Cor.~4.1.5]{Groth-Murre}.)

Suppose conversely that~(\ref{eqn:tamely ramified composite}) 
is tamely ramified along $(\cX_0)_s$.
Since this morphism factors through the natural morphism $(\cX_0)_{/\cO_L} \to \cX_0,$
it follows from \cite[Lem.~2.2.5]{Groth-Murre} that this latter morphism
is tamely ramified, and hence (e.g.~by
\cite[Prop.~2.2.9]{Groth-Murre}, although our particular situation
is much simpler than the general case of faithfully flat descent for tamely ramified
covers considered in that proposition) that $\Spec \cO_L \to
\Spec \cO_K$ is tamely ramified, i.e.\ that $L$ is tamely ramified over $K$,
of some ramification degree~$e$.  Thus condition~(1) holds.

Now choose a closed geometric point $\overline{x}_1$ of $(\cX_1)_s$ lying
over the closed geometric point $\overline{x}_0$ of $(\cX_0)_s$,
and fix an isomorphism of the form~(\ref{eqn:semistable model}).
Since $\cX_1 \to \cX_0$ is tamely ramified along the divisor $\unif_K = 0$
of $\cX_0$,
Abhyankar's Lemma \cite[Exp.\ XIII Cor.\ 5.6]{MR0354651}  (see also
\cite[Thm.~2.3.2]{Groth-Murre} for a concise statement)
implies that we may find regular elements $\{a_j\}_{j = 1,\ldots,k}$ of
$\cshO{K}[[x_1,\ldots,x_n]]/(x_1\cdots x_m - \unif_K)$
so that $a_1\cdots a_k$ generates the ideal $(\unif_K)$ of
$\cshO{K}[[x_1,\ldots,x_n]]/(x_1\cdots x_m - \unif_K)$,
exponents $e_1,\ldots,e_k$ all coprime to $p$,
and a subgroup $H \subset \bbmu_{e_1} \times \cdots \times \bbmu_{e_k},$
such that the
$\cshO{K}[[x_1,\ldots,x_n]]/(x_1\cdots x_m - \unif_K)$-algebra
$\widehat{\cO_{\overline{x}_1,\cX_1}}$ is isomorphic to
$$\bigl(\cshO{K}[[x_1,\ldots,x_n]][T_1,\ldots,T_k]/(x_1\cdots x_m - \unif_K,
T_1^{e_1}-a_1,\ldots, T_k^{e_k}-a_k)\bigr)^H.$$
(Here $\bbmu_{e_1} \times \cdots \times \bbmu_{e_k},$ and hence $H$, acts on
$\cshO{K}[[x_1,\ldots,x_n]][T_1,\ldots,T_k]/(x_1\cdots x_m - \unif_K,
T_1^{e_1}-a_1,\ldots, T_k^{e_k}-a_k)$
in the obvious
manner: namely, an element $(\zeta_1,\ldots,\zeta_k)$ acts on $T_j$
via multiplication by $\zeta_j$.)

Since each $e_j$ is prime to $p$ and $\cshO{K}[[x_1,\ldots,x_n]]/(x_1\cdots x_m -\unif_K)$
is strictly Henselian, any unit in this ring has an $e_j$th root,
and thus we are free to multiply any of the $a_j$ by a unit.  Consequently,
we may assume that in fact $a_1\cdots a_k = \unif_K = x_1\cdots x_m,$
and hence (again taking advantage of our freedom to modify the $a_j$ by units)
that $\{1,\ldots,m\}$ is partitioned into $k$ sets $J_1,\ldots, J_k$,
such that $a_j = \prod_{i \in J_j} x_i$.  Now if we extract the $e_j$th roots
of each $x_i$ for $i \in J_j$, the resulting extension of
$\cshO{K}[[x_1,\ldots,x_n]]/(x_1\cdots x_m - \unif_K)$
contains
$$\cshO{K}[[x_1,\ldots,x_n]][T_1,\ldots,T_k]/(x_1\cdots x_m - \unif_K,
T_1^{e_1}-a_1,\ldots, T_k^{e_k}-a_k);$$
thus it is no loss of generality
to assume that $k = m$ and that $a_j = x_j$,
and so we conclude that
$\widehat{\cO_{\overline{x}_1,\cX_1}}$ is isomorphic,
as an
$\cshO{K}[[x_1,\ldots,x_n]]/(x_1\cdots x_m - \unif_K)$-algebra,
to
$$\bigl(\cshO{K}[[x_1,\ldots,x_n]][T_1,\ldots,T_m]/(x_1\cdots x_m - \unif_K,
T_1^{e_1}-x_1,\ldots, T_m^{e_m}-x_m)\bigr)^H,$$
for some subgroup $H \subset \bbmu_{e_1} \times \cdots \times \bbmu_{e_m}.$

Let $I_j$ denote the subgroup
$1 \times \cdots \times \bbmu_{e_j} \times \cdots \times 1$
of $\bbmu_{e_1} \times \cdots \times \bbmu_{e_m}$; this is the inertia group
of the divisor $(x_j)$ with respect to the cover
\begin{multline*}
\Spec
\cshO{K}[[x_1,\ldots,x_n]][T_1,\ldots,T_m]/(x_1\cdots x_m - \unif_K,
T_1^{e_1}-x_1,\ldots, T_m^{e_m}-x_m)
\\
\to
\Spec
\cshO{K}[[x_1,\ldots,x_n]]/(x_1\cdots x_m - \unif_K).
\end{multline*}

If we write $H_j = H\cap I_j$, then $H' := H_1\times \cdots\times H_m$ is a 
subgroup of $H$, and the cover
\begin{multline*}
\Spec \bigl(\cshO{K}[[x_1,\ldots,x_n]][T_1,\ldots,T_m]/(x_1\cdots x_m - \unif_K,
T_1^{e_1}-x_1,\ldots, T_m^{e_m}-x_m)\bigr)^{H'}
\\
\to 
\Spec
\bigl(\cshO{K}[[x_1,\ldots,x_n]][T_1,\ldots,T_m]/(x_1\cdots x_m - \unif_K,
T_1^{e_1}-x_1,\ldots, T_m^{e_m}-x_m)\bigr)^{H}
\end{multline*}
is unramified in codimension one.  Since $\cX_1$ is regular, being 
semistable over $\cO_L$, so is the target of this map (since we recall that
this target is isomorphic to $\widehat{(\cX_1)_{\overline{x}_1}}$).
The purity of the branch locus then implies that this cover is \'etale,
and hence is an isomorphism (since its target is strictly Henselian).
Consequently $H = H'$. 

If we write
$$
H_j = 
1 \times \cdots \times \bbmu_{e_j'} \times \cdots \times 1
\subset
1 \times \cdots \times \bbmu_{e_j} \times \cdots \times 1
= I_j,
$$
and set $d_j = e_j/e_j'$ and $S_j = T_j^{e_j'}$,
then we conclude that
\begin{multline*}
\widehat{\cO_{\overline{x}_1,\cX_1}} \\ \,\, \cong 
 \bigl(\cshO{K}[[x_1,\ldots,x_n]][T_1,\ldots,T_m]/(x_1\cdots x_m - \unif_K,
T_1^{e_1}-x_1,\ldots, T_m^{e_m}-x_m)\bigr)^{H'}
\\
\cong
\cshO{K}[[x_1,\ldots,x_n]][S_1,\ldots,S_m]/(x_1\cdots x_m - \unif_K,
S_1^{d_1}-x_1,\ldots, S_m^{d_m}-x_m).
\end{multline*}
Now $\mathcal X_1$ is an $\cO_L$-scheme with reduced special fibre (again because
it is semistable over $\cO_L$).  Since 
$\widehat{\cO_{\overline{x}_1,\cX_1}}$
is strictly Henselian, it contains
$\cshO{L}$, and we may choose a uniformiser $\unif_L$ of this ring such
that $\unif_L^e = \unif_K$.  Looking at the above description of 
$\widehat{\cO_{\overline{x}_1,\cX_1}}$, and taking into account that its reduction
modulo $\unif_L$ must be reduced, we see that this special fibre must
be the zero locus of the element $S_1\cdots S_m,$ hence that $S_1\cdots S_m = u \unif_L$
for some unit $u$,
and thus that $(S_1\cdots S_m)^e = u^e \unif_K$.  We conclude that $d_1 = \ldots = d_m = e$
and that $u^e = 1$, and hence, replacing $\unif_L$ by $u\unif_L$,
we find that $S_1\cdots S_m  = \unif_L$.  This shows that~(2) holds (for our
given choice of isomorphism~(\ref{eqn:semistable model})).
\end{proof} 

\begin{remark}
{\em  Note that we could avoid the appeal to the general theory of tame
ramification (in particular, to Abhyankar's Lemma) by just directly
stipulating in our context that conditions~(1) and~(2) of Lemma~\ref{lem:tame} hold;
indeed, in the proof of Theorem~\ref{thm:arb cohom equivariant} below, we will
work directly with these conditions, and in our applications to Shimura varieties,
we will also see directly that these conditions hold.
Nevertheless, we have included Lemma~\ref{lem:tame} as an assurance to ourselves
(and perhaps to the reader) that these conditions are somewhat natural.
}
\end{remark}

\begin{lem}
  \label{lem:strict tame semistable implies X_1 strict}In a tame
  strictly semistable context as above, the $\cO_L$-scheme $\cX_1$ is
  also strictly semistable.
\end{lem}
\begin{proof}Since $\cX_1$ is semistable by assumption, it is enough
to show that the components of the special fibre $(\cX_1)_s$ are regular.
Suppose that $D$ is a non-regular component of $(\cX_1)_s$, and let $\overline{x}_1$
be a closed geometric point of $D$ whose local ring on $D$ is not regular. 
If we let $\overline{x}_0$ denote the image of $\overline{x}_1$ in $(\cX_0)_s$,
then Lemma~\ref{lem:tame}~(2) shows that we may find isomorphisms
$\widehat{\cO_{\overline{x}_1,(\cX_1)_s}}
\cong
\overline{\cO_K/\varpi_K}[[y_1,\ldots,y_n]]/(y_1\cdots y_m)$
and
$\widehat{\cO_{\overline{x}_0,(\cX_0)_s}}
\cong
\overline{\cO_K/\varpi_K}[[x_1,\ldots,x_n]]/(x_1\cdots x_m),$
with $1 \leq m \leq n,$
such that the morphism
$
\widehat{\cO_{\overline{x}_1,(\cX_1)_s}}
\to
\widehat{\cO_{\overline{x}_0,(\cX_0)_s}}
$
is  given by $x_j = y_j^e$ for $1 \leq j \leq m$ and $x_j = y_j$ for $m < j \leq n$.
Since by assumption $\overline{x}_1$ is not a regular point of $D$,
we find that necessarily $m \geq 2,$ and that (possibly after permuting indices)
there is an isomorphism
$\widehat{\cO_{\overline{x}_1,D}}
\cong
\overline{\cO_K/\varpi_K}[[y_1,\ldots,y_n]]/(y_1\cdots y_{m'}),$
where $2 \leq m' \leq m$.
If we let $D'$ denote the image of $D$ in $(\cX_0)_s$,
we conclude that
there is an isomorphism
$\widehat{\cO_{\overline{x}_0,D'}}
\cong
\overline{\cO_K/\varpi_K}[[x_1,\ldots,x_n]]/(x_1\cdots x_{m'}),$
and thus  that $D'$ is not regular.
Hence $\cX_0$ is not strictly semistable,
a contradiction.\end{proof}

We now suppose that we are in a tame semistable context, as described above, and
suppose for the moment that we are in the local case. Then we have an action of
$I(L/K)^\op\times G$ on the special fibre $(\cX_1)_s$. Let $D$ be an irreducible
component of the special fibre of $(\cX_0)_s$, and let $\tD$ denote its preimage
in $(\cX_1)_s$, so that $\tD$ is an $I(L/K)^{\op}\times G$-invariant union of
irreducible components of $(\cX_1)_s$.

\begin{lem}
\label{lem:inertia computation}
If $G$ is abelian,
then there is a homomorphism $\psi:I(L/K) \to G$ such that the action of $I(L/K)^{\op}$ on
$\tD$ is given by composing the action of $G$ with $\psi$;
i.e.\ if $i \in I(L/K)$, then the action of $i$ on $\tD$ coincides with the action of $\psi(i)$.
\end{lem}

We first prove a general lemma.

\begin{lem}
\label{lem:action on torsors}
Let $S$ be a connected Noetherian scheme, 
let $G$ and $I$ be finite groups, and let $f:T \to S$ be a finite \'etale morphism
with the property that $ I^{\op} \times G$ acts on $T$ over $S$ in such
a way that $T$ becomes a $G$-torsor over $S$.
If $G$ is abelian, then there exists a morphism $\psi: I \to G$
such that the action of $I$ on $T$ is given by composing the action
of $G$ on $T$ with the morphism $\psi$.
\end{lem}
\begin{proof}
  For clarity, we will not impose the assumption that $G$ is abelian until it is
  required.

If we fix a geometric point $\overline{s}$ of $S$,
then the theory of the \'etale fundamental group \cite[Exp.\ V Thm.\ 4.1]{MR0354651}
shows that passing to the fibre over $\overline{s}$ gives an equivalence of
categories between the category of finite \'etale covers of $S$
and the category of (discrete) finite sets with a continuous action of
$\pi_1(S,\overline{s}).$ In this way, $T$ is classified 
by an object $P$ of this latter category equipped with an action of $I^{\op} 
\times G$, with respect to which the $G$-action makes $P$ 
a principal homogeneous $G$-set.  

If we fix a base-point $p \in P$, then we may identify $P$ with $G$, thought of as
a principal homogeneous $G$-set via left multiplication.  As the automorphisms of $G$ as
a principal homogeneous $G$-set are naturally identified with $G^{\op}$ acting
by right multiplication, we obtain a homomorphism $\psi_p: I^{\op} \to G^{\op}$,
or equivalently a homomorphism $\psi_p: I \to G$, describing the
action of $I^{\op}$ on $P$.  If we replace $p$ by $g \cdot p$ (for some
 $g \in G$), then one finds that $\psi_{gp} = g \psi_p g^{-1}.$
Thus, if we now assume furthermore that $G$ is abelian, then $\psi_{p} = \psi_{gp}$,
and so it is reasonable in this case to write simply $\psi$ for this
homomorphism, which is well-defined independently of the choice of base-point
for $P$.   Furthermore, when $G$ is abelian, left and right multiplication
by an element $g \in G$ coincide, and so the action of $I^{\op}$ on
$P$ is given by the formula $i\cdot p = \psi(i)\cdot p$ for all $p \in P$.
Since the automorphisms of $T$ over $S$ induced by $i$ and $\psi(i)$
coincide on $P$, they in fact coincide on all of $T$.  
\end{proof}

\noindent
{\em Proof of Lemma~\ref{lem:inertia computation}.}
The morphism $\cX_1 \to (\cX_0)_{/\cO_L}$ is \'etale on generic fibres,
and the explicit local formulas for this morphism provided by
Lemma~\ref{lem:tame} show that it is in fact \'etale over
an open subset $\cU_0$ of $(\cX_0)_{/\cO_L}$ whose intersection with the
special fibre $\bigl((\cX_0)_{/\cO_L}\bigr)_s$ is
Zariski dense.
Replacing $\cU_0$ with the intersection of all of its $I(L/K)^{\op}$-translates,
we may furthermore assume that $\cU_0$ is invariant under the action
of $I(L/K)^{\op}$ on $(\cX_0)_{/\cO_L}$.

If we let $\cU_1$ denote the preimage of $\cU_0$ in $\cX_1$,
then $\cU_1$ is invariant under the $I(L/K)\times G$-action on $\cX_1$,
and the morphism $\cU_1 \to \cU_0$ is a finite \'etale cover, 
for which the corresponding map $U_1 \to U_0$ on generic fibres
realizes $U_1$ as a $G$-torsor over $U_0$.   It follows that
the $G$-action on $\cU_1$ realizes
$\cU_1$ as an \'etale $G$-torsor over $\cU_0$,
and hence, passing to special fibres, that $(\cU_1)_s$ is an
\'etale $G$-torsor over $(\cU_0)_s$.

Now the induced $I(L/K)^{\op}$-action on $(\cU_0)_s$ is trivial,
and so $I(L/K)^{\op}$ acts on $(\cU_1)_s$ as a group of automorphisms
of the $G$-torsor $(\cU_1)_s$ over $(\cU_0)_s$.
If $D' := D \cap (\cU_0)_s$, then $D'$ is an irreducible component of $(\cU_0)_s$,
and $\tD' := \tD \cap (\cU_1)_s$ is the restriction of $(\cU_1)_s$ to $D'$.
Thus $\tD' \to \tD$ is again an \'etale $G$-torsor, with an
action of $I(L/K)^{\op}$ via automorphisms.
Lemma~\ref{lem:action on torsors} then shows that
there is a homomorphism $\psi: I(L/K)^{\op} \to G^{\op}$, or equivalently a homomorphism $\psi: I(L/K) \to G$, such that
the action of $I(L/K)^{\op}$ on the points of $\tD'$ 
is given by composing the action of $G$ with the homomorphism $\psi$. 
Since $\tD$ is equal to the Zariski closure of $\tD'$ in $(\cX_1)_s$,
the claim of the lemma follows.
\qed

\begin{lemma}
\label{lem:intersection} Suppose that we are in a tame strictly
semistable context.
If $g \in I(L/K) \times G$ and $D$ is a component of 
$(\cX_1)_s$, then $D$ and $g D$ either coincide or are disjoint.
\end{lemma}
\begin{proof}
The images of $D$ and $g D$ in $(\cX_0)_s$ coincide,
and it follows from Lemma~\ref{lem:tame}
that two distinct components of $(\cX_1)_s$ that have non-empty intersection
must have distinct images in $(\cX_0)_s$.
\end{proof}

If we now suppose that we are in the global case, then the discussion applies
with $L/K$ everywhere replaced by $L_w/K_v$, and in particular for
each component $D$ we may 
define a character $\psi: I(L_w/K_v) \to G$ describing the action of $I(L_w/K_v)$
on $D$.

Our next result describes how our tame  semistable context behaves upon passage
to a semistable hypersurface section of $\cX_0$. 
In its statement we assume for simplicity that we are in the local case.  

\begin{prop}
\label{prop:tame hyperplane}Suppose that we are in the tamely ramified semistable
  context described above, and let $\cY_0$ be a regular hypersurface
  section of $\cX_0$ such that the union of $\cY_0$ and $(\cX_0)_s$
  forms a divisor with normal crossings on $\cX_0$. Let $Y_0$ denote
  the generic fibre of $\cY_0$, let $Y_1$ denote
  the preimage of $Y_0$ under the morphism $\pi: X_1 \to X_0,$ and let
  $\cY_1$ be the preimage of $\cY_0$ under~{\em (\ref{eqn:tamely
      ramified composite})} {\em (}so that $(Y_1)_{/L}$ is the generic
  fibre of the $\cO_L$-scheme $\cY_1${\em )}. Then:
\begin{enumerate}
\item The complement of $Y_1$ in $X_1$ is affine.
\item
The generic fibre $Y_0$ of $\cY_0$ is smooth over $K$,
the morphism
$Y_1 \to Y_0$ is an \'etale $G$-torsor {\em (}so in particular $Y_1$ is 
also smooth over~$K${\em )},
$\cY_0$ is a semistable model of $Y_0$ over~$\cO_K$,
$\cY_1$ is a semistable model for $(Y_1)_{/L}$ over~$\cO_L$,
and the morphism $\cY_1 \to \cY_0$ is tamely ramified;  consequently
$\cY_1 \to \cY_0$ is again a tamely ramified semistable context.
\item
Suppose that $G$ is abelian.
If $D'$ is an irreducible component of $(\cY_1)_s$, contained in an irreducible
component $D$ of $(\cX_1)_s$, then the homomorphism $\psi: I(L/K) \to G$, 
which describes the action of $I(L/K)$
on~$D$, also describes the action of $I(L/K)$
on~$D'$.
\end{enumerate}
\end{prop}
\begin{proof}
Since $\cY_0$ is a hypersurface section of $\cX_0$, its generic fibre
$Y_0$ is a hypersurface section of $X_0$.  Thus its complement is affine.
Since $\pi$ is a finite morphism by assumption, the complement of $Y_1$
in $X_1$ is again affine.    Since $Y_0$ is a regular projective
$K$-scheme (being the generic fibre of $\cY_0$, which is regular by assumption),
it is in fact smooth over $K$.  By definition $Y_1$ is the preimage of $Y_0$ under
the morphism $X_1 \to X_0$, which is an \'etale $G$-torsor by assumption.
Thus $Y_1 \to Y_0$ is indeed an \'etale $G$-torsor (and so $Y_1$ is also 
smooth over $K$).

Let $\overline{x}_0$ be a closed geometric point of the special fibre $(\cY_0)_s$.
Since
$(\widehat{(\cX_0)_{\overline{x}_0}})_s \cup \widehat{(\cY_0)_{\overline{x}_0}}$
forms a divisor with normal crossings,
since each component of
$(\widehat{(\cX_0)_{\overline{x}_0}})_s$
is regular, and since $\cY_0$ is regular by assumption,
it follows from \cite[Lem.~1.8.4]{Groth-Murre} that
$(\widehat{(\cX_0)_{\overline{x}_0}})_s \cup \widehat{(\cY_0)_{\overline{x}_0}}$
is in fact a divisor with strictly normal crossings in
$\widehat{(\cX_0)_{\overline{x}_0}}$,
and hence the local equation $\ell$ of $\widehat{(\cY_0)_{\overline{x}_0}}$,
together with the elements $x_1,\ldots,x_m$ that cut out the irreducible
components of 
$(\widehat{(\cX_0)_{\overline{x}_0}})_s$,
form part of a regular system of parameters for
$\widehat{\cO_{\overline{x}_0,\cX_0}}$.
Thus we may choose a model of the form~(\ref{eqn:semistable model}) for
$\widehat{(\cX_0)_{\overline{x}_0}}$ for which $m < n$
and in which $\ell$ is equal to the element $x_n$;
i.e.\ in which $\widehat{(\cY_0)_{\overline{x}_0}}$
is the zero locus of the element $x_n$.

We now choose a closed geometric point $\overline{x}_1$
of $(\cX_1)_s$ lying over $\overline{x}_0$,
as well as a model for the tamely ramified morphism
$\widehat{(\cX_1)_{\overline{x}_1}} \to
\widehat{(\cX_0)_{\overline{x}_0}}$
as in part~(2) of Lemma~\ref{lem:tame}.
Thus this morphism has the form
$$
\Spec \cshO{L}[[y_1,\ldots,y_n]]/(y_1\cdots y_m - \unif_L)
\to
\Spec \cshO{K}[[x_1,\ldots,x_n]]/(x_1\cdots x_m - \unif_K),$$
with $x_j = y_j^e$ for $1 \leq j \leq m$, and $x_j = y_j$ for $m < j \leq n.$
In particular, we see that $x_n = y_n$,
and thus we see that the
induced morphism 
\numequation
\label{eqn:tame divisor morphism}
\widehat{(\cY_1)_{\overline{x}_1}} \to \widehat{(\cY_0)_{\overline{x}_0}}
\end{equation}
can be written as
\nummultline
\label{eqn:explicit model}
\Spec \cshO{L}[[y_1,\ldots,y_{n-1}]]/(y_1\dots y_m - \unif_L)
\\
\to
\Spec \cshO{L}[[x_1,\ldots,x_{n-1}]]/(x_1 \ldots x_m - \unif_K).
\end{multline}
Thus we see that $\cY_0$ and $\cY_1$ are indeed semistable models
of their generic fibres (over $\cO_K$ and $\cO_L$ respectively),
and that the morphism $\cY_1 \to \cY_0$ is tamely ramified.
This completes the verification of~(2).  The
claim of~(3) follows
from the fact that the action of $I(L/K)^{\op} \times G$ on
$(\cY_1)_s$ is the restriction of the corresponding action
on $(\cX_1)_s$, together with the fact that any component
of $(\cY_1)_s$ is contained in a component of $(\cX_1)_s$.
\end{proof}

We now suppose that $E$ is an algebraic extension of $\Q_p$ containing $K_0$.
Recall that if $\rho:G_{K} \to \GL_n(E)$ is a potentially semistable 
representation, then we may attach a Weil--Deligne representation $\WD(\rho)$ to $\rho$
by first passing to the potentially semistable Dieudonn\'e module $D_{\pst}(\rho)$ of
$\rho$, which is a module over $E\otimes_{\Q_p} K_0$,
then fixing an embedding $K_0 \hookrightarrow E$,
and hence a projection $\operatorname{pr}:E\otimes_{\Q_p} K_0 \to E,$
and, finally, forming $\WD(\rho) := E\otimes_{E\otimes_{\Q_p}
  K_0,\operatorname{pr}} D_{\pst}.$ Although $\WD(\rho)$ depends on the choice of the embedding $K_0 \hookrightarrow E$,
up to isomorphism it is independent of this choice, as the Frobenius $\phi$ on
$D_{\pst}(\rho)$ provides isomorphisms between the different choices. (See for example Appendix B of
\cite{MR1639612} and~\cite[p.\ 78-79]{MR2060030} for discussions of this construction and its properties.)

In the tame strictly semistable case, the following result will allow us to describe the inertial 
part of the Weil--Deligne representation associated to the $p$-adic \'etale cohomology 
of $X_1$,
or of a pair $(X_1,Y_1)$ that arises in the context of the preceding proposition.
Before stating the result we introduce some additional notation, and an additional
assumption.

Assume that $G$ is abelian, and let $J$ denote the set of $I(L/K) \times G$ orbits
on the set of irreducible components of $(\cX_1)_s$,
and let $D_{j}$ (for $j \in J$)
denote the union of the components lying in the orbit labelled 
by $j$. Let $\psi_j:I(L/K)\to G$ be the homomorphism provided by Lemma~\ref{lem:inertia computation}, describing the action of
$I(L/K)$ on the points of $D_j$.

\begin{prop}
\label{prop:Weil--Deligne} Suppose that we are in a tame strictly
semistable context as above.
Either 
let $W$ denote the Weil--Deligne representation associated to the potentially
semistable $G_K$-representation
$H^i_{\et}\bigr((X_1)_{/\Kbar},E\bigr)$, 
or else
suppose that we are in the context of Proposition~{\em \ref{prop:tame hyperplane}},
and
let $W$ denote the Weil--Deligne representation associated to the potentially
semistable $G_K$-representation
$H^i_{\et}\bigr((X_1)_{/\Kbar},(Y_1)_{/\Kbar}),E\bigr)$
{\em (}here $i$ is some given degree of cohomology{\em )};
in either case $W$
is a representation of the product $\WD_K \times G$.
Assume furthermore that $G$ is abelian.

Then, if, as in the above discussion, $\tJ$ denotes the
set of $I(L/K)\times G$-orbits of irreducible components 
of $(\cX_1)_s$, we may decompose $W$ as a direct sum
$W = \bigoplus_{\tj \in \tJ} W_{\tj}$, such that on $W_{\tj}$, the action of the
inertia group in $\W_K$ is obtained by composing the $G$-action on $W_{\tj}$
with the homomorphism
$I_K \to I(L/K) \buildrel \psi_{\tj} \over \longrightarrow~G.$
\end{prop}
\begin{proof}Since the action of the inertia subgroup of $\W_K$ on $W$ factors
through a finite group, and representations of a finite group over
a field of characteristic zero are semisimple, the claimed property
of $W$ is stable under the formation of subobjects, quotients, and
extensions (in the category of $W_K\times G$-representations).
A consideration of the long exact sequence of cohomology
associated to the pair $(X_1, Y_1)$ (cf.\ Appendix \ref{appendix: cohomology results on
  pairs etc}) then reduces the claim for 
the cohomology of the pair to the claim for the cohomology of $X_1$ and $Y_1$ individually.
Since Proposition~\ref{prop:tame hyperplane} shows that the strictly semistable model
$\cY_1$ of $(Y_1)_{/L}$ behaves in an identical manner to the strictly semistable model
of $\cX_1$ of $(X_1)_{/L}$,
it in fact suffices to consider the case of $X_1$.

Thus we now restrict our attention to the $\W_K$-representation $W$
underlying the potentially semistable Dieudonn\'e module associated to
$H^i_{\et}\bigr((X_1)_{/\Kbar},E\bigr)$.
By \cite{Tsu} this Dieudonn\'e module is naturally identified with
the log-crystalline cohomology
$H^i\bigr((\cX_1)_s^\times/W(k)^\times\bigr)\otimes_{W(k)}E$
of the special fibre $(\cX_1)_s$ with its natural
log-structure.
Lemma~\ref{lem:intersection} shows that if an intersection
of distinct components of the special fibre $(\cX_1)_s$ is non-empty,
then the various components appearing must lie in mutually distinct
orbits of $I(L/K) \times G$ acting on the set of components. Recalling that $J$ denotes the indexing set for
the collection $\{D_j\}_{j\in J}$ of $I(L/K)\times G$ -orbits of components of $(\cX_1)_s$,
this log-crystalline cohomology may be computed by the following spectral sequence of \cite{Mok}:

\begin{multline*}
E_1^{-m,i+m} = 
\bigoplus_{l\ge \max\{0, -m\}
\atop
\{j_1,\dots,j_{2l+m+1}\}\subset
  J}H^{i-2l-m}_{\cris}\bigl(D_{j_1}\cap\dots\cap
D_{j_{2l+m+1}}/W(k)\bigr)\otimes_{W(k)}E(-l-m)
\\
\implies H^i\bigl((\cX_1)_s^\times/W(k)^\times\bigr)\otimes_{W(k)}E. 
\end{multline*}

The constructions of~\cite{Tsu,Mok} are both functorial, so that everything here
is compatible with the $I(L/K)\times G$-actions. Each of the summands in the $E_1$-term is naturally
an $I(L/K)\times G$-representation, and furthermore the action of
$I(L/K)$ is given by the composite of the action
of $G$ with one of the characters $\psi_{\tj}$.
Thus the $E_1$-terms of this spectral sequence satisfies the claimed property
of $W$.  Thus $W$ also satisfies this property,
since it is obtained as a successive extension of subquotients
of these $E_1$-terms.
\end{proof}

We are now ready to prove our equivariant versions of Theorems~\ref{thm:first cohom} 
and \ref{thm:arb cohom}.  For the first result, we place ourselves in the local
case (since the global case immediately reduces to the local case by passing
from $K$ to $K_v$).

\begin{theorem} 
\label{thm:degree one cohom equivariant}
Suppose that we are in the tame strictly semistable context described above.
Then the $G_K\times G$-representation $H^1_{\et}\bigl((X_1)_{\overline{K}},\Fp\bigr)$ 
embeds $G_K\times G$-equivariantly
into  the reduction
modulo the uniformiser of a $G_K\times G$-invariant
$\mathcal O_E$-lattice in a representation $V$
of $G_K\times G$
over~$E$, having the following properties:
\begin{enumerate}
\item the restriction of $V$ to $G_L$
is semistable,
with Hodge--Tate weights contained in the interval $[-1,0]$;
\item The Weil--Deligne representation associated to $V$,
which is naturally a representation of $\WD_K \times G $,
when restricted to a representation of $I_K \times G$,
can be written as
a direct sum $\bigoplus_{\tj \in \tJ} W_{\tj}$  of $I_K \times G$-representations,
 where $\tj$ runs over the same
index set that labels the set of $I(L/K)\times G$-orbits of irreducible components 
of $(\cX_1)_s$, such that on $W_{\tj}$, the action of the
inertia group in $\W_K$ is obtained by composing the $G$-action on $W_{\tj}$
with the homomorphism
$I_K \to I(L/K) \buildrel \psi_{\tj} \over \longrightarrow~G.$
\end{enumerate}
\end{theorem}
\begin{proof}
We follow the proof of Theorem~\ref{thm:first cohom},
proceeding by descending induction on the dimension of $X_0$ and $X_1$, 
and passing to appropriately chosen hypersurface sections $\mathcal Y_0$ 
of $\mathcal X_0$,
and their corresponding preimages $Y_1$ in $X_1$
and $\mathcal Y_1$ in $(\mathcal X_1)_{/L}$.
Taking into account Proposition~\ref{prop:tame hyperplane},
we thus reduce to the case when $X_0$ and $X_1$ are curves, so that
$H^1_{\et}\bigl((X_1)_{/\Kbar},\F_p\bigr)$
is the reduction mod $p$ of 
$H^1_{\et}\bigl((X_1)_{/\Kbar},\Z_p\bigr)$,
which is in turn a lattice in
$H^1_{\et}\bigl((X_1)_{/\Kbar},\Q_p\bigr)$.
This latter representation is potentially semistable with Hodge--Tate weights
in $[-1,0]$, by \cite{Tsu}; 
the claim regarding Weil--Deligne representations follows from
Proposition~\ref{prop:Weil--Deligne}.
\end{proof}

For our second result, we allow ourselves to be in either one of the local or global
contexts.

\begin{theorem} 
\label{thm:arb cohom equivariant}
Suppose that we are in the tame strictly semistable context described above,
and let $\rho:G_K \times G \to \GL_n(k_E)$ be an irreducible and continuous
representation that embeds as a subquotient of
$H^i_{\et}\bigl((X_1)_{\overline{K}},k_E\bigr)$.
Then $\rho$
also embeds as a
subquotient of a $G_K\times G$-representation
over $k_E$
which is the reduction
modulo the uniformiser of a $G_K\times G$-invariant
$\mathcal O_E$-lattice in a representation $V$
of $G_K\times G$
over~$E$, having the following properties:
\begin{enumerate}
\item the representation $V$
becomes semistable when restricted to $G_L$
{\em (}respectively  to the decomposition group $D_w \subset G_L$
in the global case{\em )},
with Hodge--Tate weights contained in the interval $[-i,0]$;
\item The Weil--Deligne representation associated to $V$,
which is naturally a representation of $\WD_K \times G $
{\em (}respectively of $ \WD_{K_v} \times G$ in the global case{\em )},
when restricted to a representation of $I_K \times G$
{\em (}respectively of $I_{K_v} \times G$ in the global case{\em )},
can be written as
a direct sum
$\bigoplus_{\tj} W_{\tj}$  of $I_K \times G$-representations
{\em (}respectively of $I_{K_v} \times G$-representations{\em )},
where $\tj$ runs over the same
index set that labels the set of $I(L/K)\times G$-orbits {\em (}respectively of $I_{K_v} \times G$-orbits{\em )} of irreducible components 
of $(\cX_1)_s$, such that on $W_{\tj}$, the action of the
inertia group is obtained by composing the $G$-action on $W_{\tj}$
with the homomorphism
$I_K \to I(L/K) \buildrel \psi_{\tj} \over \longrightarrow G$
{\em (}respectively the homomorphism
$I_{K_v} \to I(L_w/K_v) \buildrel \psi_{\tj} \over \longrightarrow G$
in the global case{\em )}.
\end{enumerate}\end{theorem}
\begin{proof}
  We can be proved in exactly the same way as Theorem~\ref{thm:arb
    cohom}, taking into account Propositions~\ref{prop:tame
    hyperplane} and Proposition~\ref{prop:Weil--Deligne}. \end{proof}

\section{Breuil modules with descent data}
\label{sec:Breuil}
In this section we establish a result
(Theorem~\ref{thm: generic vanishing in small HT weights} below)
which imposes some constraints on the reductions of certain
tamely potentially semistable $p$-adic representations of $G_{\mathbb Q_p}$.

\subsection{Preliminaries}We begin by
recalling a variety of
results from Section~3 of~\cite{EGH}.
To this end, 
let $p$ be an odd prime, let $\Qpbar$ be a fixed algebraic closure of
$\Qp$, and let $E$ and $K$ be finite extensions of $\Qp$ inside
$\Qpbar$. Assume that $E$ contains the images of all embeddings $K\into\Qpbar$. Let
$K_0$ be the maximal absolutely unramified subfield of $K$, so that
$K_0=W(k)[1/p]$, where $k$ is the residue field of~$K$. Let $K/K'$ be a Galois extension, with $K'$ a field lying between
$\Qp$ and $K$. Assume further that $K/K'$ is tamely ramified with
ramification index $e$, and fix a uniformiser $\pi\in K$ with
$\pi^{e}\in K'$.  Let $E(u) \in W(k)[u]$ be the minimal
polynomial of $\pi$ over $K_0$.

Let $k_E$ be the residue field of $E$, and let $0\le r\le p-2$ be an
integer.  Recall that the category $\FBrModdd$ of \emph{Breuil modules
  of weight $r$ with descent data} from $K$ to $K'$ and coefficients
$k_E$ consists of quintuples
$(\mathcal{M},\mathcal{M}_{r},\varphi_{r},\hat{g},N)$ where: \begin{itemize}\item $\mathcal{M}$ is a finitely generated
  $(k\otimes_{\F_p}k_E)[u]/u^{ep}$-module, free over $k[u]/u^{ep}$.

\item $\M_{r}$ is a $(k\otimes_{\F_p}k_E)[u]/u^{ep}$-submodule of $\M$
  containing $u^{er}\M$.
\item $\varphi_{r}:\M_{r}\to\M$ is $k_E$-linear and $\varphi$-semilinear
  (where $\varphi:k[u]/u^{ep}\to k[u]/u^{ep}$ is the $p$-th power map)
  with image generating $\M$ as a
  $(k\otimes_{\F_p}k_E)[u]/u^{ep}$-module.
\item $N:\M\to \M$ is $k\otimes_{\F_{p}}k_E$-linear and satisfies
  $N(ux)=uN(x)-ux$ for all $x\in\M$,
  $u^{e}N(\M_{r})\subset\M_{r}$, and
  $\varphi_{r}(u^{e}N(x))=cN(\varphi_{r}(x))$ for all $x\in\M_{r}$. Here,
  $c = \bar F(u)^p \in (k[u]/u^{ep})^\times$, where $E(u) = u^e+pF(u)$
  in $W(k)[u]$.
\item $\hat{g}:\M\to\M$ are additive bijections for each
  $g\in\Gal(K/K')$, preserving $\M_{r}$, commuting with the $\varphi_{r}$-
  and $N$-actions, and satisfying $\hat{g}_1\circ
  \hat{g}_2=\widehat{g_1\circ g_2}$ for all $g_1, 
  g_2\in\Gal(K/K')$. Furthermore, if $a\in
  k\otimes_{\F_{p}}k_E$, $m\in\M$ then
  $\hat{g}(au^{i}m)=g(a)((g(\pi)/\pi)^{i}\otimes
  1)u^{i}\hat{g}(m)$.
\end{itemize}

There is a covariant functor $\Tst^{*,r}$ from $\FBrModdd$ to
the category of $k_E$-rep\-re\-sent\-a\-tion\-s of $G_{K'}$.

\begin{lem}
  \label{lem:unique-sub} Suppose that $\cM\in\FBrModdd$, and that $T'$ is a
  $G_{K'}$-sub\-rep\-re\-sent\-a\-tion of $\Tst^{*,r}(\cM)$ {\em (}so that in particular
  $T'$ has the structure of a $k_E$-vector space{\em )}. Then there is a
  unique subobject $\cM'$ of $\cM$ such that if $f:\cM'\to\cM$ is the
  inclusion map, then $\Tst^{*,r}(f)$ is identified with the inclusion
  $T'\into \Tst^{*,r}(\cM)$. {\em (}Here $\M'$ is
    a subobject of $\M$ in the naive sense that it is a
    sub-$(k\otimes_{\F_p}k_E)[u]/u^{ep}$-module of $\M$, which inherits the
    structure of an object of $\FBrModdd$ from $\M$ in the obvious
    way.{\em )}
\end{lem}
\begin{proof}
  This is Corollary 3.2.9 of \cite{EGH}.\end{proof}

We now specialise to the particular situation of interest to us in
this paper, namely we let $K_0$ be the unique unramified extension of $\Qp$ of
degree $d$, we take $K=K_0((-p)^{1/(p^d-1)})$, and we set $K'=K_0$, so
that $e=p^d-1$. Fix
$\pi=(-p)^{1/(p^d-1)}$. We write $\wt\omega_d:\Gal(K/K_0)\to K_0^\times$ for
the character $g\mapsto g(\pi)/\pi$, and we let $\omega_d$ be the
reduction of $\wt\omega_d$ modulo $\pi$. (By inflation we can also think of
$\wt\omega_d$ and $\omega_d$ as characters of $I_{K_0} = I_\Qp$.
Note that $\omega_d$ is a tame fundamental character
of niveau~$d$ and that $\wt\omega_d$ is the Teichm\"uller lift of $\omega_d$.)
Note that when $d = 1$, we have $\omega_1 = \omega$, the mod $p$ cyclotomic character.

Let $\varphi$ be the arithmetic Frobenius on $k$, and let
$\sigma_0:k\into k_E$ be a fixed embedding. Inductively define
$\sigma_1,\dots,\sigma_{d-1}$ by
$\sigma_{i+1}=\sigma_i\circ\varphi^{-1}$; we will often consider the
numbering to be cyclic, so that $\sigma_d=\sigma_0$.
There are idempotents $e_i\in
k\otimes_\Fp k_E$ such that if $M$ is any $k\otimes_\Fp k_E$-module,
then $M=\bigoplus_ie_iM$, and $e_iM$ is the subset of $M$ consisting of
elements $m$ for which $(x\otimes 1)m=(1\otimes\sigma_i(x))m$ for all
$x\in k$. Note that $(\varphi\otimes 1)(e_i) = e_{i+1}$ for all $i$.

If $\rho:G_{K_0}\to\GL_n(E)$ is a
potentially semistable representation which becomes semistable over
$K$, then the associated inertial type (that is, the restriction to
$I_{K_0}$ of the Weil--Deligne representation associated to $\rho$) is
a representation of $I_{K_0}$ which becomes trivial when restricted to
$I_K$, so we can and do think of it as a representation of
$\Gal(K/K_0)\cong
I_{K_0}/I_K$. \begin{prop}\label{prop: form of descent data on breuil module}
  Maintaining our current assumptions on $K$,
  suppose that $\rho:G_{K_0}\to\GL_n(E)$ is a continuous representation 
  whose restriction to ${G_K}$ is semistable with Hodge--Tate weights contained
  in $[0,r]$, with $r\le p-2$, and let the inertial type of $\rho$ be
  $\chi_1\oplus\dots\oplus\chi_n$, where each $\chi_i$ is a character
  of $I_{K_0}/I_K$. If $\rhobar$ denotes the
  reduction modulo $\m_E$ of a $G_{K_0}$-stable $\cO_E$-lattice in
  $\rho$, then there is an element $\M$ of $\FBrModdd$,
  admitting a  $(k\otimes_{\F_p}k_E)[u]/u^{ep}$-basis $v_1,\dots,v_n$ such that
  $\hat{g}(v_i)=(1\otimes\overline{\chi}_i(g))v_i$ for all
  $g\in\Gal(K/K_0)$,
  and for which $\Tst^{*,r}(\M)\cong\rhobar.$ 
\end{prop}
\begin{proof}
  This is Proposition 3.3.1 of \cite{EGH}. (Note that the conventions
  on the sign of the Hodge--Tate weights in \cite{EGH} are the opposite
  of the conventions in this paper.)
\end{proof}

\begin{lem}\label{lem: form of rank one objects and their generic fibers}
  Maintain our current assumptions on $K$, so that in particular we
  have $e = p^d-1$. Then every
  rank one object of $\FBrModdd$ may be written in the form:
  \begin{itemize}
  \item $\M=((k\otimes_\Fp k_E)[u]/u^{ep})\cdot m$,
  \item $e_i\M_r=u^{r_i}e_i\M$,
  \item $\varphi_r(\sum_{i=0}^{d-1}u^{r_i}e_im)=\lambda m$ for some
    $\lambda\in (k\otimes_\Fp k_E)^\times$,
  \item $\hat{g}(m)=(\sum_{i=0}^{d-1}(\omega_d(g)^{k_i}\otimes
    1)e_i)m$ for all $g\in\Gal(K/K_0)$, and
  \item $N(m) = 0$.
  \end{itemize}
Here the integers $0\le r_i\le (p^d - 1)r$ and $k_i$ satisfy $k_i\equiv p(k_{i-1}+r_{i-1})\bmod{(p^d-1)}$
for all~$i$.
Conversely, any module $\M$ of this form is a rank one object of $\FBrModdd$. Furthermore, 
\[\Tst^{*,r}(\M)|_{I_{K_0}}\cong\sigma_0\circ\omega_d^{\kappa_0},\]
where $\kappa_0\equiv k_0+p(r_0p^{d-1}+r_{1}p^{d-2}+\dots+r_{d-1})/(p^d-1) \bmod (p^d-1)$.
\end{lem}
\begin{proof}
  This is Lemma 3.3.2 of \cite{EGH}.
\end{proof}
\begin{rem} 
  \label{rem:generic fibre of niveau 1 breuil module}{\em In the
    sequel, we will only be interested in the case that for each $i$
    we have $k_i=(1+p+\dots+p^{d-1})x_i$ for some $0\le x_i<p-1$. In
    this case, the condition that $pr_i\equiv
    k_{i+1}-pk_i\bmod(p^d-1)$ implies that $r_i\equiv
    (1+p+\dots+p^{d-1})(x_{i+1}-x_i)\bmod(p^d-1)$, so the condition
    that $0\le r_i\le (p^d-1)r$ means that we can write
    $r_i=(1+p+\dots+p^{d-1})(x_{i+1}-x_i)+(p^d-1)y_i$ with $0\le
    y_i\le r$. An elementary calculation shows that we then
    have \[\kappa_0\equiv x_0+y_0+p^{d-1}(x_1+y_1)+\dots+p(x_{d-1}+y_{d-1})\bmod{(p^d-1)}.\]}
\end{rem}

\subsection{Regularity }Let $\Qpn$ denote the unique unramified extension of $\Qp$ of
degree~$n$, with residue field $\F_{p^n}$. Regarding $\F_{p^n}$ as a subfield
of $\Fpbar$, we may then regard $\omega_n$ as a character $I_{\Qp}\to\Fpbartimes$.
\begin{defn}
  Let $\rhobar:\GQp\to\GL_n(\Fpbar)$ be an irreducible representation,
  so that $\rhobar\cong\Ind_{G_\Qpn}^\GQp\chi$ for some character
  $\chi:\GQp\to\Fpbartimes$. If we write
  $$\chi|_\IQp=\omega_n^{(a_0+pa_1+\dots+p^{n-1}a_{n-1})},$$
  where each $a_i\in [0,p-1]$ and not all the $a_i=p-1$, then the multiset of \emph{exponents} of
  $\rhobar$ is defined to be the multiset of residues of the $a_i$ in
  $\Z/p\Z$.
\end{defn}
\begin{defn}
  Let $\rhobar:\GQp\to\GL_n(\Fpbar)$ be a representation. Then the
  multiset of \emph{exponents} of $\rhobar$ is the union of the
  multisets of exponents of each of the Jordan--H\"older factors of $\rhobar$.
\end{defn}

\begin{defn}
  Let $\rhobar:\GQp\to\GL_n(\Fpbar)$ be a representation, and let
  $r$ be a non-negative integer. Then we say that $\rhobar$ is \emph{$r$-regular} if
  the exponents $a_1,\dots,a_n$ of $\rhobar$ are such that the
  residues $a_i+k\in\Z/p\Z$, $1\le i\le n$, $0\le k\le r+1$, are
  pairwise distinct.
\end{defn}

The following theorem is the main result we will need from explicit
$p$-adic Hodge theory.
\begin{thm}
  \label{thm: generic vanishing in small HT weights}Let $r$ be a non-negative integer, and let $s:\GQp\to\GL_m(\Qpbar)$ be a potentially semistable
  representation, with Hodge--Tate weights contained in the interval
  $[0,r]$, and
  inertial type $\chi_1\oplus\cdots\oplus\chi_m$. Suppose that there
  are {\em (}not necessarily distinct{\em )} integers $0\le a_1,\dots,a_n<p-1$
  such that each $\chi_i$ is equal to some $\omegat^{a_j}$.

Suppose that $\rhobar:\GQp\to\GL_n(\Fpbar)$ is a subquotient of $\overline{s}$,
the reduction mod $\m_{\cO_{\Qpbar}}$ of some $\GQp$-stable
$\Zpbar$-lattice in $s$. Suppose also that
\begin{itemize}
\item $\det\rhobar|_\IQp=\omega^{a_1+\dots+a_n+n(n-1)/2}$, 
\item $r\le (n-1)/2$, and
\item $p> n(n-1)/2+1$.
\end{itemize}
If $r=(n-1)/2$ then assume further that some irreducible subquotient
of $\rhobar$ has dimension greater than one. Then $\rhobar$ is not $r$-regular.
\end{thm}
\begin{proof}Modifying the choice of $\Zpbar$-lattice if necessary, it suffices
  to treat the case that $\rhobar$ and $\overline{s}$ are semisimple. Let
  $\rhobar'\cong\Ind_{G_{\Q_{p^d}}}^\GQp\chi$ be an irreducible
  subrepresentation of~$\rhobar$. Take $K'=K_0=\Q_{p^d}$ in the above
  notation, so that $e=p^d-1$. Taking $E$ to be sufficiently large so
  that $s$ is defined over $E$ and $\rhobar$ is defined over $k_E$, and
  applying Proposition~\ref{prop: form of descent data on breuil
    module} to $s|_{G_{\Q_{p^d}}}$, we see that there is an element  $\M$ of $\FBrModdd$
  with \[\Tst^{*,r}(\M)\cong\overline{s}|_{G_{\Q_{p^d}}},\] such that $\M$ has a
   $(k\otimes_{\F_p}k_E)[u]/u^{ep}$-basis $v_1,\dots,v_m$ such that
  $\hat{g}(v_i)=(1\otimes\overline{\chi}_i(g))v_i$ for all
  $g\in\Gal(K/K_0)$. Since $\overline{s}|_{G_{\Q_{p^d}}}$ contains a
  subrepresentation isomorphic to~$\chi$, we see from Lemma
  \ref{lem:unique-sub} that there is a rank one subobject $\cN$ of
  $\cM$ for which $\Tst^{*,r}(\cN)\cong \chi$. 

Since $\cN/u\cN$ embeds into $\cM/u\cM$ (as $\cN$ is a free $k[u]/u^{ep}$-submodule
of the free $k[u]/u^{ep}$-module $\cM$), we see from Lemma \ref{lem: form of
  rank one objects and their generic fibers} and our assumption on the
characters $\chibar_i$ that we may write $\cN$ in the form \begin{itemize}
  \item $\cN=((k\otimes_\Fp k_E)[u]/u^{(p^d-1)p})\cdot w$,
  \item $e_i\cN_r=u^{r_i}e_i\cN$,
  \item $\varphi_r(\sum_{i=0}^{d-1}u^{r_i}e_iw)=\lambda w$ for some
    $\lambda\in (k\otimes_\Fp k_E)^\times$,
  \item $\hat{g}(w)=(\sum_{i=0}^{d-1}(\omega_d(g)^{k_i}\otimes
    1)e_i)w$ for all $g\in\Gal(K/K_0)$, and
  \item $N(w) = 0$.
  \end{itemize}
  Here the integers $0\le r_i\le (p^d-1)r$ and $k_i$ satisfy
  $k_i\equiv p(k_{i-1}+r_{i-1})\pmod{p^d-1}$ for all $i$, and each
  $k_i$ is equal to some $(1+p+\dots+p^{d-1})a_j$. (The conditions on
  the $r_i$ come from Lemma \ref{lem: form of rank one objects and
    their generic fibers}, and the fact that each $k_i$ is equal to
  some $(1+p+\dots+p^{d-1})a_j$ comes from the fact that $\cN$ is a
  submodule of $\M$, which has a basis $v_1,\dots,v_n$ such that
  $\hat{g}(v_i)=(1\otimes\overline{\chi}_i(g))v_i$, and the assumption
  that each $\chibar_i$ is equal to some $\omega^{a_j}$.)
 Writing $k_i=(1+p+\dots+p^{d-1})x_i$, $0\le
  x_i<p-1$, we see as in Remark~\ref{rem:generic fibre of niveau 1 breuil module} that we can write
  $r_i=(1+p+\dots+p^{d-1})(x_{i+1}-x_i)+(p^d-1)y_i$, with $0\le y_i\le
  r$. By Lemma~\ref{lem: form of rank one objects and their generic fibers}
and Remark~\ref{rem:generic fibre of niveau 1 breuil module}, we have\[\chi|_\IQp=\sigma_0\circ\omega_d^{x_0+y_0+p^{d-1}(x_{1}+y_{1})+\dots+p(x_{d-1}+y_{d-1})}.\]
  Since $r\le p-2$ we have $0\le y_i\le p-2$, and we conclude (after
  allowing for ``carrying'') that each exponent of $\rhobar$ is of the
  form $a_j+k$, with $0\le k\le r+1$.

Suppose that $\rhobar$ is $r$-regular. It must then be the case that
the $x_i$ as above are all distinct. Applying this analysis to each
irreducible subrepresentation of $\rhobar$, we conclude that the $a_i$
are all distinct. Since we have
$\det(\rhobar')|_\IQp=\chi^{1+p+\dots+p^{d-1}}|_\IQp=\omega^{(x_0+y_0)+\dots+(x_{d-1}+y_{d-1})}$,
we conclude that $\det\rhobar|_\IQp=\omega^{a_1+\dots+a_n+y}$ for some
$0\le y\le nr\le n(n-1)/2$. The assumption on $\det\rhobar$ then
implies that $y\equiv n(n-1)/2\pmod{p-1}$. If in fact $r<(n-1)/2$ then
we have $0\le y<n(n-1)/2$, which contradicts the assumption that
$p>n(n-1)/2+1$.

It remains to treat the case that $r=(n-1)/2$, in which case we may
assume (by the additional hypothesis that we have assumed in this
case) that the representation $\rhobar'$ above has dimension $d>1$. By
the above analysis we must have $y=n(n-1)/2$, so that each
$y_i=r$. Since we have $r_i\le (p^d-1)r$, we must have $x_{i+1}-x_i\le 0$
for each $i$, so that in fact $x_0=x_1=\dots=x_{d-1}$, a contradiction
(as we already showed that the $x_i$ are distinct).\end{proof}

\section{The cohomology of Shimura varieties}
\label{sec:Shimura}

\subsection{The semistable reduction of certain $U(n-1,1)$-Shimura varieties}
\label{subsec:semistable reduction}
Fix $n\ge 2$, and fix an odd prime $p$.\footnote{The reason for assuming that the prime $p$ is odd
is that below we will want to apply the discussion and results of Subsection~\ref{sec:Breuil},
in which this assumption was made.}
We now recall the
definitions of the $U(n-1,1)$-Shimura varieties with which we will
work, and some associated integral models. For simplicity we work over
$\Q$ (or rather an imaginary quadratic extension of $\Q$) rather than
over a general totally real field.

For the most part we will follow Section~3 of~\cite{HR} (which uses a
similar approach to \cite{MR1902647}), with the occasional reference
to \cite{ht}. Fix an imaginary quadratic field $F$ in which the prime
$p$ splits, say  $(p)=\gp\overline{\gp}$ for some choice of $\p$, let $x\mapsto\xbar$ be
the nontrivial automorphism of $F$, and regard $F$ as a subfield of
$\C$ via a fixed embedding $F\into\C$.

Let $D$ be a division algebra over $F$ of dimension $n^2$, and let
$*$ be an involution of $D$ of the second kind (that is, $*|_F$ is
nontrivial). Assume that $D$ splits at $\p$ (and hence at $\overline{\gp}$),
and fix isomorphisms
$D_\p\cong M_n(\Qp)$ and $D_{\overline{\p}}\cong M_n(\Qp)$ with the property that under the induced
isomorphism \[D\otimes\Qp\cong M_n(\Qp)\times M_n(\Qp)^\op,\] the
involution $*$ corresponds to $(X,Y)\mapsto(Y^t,X^t)$.

Let $G_{/\Q}$ be the algebraic group whose $R$-points
are \[G(R)=\{x\in(D\otimes_\Q R)^\times|x\cdot x^*\in R^\times\}\]for
any $\Q$-algebra $R$. Thus our fixed isomorphism $D_\p\cong M_n(\Qp)$
induces an isomorphism $G\times_\Q\Qp\cong\GL_n\times\Gm$.

Now let $h_0:\C\to D_\R$ be an $\R$-algebra homomorphism with the
properties that $h_0(z)^*=h_0(\zbar)$ and the involution $x\mapsto
h_0(i)^{-1}x^*h_0(i)$ is positive (that is, $\tr_{B/\Q}(xh_0(i)^{-1}x^*h_0(i))>0$ for all nonzero $x$). Let $B=D^\op$ and let $V=D$, which
we consider as a free left $B$-module of rank one by multiplication on
the right. Then $\End_B(V)=D$, and one can find an element $\xi\in
D^\times$ with the properties that $\xi^*=-\xi$ and such that the
involution $\iota$ of $B$ defined by $x^\iota=\xi x^*\xi^{-1}$ is
positive  (see
Section I.7 of \cite{ht} or Section 5.2 of \cite{MR2192017} for the
existence of such a $\xi$).

We have an alternating pairing $\psi(\cdot,\cdot):D\times D\to\Q$
defined by $\psi(x,y)=\tr_{D/\Q}(x\xi y^*)$, and one sees easily that
$\psi(bx,y)=\psi(x,b^\iota y)$, and that $\psi(\cdot,h_0(i)\cdot)$ is
either positive- or negative-definite. After possibly replacing $\xi$
by $-\xi$, we can and do assume that it is
positive-definite.

It is easy to see that one has \[G(\R)\cong \operatorname{GU}(r,s)\]
for some $r$, $s$ with $r+s=n$. We impose the additional assumption
that in fact $\{r,s\}=\{n-1,1\}$. Note that by Lemma I.7.1 of
\cite{ht} one can find division algebras $D$ for which this
holds. We say that a compact open subgroup $K\subset G(\A^\infty)$
(respectively $K^p\subset G(\A^{p,\infty})$) is \emph{sufficiently small} if
for some prime $q$ (respectively some prime $q\ne p$) the
projection of $K$ (respectively $K^p$) to $G(\Q_q)$ contains no element of finite order
other than $1$. If $K$ is sufficiently small, we will consider the Shimura
variety $Sh(G,h_0|_{\C^\times}^{-1},K)$. It has a canonical model over
 $F$, which we denote by $X(K)$ (note that if $n>2$ the reflex
field is $F$, while if $n=2$ the reflex field is $\Q$, and we let
$X(K)$ denote the base-change of the canonical model from $\Q$ to $F$).

We say that a compact open subgroup $K$ of $G(\A)$ is of {\em level dividing $N$},
for some integer $N \geq 1$,
if for all primes $l\nmid
N$ we can write $K=K_lK^l$, where $K_l$ is a hyperspecial maximal
compact subgroup of $G(\Ql)$ and $K^l$ is a compact open subgroup of
$G(\A^{\infty,l})$.  (Note then that in fact $K = K_N \times \prod_{l \nmid N} K_l$,
for some compact open subgroup $K_N$ of $\prod_{l | N} G(\Ql)$.)
If $K$ is of level dividing $N$, 
then we similarly refer to $X(K)$ as a
{\em $U(n-1,1)$-Shimura variety of level dividing
$N$}.

We will now define integral models of these Shimura varieties for two
specific kinds of level structure.  We begin by introducing notation
related to the level structures in question.

We write $I_0$ for the Iwahori subgroup of
$\GL_n(\mathbb Q_p)\times \Q_p^{\times}$;
namely, $I_0$ is the subgroup of
$\GL_n(\Zp)\times\Zptimes$ consisting of elements whose first factor
lies in the usual Iwahori subgroup of matrices which are upper
triangular mod $p$. We write $I_1$ to denote the pro-$p$-Iwahori subgroup of 
$\GL_n(\mathbb Q_p)\times \Q_p^{\times}$; namely, $I_1$ is the 
(unique) pro-$p$ Sylow subgroup of $I_0$, and consists of those elements
of 
$\GL_n(\Zp)\times\Zptimes$ whose first factor
is upper triangular unipotent mod $p$, and whose second factor is congruent to $1 \bmod p$.
We write $I_1^*$ to denote the subgroup of $I_0$ consisting of matrices in
$\GL_n(\Zp)\times\Zptimes$ whose first factor
is upper triangular unipotent mod $p$. 

There is a natural isomorphism $\Z_p^{\times} = \F_p^{\times} \times (1 + p\Z_p),$
and this induces a natural isomorphism $I_1^* = \F_p^{\times} \times I_1.$
If we let $T$ denote the diagonal torus in $\GL_n$, then there is also a natural isomorphism
$I_0 = T(\mathbb F_p) \times I_1^*.$

We will define
integral models for $X(I_0K^p)$ and $X(I_1^*K^p)$ over the local rings
$\cO_{F,(\p)}$
and $\cO_{F(\zeta_{p-1}),(v)}$ respectively, where $\zeta_{p-1}$ denotes
a primitive $(p-1)$st root of unity, and $v$ is some fixed place
in $F(\zeta_{p-1})$ above $\p$;
here $K^p$ is a sufficiently small compact open subgroup
of $G(\A^{\infty,p})$, and we consider $I_0$ and $I_1^*$ as subgroups of
$G(\Qp)\cong\GL_n(\Qp)\times\Qptimes$. We will typically not
include $K^p$ in the notation, and we will write $\cX_0(p)$,
$\cX_1(p)$ for our integral models of $X_0(p):=X(I_0K^p)$ and $X_1(p):=X(I_1^*K^p)$ respectively.

\begin{remark}
{\em Note that in~\cite{HR},  
\cite{MR1902647}, 
and \cite{ht}, the authors work over~$\mathbb Z_p$, but we follow \cite{MR1124982}
in working over $\cO_{F,(\mathfrak p)}$, so as to satisfy the hypothesis
required to be in the global case of Section~\ref{sec:cohom}.
The appearance of $\zeta_{p-1}$ in the ring of definition
of $\cX_1(p)$ is a consequence of our use
of Oort--Tate theory in the definition of the integral model
in this case.
}
\end{remark}

In order to define
these integral models,
we firstly recall a certain category of abelian schemes (up to isogeny) with
polarisations and endomorphisms. If $\cF$ is a set-valued functor on the
category of connected, locally noetherian $\cO_{F,(p)}$-schemes, we will also
consider it to be a functor on the category of all locally noetherian
$\cO_{F_(p)}$-schemes by setting \[\cF(\coprod S_i):=\prod\cF(S_i).\] Let $\cO_B$ be the unique maximal $\Z_{(p)}$-order in $B$
which under our fixed identification $B\otimes_\Q\Qp=M_n(\Qp)\times
M_n(\Qp)^\op$ is identified with $M_n(\Z_{(p)}\times M_n(\Z_{(p)})^\op$. Let $S$ be a
connected, locally noetherian $\cO_{F,(\p)}$-scheme, and let $AV_S$ be the
category whose objects are pairs $(A,i)$, where $A$ is an abelian scheme over
$S$ of dimension $n^2$, and $i:\cO_B\to\End_S(A)\otimes\Z_{(p)}$ is a
homomorphism. We define homomorphisms in $AV_{S}$ by
\[\Hom\bigl((A_1,i_1),(A_2,i_2)\bigr)
=\Hom_{\cO_B}\bigl((A_1,i_1),(A_2,i_2)\bigr)\otimes\Z_{(p)}\]
(that is, the elements of $\Hom_S(A_1,A_2)\otimes\Z_{(p)}$ which
commute with the action of~$\cO_B$). The dual of an object $(A,i)$ of
$AV_{S}$ is $(\hat{A},\hat{i})$, where $\hat{A}$ is the dual abelian
scheme of $A$, and $\hat{i}(b)=(i(b^\iota))^\wedge$. A polarisation of
$(A,i)$ is a homomorphism $\lambda:(A,i)\to (\hat{A},\hat{i})$  in $AV_S$ with
the property that for some $n\ge 1$, $n\lambda$ is induced an ample
line bundle on~$A$. A principal polarisation is a polarisation which
is also an isomorphism in $AV_S$. A $\Q$-class of polarisations
is an equivalence class of homomorphisms $(A,i)\to(\hat{A},\hat{i})$ which contains a
polarisation, under the equivalence relation of differing by a 
$\Q^\times$-scalar.

Fix $K^p$ a sufficiently small open
compact subgroup of $G(\A^{\infty,p})$. Let $\cA_0$ be the set-valued functor on
the category of locally noetherian schemes over
$\cO_{F,(\p)}$ which sends a connected, locally noetherian scheme $S$ over $\cO_{F,(\p)}$ to the set of
isomorphism classes of the following data.
\begin{itemize}
\item A commutative diagram of morphisms in the category $AV_{S}$ of
  the
  form\[\xymatrix{A_0\ar[r]^{\alpha_0}\ar[d]^{\lambda_0}&A_1\ar[r]^{\alpha_1}\ar[d]^{\lambda_1}&\cdots\ar[r]^{\alpha_{n-2}}&A_{n-1}\ar[r]^{\alpha_{n-1}}\ar[d]^{\lambda_{n-1}}&A_0\ar[d]^{\lambda_0}\\ 
 \hat{A}_0&\hat{A}_{1}\ar[l]_{\hat{\alpha}_{0}}&\cdots\ar[l]_{\hat{\alpha}_{1}}&\hat{A}_{n-1}\ar[l]_{\hat{\alpha}_{n-2}}&\hat{A}_0\ar[l]_{\hat{\alpha}_{n-1}}   }\]where each $\alpha_i$ is an isogeny of degree $p^{2n}$, and
their composite is just multiplication by $p$. In addition, $\lambda_0$
is a $\Q$-class of polarisations containing a principal polarisation. Furthermore, we require
that each $A_i$ satisfies a compatibility between the two actions
of $\cO_{F}$ on the Lie algebra of $A_i$ (one action coming from the structure
morphism $\cO_{F,(\p)}\to\cO_S$, and the other from the $\cO_B$-action;
see~\cite[\S III.4]{ht} for a discussion of this condition).
\item A geometric point $s$ of $S$, and a $\pi_1(S,s)$-invariant $K^p$-orbit of isomorphisms \[\etabar:V\otimes_\Q\A^{p,\infty}\isoto
  H_1((A_0)_s,\A^{p,\infty})\] which are $\cO_B$-linear, and up
  to a constant in $(\A^{p,\infty})^\times$ take the $\psi$-pairing
  on the left hand side to the $\lambda_0$-Weil pairing on the right
  hand side. (This data is canonically independent of the choice of $s$; see the
  discussion on pages 390-391 of~\cite{MR1124982}.)

\end{itemize}
An isomorphism of this data is one induced by isomorphisms in
$AV_{S}$ which preserve the $\lambda_i$ up to an overall
$\Q^\times$-scalar. 

The functor $\cA_0$ is represented by a projective scheme $\cX_0(p)$
over $\cO_{F,(\p)}$, which is an integral
model for $X(I_0K^p)$. (See the proof of Lemma~3.2 of~\cite{ty}, which
shows that $\cX_0(p)$ is projective over the usual integral
model at hyperspecial level. At hyperspecial level, quasi-projectivity
is proved on page 391 of~\cite{MR1124982}, and projectivity can be
checked via the valuative criterion for properness as on page 392
of~\cite{MR1124982}. More properly, $\cX_0(p)$
 is an integral
model for a disjoint union of a number of copies of $X(I_0K^p)$,
due to the possible failure of the Hasse
principle; see for example section 7 of~\cite{MR1124982}, as well
as the discussion on p.~400 of the same reference. Since the
cohomology of a disjoint union of spaces is the direct sum of the
cohomologies of the individual spaces, this does not affect our
arguments, and we will not dwell on this point in the below.)
The proof of Proposition 3.4(3) of~\cite{ty}
(which goes over unchanged in our setting) shows that the special
fibre of $\cX_0(p)$ is a strict normal crossings divisor.

Our next goal is to describe an integral model $\cX_1(p)$, over
$\cO_{F(\zeta_{p-1}),(v)}$,
for $X(I_1^*K^p)$ (or rather, as in the previous paragraph, an
integral model of a
disjoint union of a number of copies of~$X(I_1^*K^p)$). Recalling that $T$ denotes the diagonal torus in $\GL_n$, we let
$\imath_i:\G_m \to T$ denote the embedding of tori identifying $\G_m$
with the subgroup of $T$ consisting of elements which are $1$ away
from the $i$th diagonal entry.  We use the same notation $\imath_i$ to
denote the map $\F_p^{\times} \to T(\mathbb F_p)$ induced by the map
of tori.

The quotient $I_0/I_1^*$
is naturally identified with $T(\mathbb F_p)$, and so $T(\mathbb F_p)$
acts on $X_1(p),$ with quotient isomorphic to $X_0(p)$.

Given an $S$-valued point of $\cA_0$, let $A_i(p^\infty)$ be the $p$-divisible
group associated to $A_i$, $i=0,\dots,n-1$. Each $A_i(p^\infty)$ has an action
of $\cO_B=M_n(\Z_{(p)})\times M_n(\Z_{p})^\op$. Let $X_i=e_{11}A_i(p^\infty)$,
where $e_{11}$ is the usual idempotent in $M_n(\Z_{(p)})$ (and is zero on the
second factor). Then each $X_i$ is a $p$-divisible group of height $n$ and
dimension~$1$, and we obtain a chain of isogenies of degree $p$ $$\mathcal C : X_0
\buildrel \alpha_0 \over \longrightarrow
X_1 
\buildrel \alpha_1 \over \longrightarrow
\cdots 
\buildrel \alpha_{n-2} \over \longrightarrow
X_{n-1} 
\buildrel \alpha_{n-1} \over \longrightarrow
 X_n := X_0,$$
whose composite is equal to multiplication by $p$.

We let $\OT$ denote the Artin stack over $\cO_{F(\zeta_{p-1}),(v)}$ given
by 
$$\OT:= \bigl[\Spec \cO_{F(\zeta_{p-1}),(v)}[X,Y]/(X Y - w_p)/\mathbb G_m\bigr],$$
where $\mathbb G_m$ acts via $\lambda\cdot (X,Y) = (\lambda^{p-1} X,
\lambda^{1-p}Y),$ and $w_p$ is some explicit element of 
$\cO_{F(\zeta_{p-1}),(v)}$ of valuation one.  Oort--Tate theory shows
that $\OT$ classifies finite flat group schemes of order $p$
over $\cO_{F(\zeta_{p-1}),(v)}$-schemes.
The universal group scheme over $\OT$ is
the stack 
$$\mathcal G:=
\bigl[ \Spec \cO_{F(\zeta_{p-1}),(v)}[X,Y,Z]/(X Y - w_p, Z^p - X Z)/\mathbb G_m\bigr],$$
where $\mathbb G_m$ acts on $X$ and $Y$ as above,
and on $Z$ via $\lambda\cdot Z = \lambda Z$.
The morphism $\mathcal G \to \OT$ is the evident one,
the zero section of $\mathcal G$ is cut out by the equation $Z = 0$,
and we let $\mathcal G^{\times}$ denote the closed subscheme of
$\mathcal G$ cut out by the equation $Z^{p-1} - X=0$; it is 
the so-called {\em scheme of generators} of $\mathcal G$. (See
Theorem~6.5.1 of~\cite{MR2234862} for these facts, which are a
restatement of Theorem 2 of~\cite{MR0265368} in the language of stacks.)

We define $\cX_1(p)$ via the Cartesian diagram
\numequation
\label{eqn:Oort--Tate diagram}
\xymatrix{\cX_1(p) \ar[d] \ar[r] &
\mathcal G^{\times} \times_{\cO_{F(\zeta_{p-1}),(v)}} \cdots \times_{\cO_{F(\zeta_{p-1}),(v)}} \mathcal G^{\times} 
\ar[d] \\
\cX_0(p)_{/\cO_{F(\zeta_{p-1}),(v)}} \ar[r] & 
\OT\times_{\cO_{F(\zeta_{p-1}),(v)}} \cdots \times_{\cO_{F(\zeta_{p-1}),(v)}} \OT,}
\end{equation}
where the bottom horizontal arrow is given by
\begin{equation*}
\mathcal C \mapsto \bigl(\ker(\alpha_0),\ldots,
\ker(\alpha_{n-1})\bigr).
\end{equation*}
Note that the right-hand vertical arrow is finite and relatively representable by
construction, and so the left-hand vertical arrow is a finite morphism of schemes.
The action of $T(\mathbb F_p)$ on $X_1(p)$ extends to an action
on $\cX_1(p)$, namely the action pulled back from the action of
$T(\mathbb F_p)$ on 
$\mathcal G^{\times} \times_{\cO_{F(\zeta_{p-1}),(v)}} \cdots \times_{\cO_{F(\zeta_{p-1}),(v)}} \mathcal G^{\times}$
over
$\OT\times_{\cO_{F(\zeta_{p-1}),(v)}} \cdots \times_{\cO_{F(\zeta_{p-1}),(v)}} \OT.$

Let $\pi := (-p)^{1/(p-1)}$, and let $w$ be the unique finite place of
$L:=F(\zeta_{p-1},\pi) $ lying over our fixed place $v$ of
$F(\zeta_{p-1})$. We let $\mathcal O$ denote the localization of $\cO_L$ at~$w$, and 
we let $\cX_1(p)_{\mathcal O}$ denote the normalisation of 
the base-change $\cX_1(p)_{/\mathcal O}$ of $\cX_1(p)$ over $\mathcal O$.
We write $I := \Gal\bigr(L/ F(\zeta_{p-1})\bigl);$ this is also the inertia group
at~$w$ in $\Gal(L/F)$. The group $I$ acts naturally on
$\cX_1(p)_{\mathcal O}$. The mod $p$ cyclotomic character $\omega$ induces an isomorphism (which we continue to denote
by~$\omega$)
$$\omega:I \iso \mathbb F_p^{\times}.$$  For each $i = 0,\ldots,n-1$,
we let $\alpha_i:I \to T$ denote the composite $\imath_i \circ \omega^{-1}.$

\begin{lemma}\label{lem:semistable model}
The scheme $\cX_1(p)_{\mathcal O}$
is a semistable projective model for $X_1(p)_{/L}$ over $\mathcal O$,
the natural morphism 
\numequation
\label{eqn:x1 to x0}
\cX_1(p)_{\mathcal O} \to \cX_0(p)
\end{equation}
is tamely ramified,
and the action of $I\times T$ on $X_1(p)_{/L}$ extends to an action
on $\cX_1(p)_{\cO}$.   
Furthermore,
on each irreducible component of its special fibre,
the inertia group $I$ acts through the composite of 
the $T$-action with one of the characters $\alpha_i$.
\end{lemma}
\begin{proof}
We will apply a form of Deligne's homogeneity principle, as described in the proof of
\cite[Prop.~3.4]{ty}, to the morphism~(\ref{eqn:x1 to x0}). 
The scheme here denoted $\cX_0(p)$
is there denoted $X_{U}$ (and the integral model there is considered over $\Z_p$ rather
than $\cO_{F,(\p)}$, but this is immaterial for our present purposes),
while the scheme there denoted $X_{U_0}$ is an integral model
of the Shimura variety (in the notation of the present paper) $X(K_p K^p)$, where 
$K_p = \GL_n (\Z_p)\times \Z_p^{\times}.$   
We let $\cX_0(p)_s^{(h)}$ (for $0 \leq h \leq n-1$) denote the locally closed subset
of the special fibre $\cX_0(p)_s$, obtained by pulling back the
locally closed subset
$\overline{X}_{U_0}^{(h)}$ defined in section 3 of \cite{ty} under the natural projection $\cX_0(p)_s = \overline{X}_U \to
\overline{X}_{U_0}.$ 

We will show that the morphism~(\ref{eqn:x1 to x0}) 
is tamely ramified in the formal neighbourhood of any closed geometric point
of the special fibre $\cX_0(p)_s$, and hence (by Lemma \ref{lem:tame})
that it is tamely ramified.  We first note that the morphism
$\cX_0(p)_{/\cO_{F(\zeta_{p-1}),(v)}} \to \cX_0(p)$ induces an isomorphism
on special fibres, and so we are free
to replace $\cX_0(p)$ by $\cX_0(p)_{/\cO_{F(\zeta_{p-1}),(v)}}$ in our considerations.
We next note that 
the completion of the bottom arrow of~(\ref{eqn:Oort--Tate diagram}) at
a closed geometric point $\overline{x}_0$ of $\cX_0(p)_s$ depends up to isomorphism only on the
value of $h$ for which $\overline{x}_0 \in \cX_0(p)^{(h)}_s$ (since the $p$-divisible 
group attached to the point $\overline{x}_0$ depends only on the value of~$h$), 
and hence that
the restriction of~(\ref{eqn:x1 to x0})
to a formal neighbourhood of $\overline{x}_0$ depends only on the value of $h$.
Then
by \cite[Lem.~3.1]{ty}, the closure of $\cX_0(p)^{(h)}_s$ contains $\cX_0(p)^{(0)}_s$.
Since being tamely ramified is an open condition, we conclude from these
two conditions that in order to prove the lemma,
it suffices to show that 
the restriction of~(\ref{eqn:x1 to x0})
to a formal neighbourhood of $\overline{x}_0$ is tamely ramified at closed geometric
points $\overline{x}_0$ of $\cX_0(p)_s^{(0)}$.
(As already indicated, this argument is a variation on Deligne's homogeneity principle.)

Thus, consider a closed geometric 
supersingular point $\overline{x}_0$ of  $\cX_0(p)_s^{(0)}$, so that
$\overline{x}_0$ admits a formal neighbourhood of the form
$$\Spec W(\Fbar_p)[[T_1,\ldots,T_n]]/(T_1 \cdots T_n - w_p).$$
The proof of \cite[Prop.~3.4]{ty} shows that the $T_i$ may be taken to
be the matrix of $\alpha_{i-1}$ on tangent spaces, so that 
the map $\cX_0(p) \to \OT \times_{\cO_{F(\zeta_{p-1}),(v)}} \cdots \times_{\cO_{F(\zeta_{p-1}),(v)}} \OT$ may
be defined in the formal neighbourhood of $\overline{x}_0$ by the map
$$(T_1,\ldots,T_n) \mapsto \bigl( (T_1,U_1),\ldots,(T_n,U_n)\bigr),$$
where $U_i = T_1\cdots \hat{T_i} \cdots T_n $ (and, as is usual in these
situations, a hat on a variable denotes that that variable is omitted in the expression).
Thus a formal neighbourhood of a closed geometric point lying over $\overline{x}_0$ 
in $(\cX_1(p)_{/\mathcal O})_s$ is isomorphic to
$$\Spec W(\Fbar_p)[\pi][[V_1,\ldots,V_n]]/\bigl( (V_1\cdots V_n)^{p-1} - w_p\bigr).$$
If we write $u:= V_1\cdots V_n/\pi,$ then we see that 
$u^{p-1} = -w_p/p,$ and hence that $u$ lies in the normalisation of this formal neighbourhood.
Furthermore, on each component of this normalisation, $u$ is equal to one of 
the $(p-1)$st roots of $-w_p/p$ lying in $W(\Fbar_p)$.
Thus the normalisation of this formal neighbourhood is a union of components,
each isomorphic to
$$\Spec \cO[[V_1,\ldots,V_n]]/(V_1\cdots V_n - u\pi),$$
with the morphism~(\ref{eqn:x1 to x0}) being given by $T_i = V_i^{p-1}$.
Thus this morphism is indeed tamely ramified in the formal neighbourhood of $\overline{x}_0$.

It is clear that the $I\times T$-action on $X_1(p)_{L}$ extends to
an action on $\cX_1(p)_{/\cO}$, and hence to its normalisation $\cX_1(p)_{\cO}.$
As for the final statement, note that $I$ acts on $\pi$ via~$\omega$,
and hence on $u$ via~$\omega^{-1}$,  while $I$ fixes each $V_i$.
Also $T$ acts on $V_i$ through multiplication by the $i$th diagonal entry,
and so acts on $u$ via multiplication by the determinant.
Combining these facts, we see that $I$ acts on the components of the special fibre
on which $V_i = 0$ via $\imath_i\circ \omega^{-1}$.
\end{proof}

\begin{remark}\label{rem:shimura gives a strictly semistable context}
{\em
This lemma (and the fact that $\cX_0(p)$ is strictly semistable) shows that the map $\cX_1(p)_{\mathcal O} \to \cX_0(p)$
provides a tame  strictly
semistable context, in the sense of
Subsection~\ref{subsec:equivariant}. In particular, we can apply
Theorem \ref{thm:arb cohom equivariant} in this setting (of course
taking the group $G$ to be the abelian group $T$), and we see
that each character $\psi_j$ as in the statement of that Theorem is
equal to one of the characters~$\alpha_i$.
}
\end{remark}

\subsection{Canonical local systems}
If $K$ is a sufficiently small compact open subgroup of $G(\mathbb A^{\infty})$,
and $V$ is a continuous representation of $K$ on a finite-dimensional $\Fbar_p$-vector space (this vector
space being equipped with its discrete topology),
then we may associate to $V$ an \'etale local system $\mathcal F_V$
of $\Fbar_p$-vector spaces on $X(K)$ as follows:
Choose an open normal subgroup $K' \subset K$ lying in the kernel of $V$,
and regard $V$ as a representation of the quotient $K/K'$.
Since  $X(K')$ is naturally an \'etale $K/K'$-torsor over $X(K)$,
we may form the \'etale local system of $\Fbar_p$-vector spaces over $X(K)$ associated
to the $K/K'$-representation $V$.  This local system is independent of the choice of $K'$,
up to canonical isomorphism, and we define it to be $\mathcal F_V$.

\begin{defn}
{\em
We refer to the \'etale local systems $\mathcal F_V$ that arise by the preceding
construction as the {\em canonical local systems}
on $X(K)$.  If we may choose $K'$ in the kernel of $V$
to be of level dividing~$N$ (so that in particular $X(K)$ is of level
dividing~$N$),
then we say that $\mathcal F_V$ can be {\em trivialised
at level~$N$}.
}
\end{defn}

\subsection{The Eichler--Shimura relation}

Let $X$ be a $U(n-1,1)$-Shimura variety of level dividing $N$. Let   $w$ be a
place of $F$ such that $l:=w|_\Q$ splits in $F$ and does not divide
$N$. There is a natural action via correspondences on $X$ of Hecke operators $T_{w}^{(i)}$, $0\le
i\le n$, where $T_{w}^{(i)}$ is the double coset operator
corresponding to \[
\begin{pmatrix}
  l1_i&0\\0&1_{n-i}
\end{pmatrix}\times 1\in\GL_n(\Ql)\times\Zl^\times,\] where we use the
assumption that $l\nmid N$ and identify a hyperspecial maximal compact
subgroup of $G(\Ql)$ with $\GL_n(\Zl)\times\Zl^\times$ via an
isomorphism $D_w\cong M_n(\Ql)$. These correspondences then act on the cohomology
$H^j_{\et}(X_{\Qbar},\Fbar_p)$.

More generally if $\mathcal F_V$ is a canonical local system on $X$ that can be trivialised at level $N$,
then we obtain an action of the double coset operators $T_w^{(i)}$ on $\mathcal F_V$,
and hence on the cohomology
$H^j_{\et}(X_{\Qbar},\mathcal F_V)$.

The following theorem regarding this action is then
an immediate consequence of the main result of
\cite{wedhorn-ES} (which proves the Eichler--Shimura relation for PEL
Shimura varieties at places of good reduction at which the group is split).

\begin{thm}\label{thm: Eichler-Shimura}
  Let $X$ be a $U(n-1,1)$-Shimura variety of level dividing~$N$, and $\mathcal F_V$ a canonical
  local system on $X$.
 Let $w$ be a place of $F$ such that $w|_\Q$ splits in $F$ and
  does not divide $Np$. Then $\sum_{i = 0}^n (-1)^{i}(\operatorname{Norm}
  w)^{i(i-1)/2} T_{w}^{(i)} \Frob_w^{n-i}$ acts as $0$ on each
  $H_{\et}^j(X_{\Qbar},\mathcal F_V)$.
\end{thm}

\subsection{Vanishing and torsion-freeness
of cohomology for certain $U(n-1,1)$-Shimura varieties}\label{subsec:
vanishing theorems for U(n-1,1)}
Let $X$ denote a $U(n-1,1)$-Shimura variety as above,
and let $\mathcal F_V$ denote a canonical local system on $X$.
Choose $N$ so that $X$ has level dividing~$N$, so that $\mathcal F_V$
can be trivialised at level~$N$, and so that $p$ divides~$N$. Assume
that the projection of the corresponding level $K$ to
$G(\A^{p,\infty})$ is sufficiently small. 

Let $\T=\Zpbar[T^{(i)}_w]$ be the polynomial ring in the
variables $T^{(i)}_w$, $1\le i\le n$, where $w$ runs over the places
of $F$ such that $w|_\Q$ splits in $F$ and does not divide~$N$.
Let $\mathfrak m$ be a maximal ideal in $\T$ with residue field
$\Fpbar$, and suppose that there exists a continuous irreducible
representation $\rho_{\mathfrak m}:G_{F} \to \GL_n(\Fbar_p)$ which is
unramified at all finite places not dividing $N$, and which satisfies
$\chara\bigl(\rho_\m(\Frob_w)\bigr) \equiv \sum_{i = 0}^n (-1)^{i}(\operatorname{Norm} w)^{i(i-1)/2}
T_{w}^{(i)} X^{n-i} \bmod \mathfrak m$ for all $w \nmid N$ such that
$w|_\Q$ splits in~$F$. Continue to fix a choice of a place $\p$ of $F$ dividing
$p$, and write $\GQp$ for $G_{F_\p}$ from now on. Recall that the choice of $\p$
also gives us an isomorphism $G(\Qp)\cong\GL_n(\Qp)\times\Qptimes$ as
in Section~\ref{subsec:semistable reduction}.

We consider the following further hypothesis on $\rho_{\mathfrak m}$
(this is Hypothesis \ref{hyp:char} below):

\begin{hyp}
\label{hyp:Galois char} 
{\em  If $\theta:G_{F} \to \GL_m(\Fpbar)$ is any  continuous, irreducible
  representation with the property that the characteristic polynomial of
$\rho_\m(g)$ 
annihilates $\theta(g)$ for every $g \in G_{F}$,
then $\theta$ is 
equivalent
to $\rho_\m$.  }
\end{hyp}

We will now prove our first main result, a vanishing theorem for the cohomology of $X$
with $\cF_V$-coefficients.

\begin{theorem}\label{thm:main one}
Suppose that $\rho_{\mathfrak m}$ satisfies Hypothesis~{\em \ref{hyp:Galois char}},
that $\rho_{\m}|_{\GQp}$ is $r$-regular for some
$r \leq (n-1)/2$, that $p>n(n-1)/2+1$ and, if $r = (n-1)/2$,
suppose in addition that
$\rho_{\mathfrak m \, | G_{\mathbb Q_p}}$ contains an irreducible subquotient
of dimension greater than one. Then the localisations $H^i_{\et}(X_\Qbar, \cF_V)_{\mathfrak m}$ vanish
for $i
\leq r$ and for $i\ge 2(n-1)-r$.\end{theorem}

\begin{remark}
{\em In Subsection~\ref{subsec:char} we will show that Hypothesis~\ref{hyp:Galois char}
is satisfied if either $\rho_{\mathfrak m}$ is induced from a character
of $G_K$ for some degree $n$ cyclic Galois extension $K/\mathbb Q$,
or 
if $p\ge n$ and $\SL_n(k) \subseteq \rho_{\mathfrak m}(G_F) \subseteq \Fbar_p^{\times} \GL_n(k)$
for some subfield $k \subset \Fbar_p$.
}
\end{remark}

\begin{remark}
  {\em While we work here with \'etale local systems and \'etale
    cohomology, by virtue of Artin's comparison theorem
    \cite[Exp.\ XI Thm.\ 4.3]{MR0354654} our vanishing results are equivalent to
    vanishing results for the cohomology of the complex
    $U(n-1,1)$-Shimura varieties with coefficients in the corresponding canonical local systems
    for the complex topology.} 
\end{remark}

\noindent {\em Proof of Theorem~\ref{thm:main one}.} Firstly, note that
it suffices to prove vanishing in degree $i\le r$ for all $V$, as
vanishing in degree $i\ge 2(n-1)-r$ then follows by Poincar\'e duality.
(Note that the dual of the canonical local system $\mathcal F_V$ attached
to a representation $V$ is
the canonical local system attached to the contragredient representation
$V^{\vee}$.)

We prove the theorem for $i\le r$ by induction on $i$, the case when
$i < 0$ being trivial.  We begin by reducing to the case when
$\mathcal F_V$ is trivial.  To this end, write $X =X(K)$, let $K'$ be
an open normal subgroup of $K$ of level dividing $N$ such that $V$ is
representation of $K/K'$, and choose an embedding of
$K/K'$-representations $V \hookrightarrow \Fbar_p[K/K']^n$ for
some~$n$; denote the cokernel by $W$. Let $\pi$ denote the projection
$X(K')\to X(K)$. Passing to the associated
canonical local systems, we obtain a short exact sequence $0 \to \cF_V
\to \pi_*\Fbar_p^n \to \cF_W \to 0,$ which gives rise to an exact
sequence of cohomology
$$H^{i-1}_{\et}(X_{\Qbar},\mathcal F_W) \to H^i_{\et}(X_{\Qbar},\mathcal F_V)
\to H^i_{\et}(X_{\Qbar},\pi_*\Fbar_p^n) = H^i_{\et}(X(K')_{\Qbar},\Fpbar)^n.$$
Localizing at $\mathfrak m$ and
applying our inductive hypothesis (with $\mathcal F_W$ in place of $\mathcal F_V$), we reduce to
the case of constant coefficients (with $X(K')$ in place of $X$).  We therefore turn to establishing
the claimed vanishing in this case.

Suppose now that $K'$ is any open normal subgroup of $K$ of level dividing $N$.  
Combining the Hochschild--Serre spectral sequence
$$E_2^{m,n} = H^m_{\et}\bigl(K/K',H^n\bigl(X(K')_{\Qbar},\Fbar_p\bigr)_{\m}\bigr) \implies
H^{m+n}_{\et}\bigl(X(K)_{\Qbar},\Fbar_p\bigr)_{\mathfrak m}$$
with our inductive hypothesis, we find that vanishing of $H^i_{\et}\bigl(X(K'),\Fbar_p\bigr)_{\m}$
implies the vanishing of $H^i_{\et}\bigl(X(K)_{\Qbar},\Fbar_p\bigr)_{\m}$.  Thus,
without loss of generality,
we may and do assume that $K = K_p K^{p}$, 
where $K_p$ is an open normal subgroup of $I_1$,
and $K^p$ is a sufficiently small compact open subgroup
of $G(\A^{\infty,p})$. 

We again consider a Hochschild--Serre spectral sequence, this time
the one relating the cohomology of $X(K)$ and $X(I_1 K^p)$,
which takes the form
$$
E_2^{m,n} = H^m\bigl(I_1/K_p,H^n_{\et}\bigl(X(K)_{\Qbar},\Fbar_p\bigr)_{\m}\bigr) \implies
H^{m+n}_{\et}\bigl(X(I_1 K^p)_{\Qbar},\Fbar_p\bigr)_{\mathfrak m}.$$
Once more taking into account our inductive hypothesis,
we obtain an isomorphism
$H^i_{\et}\bigl(X(I_1 K^p)_{\Qbar},\Fbar_p\bigr)_{\mathfrak m} \iso
H^i_{\et}\bigl(X(K)_{\Qbar},\Fbar_p\bigr)_{\mathfrak m}^{I_1/K_p}.$
Since $I_1/K_p$ is a $p$-group, while $H^i_{\et}\bigl(X(K)_{\Qbar},\Fbar_p\bigr)_{\mathfrak m}$
is a vector space over a field of characteristic $p$,
we see that the latter space vanishes if and only if its space
of $I_1/K_p$-invariants does.  Thus we are reduced to establishing the theorem in the
case when $K = I_1 K^p$.

Now recall that $I_1^* = \F_p^{\times} \times I_1,$ and that the projection onto the first factor
arises from the similitude projection $GU(n-1,1) \to \G_m$.  From this it follows that 
$X(I_1 K^p)$ is isomorphic to the product $X(I_1^*K^p)\times_F \Spec A,$ where 
$A := F[x]/\Phi_p(x)$ (where $\Phi_p(x)$ denotes the $p$th cyclotomic polynomial;
the action of $\F_p^{\times} = I_1^*/I_1$ on $X(I_1 K^p)$ is induced by the action
of $\F_p^{\times}$ on $A$ given by $x \mapsto x^a$, for $a \in \Fp^{\times}$).
Consequently there is an isomorphism of Galois representations
$$H^i\bigl(X(I_1 K^p)_{\Qbar},\Fbar_p\bigr)_{\mathfrak m} \iso 
\oplus_{j = 0}^{p-2} H^i\bigl(X(I_1^* K^p)_{\Qbar},\Fbar_p\bigr)_{\mathfrak m_j} \otimes \omega^j,$$
where $\mathfrak m_j$ is the maximal ideal in $\mathbb T$
which corresponds to the twisted Galois representation $\rho_{\mathfrak m_j} :=
\rho_{\mathfrak m} \otimes \omega^{-j}.$
Since each of the Galois representations $\rho_{\mathfrak m_j}$ satisfies 
the hypotheses of the theorem (as these hypotheses are invariant under 
twisting),
we are reduced to proving the theorem in the case of $X(I_1^* K^p)$.

Consider an irreducible $G_F\times I_0/I_1^*$-subrepresentation of
$H_{\et}^i\bigl(X(I_1^* K^p)_\Qbar,\Fpbar\bigr)_\m$, which we may write in 
the form $\theta\otimes \beta$,
where $\theta$ is an absolutely irreducible $G_F$-representation and
$\beta$ is a character of the abelian group $T(\F_p) = I_0/I_1^*$.
Theorem~\ref{thm: Eichler-Shimura} and Hypothesis~\ref{hyp:Galois char}
taken together imply that $\theta$ is equivalent to $\rho_\m$.
We consider $\beta$ as an
$n$-tuple of characters $\beta_1,\dots,\beta_n$ of $\Fptimes$. Write
$\tilde{\beta}_j$ for the Teichm\"uller lift of $\beta_j$, and $\chi_j$ for 
the character of $\IQp$ given by $\beta_j$. By Lemma~\ref{lem:semistable model}, Remark~\ref{rem:shimura gives a strictly semistable context} and Theorem~\ref{thm:arb cohom
  equivariant}, we see that $\theta$ can be embedded as the
reduction mod $p$ of a potentially semistable representation with
Hodge--Tate weights in the range $[-i,0]$, whose inertial type is a direct
sum of characters
belonging to the collection~$\{\tilde{\beta}_j\}$. (To see this, note
that as $\theta$ is in the $\beta$-part of the cohomology, it
necessarily occurs in the reduction of the $\tilde{\beta}$-part of the
$G_F\times T$-representation provided by Theorem~\ref{thm:arb cohom
  equivariant}.)
 We claim that
$\det\theta|_{\IQp}=\det\rho_\m|_{\IQp}=\chibar_1\cdots\chibar_n\omega^{-n(n-1)/2}$.
Admitting this for the moment, we may apply
Theorem~\ref{thm: generic vanishing in small HT weights} to the
representation $\rho_\m|_{\GQp}^\vee$, and we deduce that
$\rho_\m|_{\GQp}^\vee$ is not $r$-regular. Equivalently, we see that
$\rho_\m|_{\GQp}$ is not $r$-regular, which contradicts our assumptions.

It remains to establish the equality
$\det\theta|_{\IQp}=\det\rho_\m|_{\IQp}=\chibar_1\cdots\chibar_n\omega^{-n(n-1)/2}$. To see this, note
that Theorem~\ref{thm: Eichler-Shimura} implies that for each place
$w\nmid N$ of $F$ such that $w|_\Q$ splits in $F$, we have
$\det\rho_\m(\Frob_w)=(\operatorname{Norm}
w)^{n(n-1)/2}T_w^{(n)}$. Let $\psi_\m$ be the character of
$\A_F^\times/F^\times$ corresponding to the character $\omega^{n(n-1)/2}\det\rho_\m$ by global class
field theory; then we need to prove that $\psi_\m|_{\Z_p^\times}=\beta_1\cdots\beta_n$. The centre of $G(\A_\Q)$ is $\A_F^\times$, so by the definition of
the Shimura variety $X$ there is a natural action of
$\A_F^\times/F^\times$ on
$H_{\et}^i\bigl(X(I_1^* K^p)_\Qbar,\Fpbar\bigr)_\m$.
By the definition of the Hecke operators, we see that if $\varpi_w$ is  a
uniformiser at a place
$w\nmid N$ of $F$ such that $w|_\Q$ splits in $F$, then $\varpi_w$
acts as $T^{(n)}_w$. By the Chebotarev density theorem, we deduce that
the action of $\A_F^\times/F^\times$ on the underlying vector space of $\theta$
(which by definition is a subspace of $H_{\et}^i\bigl(X(I_1^*
K^p)_\Qbar,\Fpbar\bigr)_\m$) is via the character $\psi_\m$. In order to
compute $\psi_\m|_{\Z_p^\times}$, it is thus sufficient to compute the
action of $\Z_p^\times$ on $H_{\et}^i\bigl(X(I_1^*
K^p)_\Qbar,\Fpbar\bigr)_\m$, and in particular sufficient to compute the
action of the Iwahori subgroup $I_0$. Now, since
$\theta$ is assumed to be in the $\beta$-part of the cohomology, $I_0$
acts via the character $\beta$ of $I_0/I_1^*$,
so that $\psi_\m|_{\Z_p^\times}=\beta_1\cdots\beta_n$, as required.
 \qed

\begin{cor}
\label{cor:main one}
Suppose that $\rho_{\mathfrak m}$ satisfies Hypothesis~{\em
  \ref{hyp:Galois char}}, that $\rho_{\mathfrak m\, | G_{\mathbb
    Q_p}}$ is $r$-regular for some $r \leq \min\{ (n-1)/2, p-2\}$,
and, if $r = (n-1)/2$, suppose in addition that $\rho_{\mathfrak m \,
  | G_{\mathbb Q_p}}$ contains an irreducible subquotient of dimension
greater than one.  Then the localisation $H_{\et}^i(X_\Qbar, \Zbar_p)_{\mathfrak
  m}$ vanishes for $i \leq r,$ while $H_{\et}^{r+1}(X_\Qbar,\Zbar_p)_{\mathfrak
  m}$ is torsion-free.
\end{cor}
\begin{proof}This follows at once from Theorem \ref{thm:main one} and
  the short exact sequence
\begin{multline*}
0 \to H_{\et}^i(X_\Qbar,\Zpbar)_\m/\m_{\Zpbar} H_{\et}^i(X_\Qbar,\Zpbar)_\m
\to H_{\et}^i(X_\Qbar,\Fpbar)_\m 
\\
\to H_{\et}^{i+1}(X_\Qbar,\Zpbar)_\m
[\m_{\Zpbar}] \to 0.
\qedhere
\end{multline*}
\end{proof}

\begin{remark}
{\em As already remarked in the introduction, we expect some kind
of mod $p$ analogue of Arthur's conjectures to hold, and so in particular,
we expect that stronger results than Theorem~\ref{thm:main one}
and Corollary~\ref{cor:main one} should hold.  In particular,
if $\mathfrak m$
is any maximal ideal in the Hecke algebra
attached to an irreducible continuous representation $\rho_{\mathfrak m}: G_F
\to \GL_n(\Fbar_p),$
then we expect that the localizations $H^i_{\et}(X_{\Qbar},\cF_V)_{\mathfrak m}$
and $H_{\et}^i(X_{\Qbar},\Zbar_p)_{\mathfrak m}$ should vanish
in degrees $i < n-1$.

On the other hand, it need not be the case that (for example) $H^i_{\et}(X_{\Qbar},\Fbar_p)$
vanishes in all degrees in which $H^i_{\et}(X_{\Qbar},\Qbar_p)$ vanishes.
For example, in the case $n = 3$, for the unitary Shimura varieties that we consider
here, namely those that are associated to division algebras, it is known that
$H^1_{\et}(X_{\Qbar},\Qbar_p) = 0$ \cite[Thm.~15.3.1]{Rog90}. (Under additional
restrictions on the division algebra that is allowed, an analogous result
is known for all values of $n$ \cite[Thm.~3.4]{Clo93}.)   On the other hand,
one can construct examples for which $H^1_{\et}(X_{\Qbar},\Fbar_p) \neq 0,$
and hence for which $H^2_{\et}(X_{\Qbar},\Zbar_p)$ is not torsion-free,
via congruence cohomology.  (See e.g.\ the proof of \cite[Thm.~3.4]{Suh08}.)
The existence of such classes does not contradict our theorems or 
expectations, since congruence cohomology is necessarily Eisenstein (i.e.\
gives rise to Eisenstein systems of Hecke eigenvalues, in the sense
that the associated Galois representation $\rho_\m$ is completely reducible).
}
\end{remark}

\subsection{On the mod $p$ cohomology of certain $U(2,1)$-Shimura varieties}
Let $X := X(K)$ denote a $U(2,1)$-Shimura variety, with $K$ of level dividing $N$,
for some natural number $N$ divisible by $p$, and such that the projection of $K$ to
$\mathbb G(\mathbb A^{p,\infty})$ is sufficiently small.  Let
$\mathcal F_V$ be a canonical local system on $X$, which may be trivialized
at level $N$. 
The results of Section \ref{subsec: vanishing theorems for U(n-1,1)}
are particularly powerful in this case, as we now demonstrate.

\begin{cor}
  \label{cor: main vanishing for U(2,1)}Suppose that $\rho_{\mathfrak
    m}$ satisfies Hypothesis~{\em \ref{hyp:Galois char}}, that
  $\rho_{\m}|_{\GQp}$ is $1$-regular, and that
  $\rho_{\mathfrak m \, | G_{\mathbb Q_p}}$ contains an irreducible
  subquotient of dimension greater than
  one. Then the localisations $H^i_{\et}(X_\Qbar, \cF_V)_{\mathfrak m}$
  vanish for $i\ne 2$.
\end{cor}
\begin{proof}
  This follows immediately from Theorem \ref{thm:main one}, noting
  that the hypothesis that  $\rho_{\m}|_{\GQp}$
  is $1$-regular implies that $p>4$ (indeed, that $p \geq 11$).
\end{proof}

We now prove a result which does not require the existence of a Galois
representation $\rho_\m$. We begin with a lemma.

\begin{lem}
  \label{lem: Galois action on H^0}If $\m$ is a maximal
  ideal of $\T$ with residue field $\Fpbar$, such that
  $H^0_\et(X_\Qbar,\cF_V)_\m\ne 0$, then there is an abelian representation
  $\rhobar_0:G_F\to\GL_3(\Fpbar)$ such that $\chara\bigl(\rhobar_0(\Frob_w)\bigr) = \sum_{i = 0}^3 (-1)^{i}(\operatorname{Norm} w)^{i(i-1)/2}
T_{w}^{(i)} X^{n-i}\pmod{\m}$ for all split places $w$ of $E$ for which $w \nmid N p$.
\end{lem}
\begin{proof}
  This is standard and follows for example from Section 2.1 of \cite{MR546620}.
\end{proof}

\begin{theorem}
\label{thm:main two}
If $\rho$ is a three-dimensional irreducible sub-$G_F$-representation 
of the \'etale cohomology group $H^1_{\et}(X_{\Qbar},\cF_V)$,
then either every irreducible subquotient
of $\rho_{| G_{\mathbb Q_p}}$ is one-dimensional, or else 
$\rho_{| G_{\mathbb Q_p}}$ is not $1$-regular, or else $\rho(G_F)$ is not generated
by its subset of regular elements.
\end{theorem}
\begin{remark}
{\em  Recall that a square matrix is said to be regular if its minimal and
characteristic polynomials coincide.
In Subsection~\ref{subsec:reg} we will show that $\rho(G_F)$ is generated by its subset
of regular elements if either $\rho_{\mathfrak m}$ is induced from a character
of $G_K$ for some cubic Galois extension  $K/\mathbb Q$, or if
$\rho(G_F)$ contains a regular unipotent element.
}
\end{remark}
\begin{remark}
{\em In the proof of the theorem we use some of the results of
  Section~\ref{sec:group} below.
}
\end{remark}

\noindent {\em Proof of Theorem~\ref{thm:main two}.}
The argument follows similar lines to the proof of Theorem~\ref{thm:main one},
although it is slightly more involved, since we are not giving ourselves
the existence of the Galois representation $\rho_{\m}$.\footnote{In fact,
as noted in the introduction, recent work of Scholze \cite{1306.2070}
implies that $\rho_{\m}$ exists for any maximal ideal
$\m$ in the Hecke algebra.  We have left our argument as originally
written.}  The key point 
will be that in the Hochschild--Serre spectral sequences that appear,
the only other cohomology to contribute besides $H^1$ will be $H^0$,
and for maximal ideals of $\T$ in the support of $H^0$, we do have associated 
Galois representations, by Lemma~\ref{lem: Galois action on H^0}.

We first show that if
$H^0_{\et}\bigl(X_{\Qbar},\mathcal F_W\bigr)_{\mathfrak m} \neq 0$ for some
canonical local system $\mathcal F_W$ on $X$ and some 
maximal ideal $\mathfrak m$ of $\T$,
then $\Hom_{G_F}\Bigl(\rho,H^1_{\et}\bigl(X_{\Qbar},\mathcal F_V\bigr)[\mathfrak m]\Bigr)
= 0$.  To see this,
note that if
$H^0_{\et}\bigl(X_{\Qbar},\mathcal F_W\bigr)_{\mathfrak m} \neq 0$
and $\Hom_{G_F}\Bigl(\rho,H^1_{\et}\bigl(X_{\Qbar},\mathcal F_V\bigr)[\mathfrak m]\Bigr) \neq 0,$
then Lemma~\ref{lem: Galois action on H^0} and Theorem~\ref{thm: Eichler-Shimura} together
imply that there exists an
abelian representation $\rhobar_0:G_F\to\GL_3(\Fpbar)$
such that for all $g\in G_F$, the characteristic polynomial of $\rhobar_0(g)$
annihilates $\rho(g)$. By Lemma \ref{lem:kernels} this implies that
$\rho$ is abelian, which is impossible as $\rho$ is
irreducible.

Now $H^1_{\et}(X_{\Qbar},\cF_V)$ is the direct sum of its localisations 
at the various maximal ideals $\m$ of $\T$, and hence, since $\Hom_{G_F}\bigl(\rho,H^1_{\et}(X,\cF_V)\bigr) \neq 0$
by hypothesis, we see that $\Hom_{G_F}\bigl(\rho,H^1_{\et}(X,\cF_V)_{\m}\bigr) \neq 0$ for some
maximal ideal $\m$ of $\T$.  Since this is a finite length $\T_{\m}$-module, we see that its $\T_{\m}$-socle
$\Hom_{G_F}\bigl(\rho,H^1_{\et}(X,\cF_V)[\m]\bigr)$ must also be non-zero, and hence, by the preceding paragraph,
we conclude that $H^0_{\et}(X,\cF_W)_{\m} = 0$ for any canonical local system $\cF_W$ on $X$.

As in the proof of Theorem~\ref{thm:main one},
choose a short exact sequence of canonical local systems
$0 \to \cF_V \to \pi_* \Fbar_p^n \to \cF_W \to 0,$
for some $\pi:X '\to X$.  Passing to the the long exact sequence
$$ H^0_{\et}(X,\cF_W)_{\m} \to H^1_{\et}(X,\cF_V)_{\m} \to H^1_{\et}(X,\pi_* \Fbar_p)^n = H^1_{\et}(X',\Fbar_p)^n,$$
and using the result of the preceding paragraph, namely that $H^0_{\et}(X,\cF_W)_{\m} = 0$,
we conclude that $\rho$ embeds into $H^1_{\et}(X',\Fbar_p)$.  Thus, replacing $X$ by $X'$,
we reduce to the case when $\cF_V$ is constant, which we assume from now on.

We now suppose that
$\rho$ is a subrepresentation of $H^1_{\et}(X_{\Qbar},\Fbar_p)_\m$.
We will prove that $\rho$ is then necessarily a subrepresentation
of $H^1_{\et}\bigl(X(K^p I_1), \Fbar_p\bigr)_{\m}$, for some
sufficiently small open subgroup $K^p$ of $G(\mathbb A^{\infty,p})$.
The result will then follow from Lemmas~\ref{lem:semistable model}
and~\ref{lem:regular implies equal dets, starting with char polys
    rather than a repn}, and Theorems~\ref{thm:arb cohom equivariant},
\ref{thm: Eichler-Shimura},
and~\ref{thm: generic vanishing in small HT weights}.

As in the proof of Theorem~\ref{thm:main one},
we write $X = X(K)$, and choose a normal open subgroup $K' := K^p K_p$ of $K$,
with $K_p \subset I_1$.
The Hochschild--Serre
spectral sequence associated to the cover $X(K') \to X(K)$
gives rise to an exact sequence
\begin{multline*}
0 \to
H^1\bigl(K/K', H^0\bigl(X(K'),\Fbar_p\bigr)_{\m}\bigr)
\to H^1_{\et}\bigl( X(K), \Fbar_p\bigr)_{\m}
\to H^1_{\et}\bigl(X(K'),\Fbar_p\bigr)^{K/K'}_{\m}  \\
\to 
H^2\bigl(K/K', H^0\bigl(X(K'),\Fbar_p\bigr)_{\m}\bigr).
\end{multline*}
The same argument as above, using 
Lemma~\ref{lem: Galois action on H^0} (applied now with $X(K')$ in place of $X$),
Theorem~\ref{thm: Eichler-Shimura}, and Lemma~\ref{lem:kernels} below,
shows that
$$H^1\bigl(K/K', H^0\bigl(X(K'),\Fbar_p\bigr)_{\m}\bigr) = 
H^2\bigl(K/K', H^0\bigl(X(K'),\Fbar_p\bigr)_{\m}\bigr) = 0.$$
Thus in fact we have an isomorphism
$H^1_{\et}\bigl(X(K),\Fbar_p\bigr) \iso
H^1_{\et}\bigl(X(K'),\Fbar_p\bigr)^{K/K'}_{\m},$
and hence an isomorphism
\begin{multline*}
\Hom_{G_F}\Bigl(\rho,H^1_{\et}\bigl(X(K),\Fbar_p\bigr)\Bigr) \iso
\Hom_{G_F}\Bigl(\rho,H^1_{\et}\bigl(X(K'),\Fbar_p\bigr)^{K/K'}_{\m}\Bigr)
\\
=
\Hom_{G_F}\Bigl(\rho,H^1_{\et}\bigl(X(K'),\Fbar_p\bigr)_{\m}\Bigr)^{K/K'}.
\end{multline*}
In particular, if $\rho$ appears as a subrepresentation of
$H^1_{\et}\bigl(X(K),\Fbar_p)_{\m}$,
then it appears as a subrepresentation of $H^1_{\et}\bigl(X(K'),\Fbar_p)_{\m}.$

Now considering the Hochschild--Serre spectral sequence for the cover $X(K') \to X(K^p I_1)$,
and using the fact that if 
$\Hom_{G_F}\Bigl(\rho,H^1_{\et}\bigl(X(K'),\Fbar_p\bigr)_{\m}\Bigr)\neq 0,$
then also
$\Hom_{G_F}\Bigl(\rho,H^1_{\et}\bigl(X(K'),\Fbar_p\bigr)_{\m}\Bigr)^{I_1/K_p} \neq 0$
(since $I_1/K_p$ is a $p$-group), we conclude that if $\rho$ appears as a subrepresentation
of 
$H^1_{\et}\bigl(X(K'),\Fbar_p)_{\m}$,
then it appears as a subrepresentation of $H^1_{\et}\bigl(X(K^p I_1),\Fbar_p)_{\m}.$

Arguing exactly as in the proof of Theorem~\ref{thm:main one}, we then
deduce that some twist of $\rho$ appears in $H^1_{\et}\bigl(X(K^p
I_1^*),\Fbar_p)$, and so, replacing $\rho$ by this twist, it suffices
to prove that if $\rho$ is an irreducible three-dimensional
representation $\rho$ of $H^1_{\et}\bigl(X(K^p I_1^*),\Fbar_p)$ that
is generated by its regular elements, then either every irreducible
subquotient of $\rho_{| G_{\mathbb Q_p}}$ is one-dimensional, or else
$\rho_{| G_{\mathbb Q_p}}$ is not $1$-regular.  This follows from
Lemma~\ref{lem:semistable model} and
Theorems~\ref{thm:arb cohom equivariant},
\ref{thm: Eichler-Shimura}, and~\ref{thm: generic
  vanishing in small HT weights} exactly as in the proof of Theorem
\ref{thm:main one}, replacing the appeal to Hypothesis \ref{hyp:Galois
  char} with one to Lemma \ref{lem:regular implies equal dets,
  starting with char polys rather than a repn} below.\qed

\medskip

Our other main theorem concerns the weight part of the Serre-type
conjecture of \cite{bib:herzig-thesis} for $U(2,1)$.
It is proved by combining our techniques with
those of \cite{EGH}, where a similar theorem is proved for $U(3)$
(which is simpler, because one has vanishing of cohomology outside of
degree $0$). We begin by recalling some terminology from
\cite{EGH}. We will call an irreducible $\Fpbar$-representation of
$\GL_3(\Fp)$ a {\em Serre weight}. Fix an irreducible representation
$\rhobar:G_F\to\GL_3(\Fpbar)$. Let $X := X(K)$ be a $U(2,1)$-Shimura variety
such that $K$ is of level dividing $N$ and has sufficiently small
projection to $\G(\A^{p,\infty})$,
where now we assume that $(N,p)=1$.  Assume furthermore that $\rhobar$ is
unramified at all places not dividing $Np$, and define a maximal
ideal $\m$ of $\T$ with residue field $\Fpbar$ by demanding that for
each place $w\nmid Np$ of $F$ such that $w|_\Q$ splits in $F$, the
characteristic polynomial of $\rhobar(\Frob_w)$ is equal to the
reduction modulo $\m$ of $\sum_{i = 0}^3
(-1)^{i}(\operatorname{Norm} w)^{i(i-1)/2} T_{w}^{(i)} X^{n-i}$.

Let $V$ be a Serre weight; since $(N,p)$=1, we may write $K=K_pK^p$ where $K_p\subset
G(\Qp)\cong\GL_3(\Qp)\times\Qptimes$ is conjugate to
$\GL_3(\Zp)\times\Zptimes$, and we may regard $V$ as a representation
of $K_p$ via the projection $\GL_3(\Zp)\onto\GL_3(\Fp)$. As usual, write $\cF_V$
for the canonical local system associated to $V$. We say that
$\rhobar$ is {\em modular of weight $V$} if for some $N$, $X$ as above
and for some $0\le i\le 4$ we
have \[H_{\et}^i(X_\Qbar,\cF_{V})_\m\ne 0.\] 

Assume now that $\rhobar|_{\GQp}$ is irreducible. Definition 6.2.2 of
\cite{EGH} defines what it means for a Serre weight to be (strongly)
generic, and Section~5.1 of \cite{EGH} (using the recipe of
\cite{bib:herzig-thesis}) defines a set $W^?(\rhobar)$ of Serre
weights in which it is predicted that $\rhobar$ is modular. Let
$\Wgen(\rhobar)$ be the set of generic weights for which $\rhobar$ is
modular.

\begin{theorem}\label{thm:main EGH}Suppose that $\rhobar$ satisfies Hypothesis
\ref{hyp:char} below, and that $\rhobar|_{\GQp}$ is irreducible and
$1$-regular. Suppose that $\rhobar$ is modular of some strongly
generic weight. Then $\Wgen(\rhobar)=W^?(\rhobar)$. In fact, for each
$V\in\Wgen(\rhobar)$, we have  \[H_{\et}^i(X_\Qbar,\cF_{V})_\m\ne 0\]
if and only if $i=2$ and $V\in W^?(\rhobar)$.\end{theorem}
\begin{proof} By the definition of $\m$, the representation $\rhobar$
  satisfies the defining properties of the representation $\rho_\m$
  considered in Section \ref{subsec: vanishing theorems for
    U(n-1,1)}. Applying Corollary \ref{cor: main vanishing for
    U(2,1)}, we see that for any Serre weight $V$, we have
  $H_{\et}^i(X_\Qbar,\cF_V)_\m=0$ if $i\ne 2$. We will now deduce the
  result from Theorem 6.2.3 of \cite{EGH} (taking
$\rbar$ there to be our~$\rhobar$). By Theorem 4.3.3 of
  \cite{EGH}, we see that it suffices to show that we can define $S$
  and $\tS$ as in Section~4 of \cite{EGH} so that Axioms
  \~{A}1--\~{A}3 of Section~4.3 of \cite{EGH} are satisfied. Following
  \cite{EGH}, we define $S$ and $\tS$ using completed cohomology in
  the sense of \cite{MR2207783} (in \cite{EGH} the use of completed
  cohomology was somewhat disguised, but the constructions with
  algebraic modular forms in \cite{EGH} are equivalent to the use of
  completed cohomology of $U(3)$ in degree $0$). In fact, given our
  vanishing results the verification of the axioms of \cite{EGH}
  is very similar to that carried out for $U(3)$ in \cite{EGH}, and we
  content ourselves with sketching the arguments. 

From now on we regard
  the prime-to-$p$ level structure $K^p$ of $X$ as fixed, and we will
  vary $K_p$ in our arguments. We will write $K_p(0)$ for
  $\GL_3(\Zp)\times\Z_p^\times\subset G(\Qp)$. We fix a sufficiently large extension $E/\Qp$ with ring
  of integers~$\cO_E$, residue field~$k_E$, and uniformiser~$\varpi_E$, and we
  define \[S:=\varinjlim_{K_p}H^2_\et(X(K^pK_p)_\Qbar,\Fpbar)_\m,\] \[\tS:=((\varprojlim_s\varinjlim_{K_p}H^2_\et(X(K^pK_p)_\Qbar,\cO_E/\varpi_E^s)_\m)\otimes_{\cO_E}\Zpbar)^\lalg\]
  (that is, the locally algebraic vectors in the localisation at $\m$
  of the completed cohomology of degree $2$). Using the
  Hochschild--Serre spectral sequence and the vanishing of
  $H_{\et}^i(X_\Qbar,\cF_V)_\m=0$ if $i\ne 2$, it is straightforward
  to verify Axioms \~{A}1--\~{A}3 of Section~4.3 of \cite{EGH}, as we
  now explain. 

Firstly, we need to check that our definition of
  ``modular'' is consistent with that of Definition 4.2.2 of
  \cite{EGH}. This amounts to showing that for any Serre weight~$V$,
  \[(S\otimes_\Fpbar
  V)^{K_p(0)}=H^2_\et(X(K_p(0)K^p)_\Qbar,\cF_{V})_\m.\] To see this, note
  that since any sufficiently small $K_p$ acts trivially on $V$, we have \begin{align*}S\otimes_\Fpbar
    V&=\varinjlim_{K_p}H^2_\et(X(K^pK_p)_\Qbar,\Fpbar)_\m\otimes V\\ &\isoto
    \varinjlim_{K_p}H^2_\et(X(K^pK_p)_\Qbar,\cF_V)_\m,\end{align*} so it is
  enough to check that for all compact open subgroups $K_p\subset K_p(0)$ we have
  \[H^2_\et(X(K^pK_p)_\Qbar,\cF_V)_\m^{K_p(0)}=H^2_\et(X(K^pK_p(0))_\Qbar,\cF_V)_\m,\]
  which is an easy consequence of the Hochschild--Serre spectral sequence and our vanishing
  result. We also need an embedding $S\into\tS\otimes_\Zpbar\Fpbar$
  which is compatible with the actions of $\GL_3(\Qp)$ and the Hecke
  algebra. In fact, it is easy to see that we have
  $\tS\otimes_\Zpbar\Fpbar=S$. For example, there is a natural isomorphism
  $$H_{\et}^3(X_\Qbar,\cO_E)_\m/\varpi_EH_{\et}^3(X_\Qbar,\cO_E)_\m =
  H_{\et}^3(X_\Qbar,k_E)_\m,$$
and hence the vanishing of
  $H_{\et}^3(X_\Qbar,k_E)_\m$ implies that of
the finitely generated $\cO_E$-module
  $H_{\et}^3(X_\Qbar,\cO_E)_\m$. One then sees that for
  all $s$ we have a natural isomorphism
  $$H_{\et}^2(X_\Qbar,\cO_E)_\m/\varpi_E^sH_{\et}^2(X_\Qbar,\cO_E)_\m
=H_{\et}^2(X_\Qbar,\cO_E/\varpi_E^s)_\m,$$
  from which the claim follows easily.

 We now examine Axiom \~{A}1. We must show that if $\tV$ is a finite
 free $\Zpbar$-module with a locally algebraic action of $K_p(0)$
 (acting through $\GL_3(\Zp)$),
 then $(\tS\otimes_\Zpbar\tV)^{K_p(0)}$ is a finite free
 $\Zpbar$-module, and for $A=\Qpbar$, $\Fpbar$
 we have \[(\tS\otimes_\Zpbar\tV)^{K_p(0)}\otimes_\Zpbar
 A=(\tS\otimes_\Zpbar\tV\otimes_\Zpbar A)^{K_p(0)}.\] This is
 straightforward, the key point being that if $\cF_{\tV}$ denotes the
 lisse \'etale sheaf attached to $\tV$, then a straightforward
 argument with Hochschild--Serre as above
 gives \[(\tS\otimes_\Zpbar\tV)^{K_p(0)}=H^2_\et(X(K^pK_p(0))_\Qbar,\cF_{\tV})_\m,\]which
 is certainly a finite free $\Zpbar$-module (it is torsion-free by the
 proof of Corollary~\ref{cor:main one}).

The verification of Axioms \~{A}2 and \~{A}3 
is now exactly the same as in
Proposition 7.4.4 of \cite{EGH}, as the Galois representations
occurring in the localised cohomology module
$H^2_\et(X(K^pK_p(0))_\Qbar,\cF_{\tV})_\m$ are associated
to automorphic forms exactly as in \cite{EGH}.\footnote{In the
  interests of full disclosure, we are not aware of a reference in the
literature giving the precise base change result from $U(2,1)$ to
$\GL_3$ that we need, but it seems to be well known to the experts,
and will follow from the much more general work in progress of Mok and
Kaletha--Minguez--Shin--White.} (In fact,
at least for Axiom \~{A}2 this is a rather roundabout way of
proceeding, as the Galois representations in question are constructed
in \cite{ht} by using $H^2_\et(X(K^pK_p(0))_\Qbar,\cF_{\tV})$, and one
can read off the required properties directly from the comparison
theorems of $p$-adic Hodge theory. For Axiom \~{A}3 we are not aware
of any comparison theorems in sufficient generality, so it is
necessary at present to take a lengthier route through the theory of
automorphic forms.)
\end{proof}

\section{Group theory lemmas}
\label{sec:group}
The theorems of Section~\ref{sec:Shimura} contain certain hypotheses on the
Galois representations involved.  Our goal in 
this section is to establish some group-theoretic lemmas which give
sufficient criteria for 
these hypotheses to be satisfied.
Throughout the section $G$ is a finite group, and $k$ is an algebraically closed field of
characteristic~$p$.
For any square matrix $A$ with entries in~$k$,
we write $\chara(A)$ to denote the characteristic polynomial of~$A$.

\subsection{Characterising representations by their characteristic polynomials}
\label{subsec:char}
Let $\rho:G \to \GL_n(k)$ be an irreducible representation.  In this
subsection we establish some criteria for $\rho$ to satisfy the
following hypothesis:

\begin{hyp}
\label{hyp:char}
{\em If $\theta:G \to \GL_m(k)$ is irreducible,
and if
$\chara\bigl(\rho(g)\bigr)$ 
annihilates $\theta(g)$ for every $g \in G$,
then $\theta$ is 
equivalent to $\rho$.
}
\end{hyp}

\begin{remark}
{\em Any irreducible $\rho$ of dimension $2$ satisfies Hypothesis~\ref{hyp:char},
as was proved by Mazur (see the proof of Proposition 14.2 of \cite{MR488287}).  However, it is not satisfied in general if the
dimension $n$ of $\rho$ is greater than $2$ (for instance, this
already fails if $\rho$ is the irreducible 3-dimensional
representation of $A_4$, cf.\ Section~5 of \cite{MR1094193} as well as Remark~\ref{rem:A4} below).
}
\end{remark}

\begin{lemma}
\label{lem:kernels}
Let $\rho:G \to \GL_n(k)$ and $\theta:G \to GL_m(k)$ be two
representations.  If
$\theta$ is irreducible,
and if
$\chara\bigl(\rho(g)\bigr)$ 
annihilates $\theta(g)$ for every $g \in G$,
then the kernel of $\theta$ contains the kernel of $\rho$.
\end{lemma}
\begin{proof}
If $\rho(g)$ is trivial, then the assumption implies that every eigenvalue of $\theta(g)$
is equal to $1$, and hence that $\theta(g)$ is unipotent, and so of order a power of $p$.
Thus the image of $\ker(\rho)$ under $\theta$ is a normal subgroup $H$ of $\theta(G)$
of $p$-power order, and
we see that the space of invariants $(k^m)^H$ is a non-trivial subspace of $k^m$.
Since $H$ is normal in $\theta(G)$,
we see that $\theta(G)$ leaves $(k^m)^H$ invariant, and hence,
since $\theta$ is assumed to be irreducible, we see that in fact
$(k^m)^H = k^m$.  Thus $H$ is trivial, which is to say that $\ker(\rho)\subset
\ker(\theta)$, as claimed.
\end{proof}

\begin{lemma}
\label{lem:characters}
If $\rho:G \to \GL_n(k)$ is a direct sum of one-dimensional characters of $G$, and if
$\theta(g): G \to \GL_m(k)$ is an irreducible representation of $G$
such that 
$\chara\bigl(\rho(g)\bigr)$ 
annihilates $\theta(g)$ for every $g \in G$,
then $m = 1$, so that $\theta$ is a character, and every element of $G$
lies in the kernel of at least one of the summands of $\rho\otimes\theta^{-1}$. 
\end{lemma}
\begin{proof}
Since $\rho$ is a direct sum of characters, it factors through $G^{\ab}$.
Lemma~\ref{lem:kernels} then shows that $\theta$ also factors through $G^{\ab}$.
Since $\theta$ is also assumed to be irreducible, we find that $\theta$ must be
a character. 
Twisting by $\rho$ by $\theta^{-1}$, we may in fact assume
that $\theta$ is trivial, and writing $\rho = \chi_1\oplus \cdots \oplus \chi_n,$
we find that for each $g \in G$, the value $\chi_i(g)$ is equal to $1$ for
at least one value of $i$ (since $\chara\bigl(\rho(g)\bigr) = \bigl(X-\chi_1(g)\bigr)
\cdots \bigl(X -\chi_n(g)\bigr)$ annihilates $\theta(g) = 1$).  Thus
$G$ is equal to the union of its subgroups $\ker(\chi_i)$.
\end{proof}

\begin{remark}
{\em  In the context of the preceding proposition, we can't conclude in general
that $\theta$ coincides with one the summands of $\rho$.  E.g.\ if $G$ denotes
the Klein four group, if $p$ is odd, and if $\rho$ denotes the three-dimensional
representation obtained by taking the direct sum of the three non-trivial characters
of $G$, then taking $\theta$ to be the trivial representation,
the hypotheses of the proposition are satisfied, but $\theta$ is certainly
not one of the summands of $\rho$.
}
\end{remark}

\begin{lemma}
\label{lem:induction}
If $\rho:G \to \GL_n(k)$ is irreducible, and is isomorphic to an 
induction $\Ind_H^G \psi,$ where $H$ is a cyclic normal subgroup of $G$ of index $n$
and $\psi: H \to k^{\times}$ is a character,
then $\rho$ satisfies Hypothesis~{\em \ref{hyp:char}}.
\end{lemma}
\begin{proof}
The restriction $\rho|_H$ is isomorphic to the direct sum $\oplus_{g \in G/H} \psi^g.$
If we let $\theta'$ be a Jordan--H\"older constituent of the restriction $\theta_{| H}$,
then Lemma~\ref{lem:characters} (applied to the representations $\rho_{|H}$ and
$\theta'$ of $H$) implies that $\theta'$ is a character of $H$ and (because $H$
is cyclic) that $\theta' = \psi^g$ for some $g \in G/H$.
The $H$-equivariant inclusion $\psi^g = \theta' \to \theta_{| H}$ then induces a non-zero
$G$-equivariant map
$\rho = \Ind_H^G \psi = \Ind_H^G \psi^g  \to \theta,$ 
which must be an isomorphism, since both its source and target are irreducible by assumption.
This proves the lemma.
\end{proof}

\begin{remark}
\label{rem:A4}
{\em If we take $G = A_4$ and $H$ to be the normal subgroup of $G$ of order four (so that
$H$ is a Klein four group), then the induction of any non-trivial character of $H$
gives an irreducible representation $\rho:G \to \mathrm{SO}_3(k)$.  For every $g \in G$,
the characteristic polynomial of $\rho(g)$ thus has $1$ as an eigenvalue, and so 
if $\theta$ denotes the trivial character of $G$, the element $\theta(g)$ is annihilated
by $\chara\bigl(\rho(g)\bigr)$ for every $g \in G$.  Thus the analogue of Lemma~\ref{lem:induction}
does not hold in general if $H$ is not cyclic.
}
\end{remark}

We thank Florian Herzig for providing the proof of the following lemma.

\begin{lemma}
\label{lem:florian}Suppose that $G$ is a finite subgroup of $\GL_n(k)$, which contains $\SL_n(k')$ for
some subfield $k'$
of $k$, and is contained in $k^{\times} \GL_n(k')$.
\begin{enumerate}
\item 
Any irreducible representation of $G$ over $k$ remains irreducible
upon restriction to $\SL_n(k')$.
\item
Any two irreducible representations of $G$ which become isomorphic upon
restriction to $\SL_n(k')$ can be obtained one from the other via twisting
by a character of $G$ that is trivial on $\SL_n(k')$.
\end{enumerate}
\end{lemma}
\begin{proof}
Let $G$ act via $\theta$ on the
  $k$-vector space $V$, and let $(\theta,W)$ be an irreducible
  subrepresentation of $\theta|_{\SL_n(k')}$. Then $W$ is obtained by
  restriction from a representation of the algebraic group $\SL_n/k'$
  (cf. Section~1 of~\cite{MR933356}), so the action of $\SL_n(k')$ on $W$ may be extended to an action of
  $\GL_n(k)$ and thus of $G$. By Frobenius reciprocity we obtain a
  surjective map $(\Ind_{\SL_n(k')}^G1)\otimes W\to V$ of
  $G$-representations. Since $G/\SL_n(k')$ is a finite abelian group
  of prime to $p$ order, we see that $(\Ind_{\SL_n(k')}^G1)$ is a
  direct sum of one-dimensional representations, so that $V$ is a
  twist of $W$ by some character which is trivial on $\SL_n(k')$. Thus
  the restriction of $\theta$ to $\SL_n(k')$ is just $W$, which is
  irreducible, proving~(1).

This same argument also serves to establish~(2).
\end{proof}

\begin{lemma}
\label{lem:big}
Assume that $p\ge n$. If $\rho:G \to \GL_n(k)$ is irreducible, and if $\SL_n(k') \subseteq \rho(G)
\subseteq k^{\times}\GL_n(k')$ for some
subfield $k'$ of $k$,
then $\rho$ satisfies Hypothesis~{\em \ref{hyp:char}}.
\end{lemma}

\begin{proof}The case $n=1$ follows from Lemma \ref{lem:characters},
  and as remarked above the case $n=2$ is proved in the course of the
  proof of Proposition 14.2 of \cite{MR488287}, so we may assume that
  $n\ge 3$.  By Lemma
  \ref{lem:kernels}, we may assume that $\rho$ is faithful, so that we
  can identify $G$ with $\rho(G)$. In particular, $\SL_n(k')$ is a
  subgroup of $G$, and by Lemma~\ref{lem:florian}, the
restriction of $\theta$ to $\SL_n(k')$ remains irreducible.

  Since $G$ is finite, our assumption that $\SL_n(k') \subset G$  implies that
  $k'$ is finite; suppose that $k'$ has cardinality $q$. We recall some basic facts
  about the representation theory of $\SL_n(k')$; see for example
  Section~1 of~\cite{MR933356}. The irreducible
  $k$-representations of $\SL_n(k')$ are obtained by restriction from
  the algebraic group $\SL_n/k'$, and are precisely those
  representations whose highest weights are $q$-restricted. (With the
  usual choice of maximal torus $T$ of $\SL_n$, if we identify the weight lattice with $\mathbb Z^n$ modulo the
diagonally embedded copy of $\mathbb Z$, a weight
  $(a_1,\dots,a_n)$ is $q$-restricted if $0\le a_i-a_{i+1}\le q-1$
  for all $1\le i\le n-1$.) Suppose that $\theta$ has highest weight
  $(a_1,\dots,a_n)$. 
Let $g\in\SL_n(k')$ be a semisimple element
  with eigenvalues $\alpha_1,\dots,\alpha_n$. Then, since $g$ is
  conjugate to an element of $T(k)$,  by considering the
  formal character of the corresponding representation of $\SL_n/k'$
  we see that among the eigenvalues of $\theta(g)$ are each of the
  quantities \[\prod_{i=1}^n\alpha_i^{x_i},\]where the $x_i$ are a
  permutation of $a_1,\dots,a_n$.

  Our assumption on $\theta$ and $\rho$ implies that for each such
  permutation, $\prod_{i=1}^n\alpha_i^{x_i}$ must equal one of
  $\alpha_1,\dots,\alpha_n$. In particular, if we let $\alpha$ be a
  primitive $(q^n-1)/(q-1)$-st root of unity, we may consider a
  semisimple element $g$ with eigenvalues
  $\alpha,\alpha^q,\dots,\alpha^{q^{n-1}}$. Then for any
  $x_1,\dots,x_n$ as above, there must be an integer $0\le \beta\le n-1$ such
  that \[q^{n-1}x_1+q^{n-2}x_2+\dots+x_n\equiv
  q^\beta\pmod{(q^n-1)/(q-1)}.\]

 Fix some $1\le i\le n-1$, and
  consider two permutations $x_1,\dots,x_n$ and $x'_1,\dots,x'_n$ as
  above, which satisfy $x_i=x'_i$ for $1\le i\le n-2$, $x_{n-1}=a_i$,
  $x_n=a_{i+1}$, $x'_{n-1}=a_{i+1}$ and $x'_n=a_i$. Taking the difference of the
  two expressions $q^{n-1}x_1+q^{n-2}x_2+\dots+x_n$ and
  $q^{n-1}x'_1+q^{n-2}x'_2+\dots+x'_n$, we conclude that there are
  integers $0\le \beta,\gamma\le n-1$ such that \[(q-1)(a_i-a_{i+1})\equiv
  q^\beta-q^\gamma\pmod{(q^n-1)/(q-1)}.\] Since $n\ge 3$, we have
  $0\le (q-1)(a_i-a_{i+1})\le (q-1)^2<(q^n-1)/(q-1)$, and we conclude
  that either $\beta\ge \gamma$ and
  $(q-1)(a_i-a_{i+1})=q^\beta-q^\gamma$, or $\beta<\gamma$ and
  $(q-1)(a_i-a_{i+1})=(q^n-1)/(q-1)+q^\beta-q^\gamma$. In the second case, we
  have $(q-1)^2\ge (q-1)(a_i-a_{i+1})=(q^n-1)/(q-1)+q^\beta-q^\gamma\ge
  (q^n-1)/(q-1)+1-q^{n-1}=q^{n-2}+\dots+q+2$. This is a contradiction
  if $n\ge 4$. If $n=3$, $(q^3-1)/(q-1)\equiv 3 \pmod{q-1}$, so that
  $(q-1)|3$, which is a contradiction as $p>2$.

  Thus it must be the case that $\beta\ge\gamma$ and $(q-1)(a_i-a_{i+1})=q^\beta-q^\gamma$,
  so that $a_i-a_{i+1}$ is congruent to $0$ or $1\pmod{q}$, and thus
  $a_i-a_{i+1}=0$ or $1$. In particular, for each $i$ we have $0\le a_i-a_n\le n-1\le
  p-1$. 

  We now repeat the above analysis. Fix some $1\le i\le n-1$, and
  consider two permutations $x_1,\dots,x_n$ and $x'_1,\dots,x'_n$ as
  above, which satisfy $x_i=x'_i$ for $1\le i\le n-2$, $x_{n-1}=a_i$,
  $x_n=a_{n}$, $x'_{n-1}=a_{n}$ and $x'_n=a_i$. Taking the difference
  of the two expressions $q^{n-1}x_1+q^{n-2}x_2+\dots+x_n$ and
  $q^{n-1}x'_1+q^{n-2}x'_2+\dots+x'_n$, and using that $0\le
  a_i-a_n\le p-1\le q-1$, we conclude as before that each $a_i-a_n=0$
  or $1$. Thus there must be an integer $1\le r\le n$ with
  $a_1=\dots=a_r=a_n+1$, $a_{r+1}=\dots=a_n$.

  Returning to the original
  congruences \[q^{n-1}x_1+q^{n-2}x_2+\dots+x_n\equiv
  q^\beta\pmod{(q^n-1)/(q-1)},\]we see that the left hand side is
  congruent to a sum of precisely $r$ distinct values $q^j$, 
  $1\le j \le n-1$. Thus $r=1$, and $\theta|_{\SL_n(k')}$ is the
  standard representation of $\SL_n(k')$,
  i.e. $\theta|_{\SL_n(k')}\cong\rho|_{\SL_n(k')}$.
  By part~(2) of Lemma~\ref{lem:florian},
  we see that there is a character
  $\chi:G\to k^\times$ with $\chi|_{\SL_n(k')}=1$ such that
  $\theta\cong\rho\otimes\chi$.

  To complete the proof, we must show that $\chi$ is trivial. Take
  $g\in G$; we will show that $\chi(g)=1$. If $g'\in G$ has $\det
  g'=\det g$, then $g(g')^{-1}\in\SL_n(k')$, so
  $\chi(g')=\chi(g)$. Firstly, note that by assumption we can write
  $g=\lambda h$, 
  with $\lambda\in k^\times$, $h\in \GL_n(k')$. Choose
  $h'\in\GL_n(k')$ to have eigenvalues $\{1,\dots,1,\det(h)\}$. Then
  $h(h')^{-1}\in\SL_n(k')\subset G$, so $g':=\lambda h'$ is an element
  of $G$. Then the hypothesis on $\theta$ and $\rho$ shows that the
  eigenvalues of $\chi(g')g' = \chi(g)g'$ are contained in the set of eigenvalues
  of $g'$, so that if $\chi(g)\ne 1$ we must have
  $\chi(g)=\det(h)=-1$. Assume for the sake of contradiction that this
  is the case. If $n$ is odd, then we now choose $h'$ to have
  eigenvalues $\{-1,\dots,-1\}$, and we immediately obtain a
  contradiction from the same argument. If $n$ is even then since
  $p\ge n$ we have $p\ge 5$ (recall that we are assuming $n\geq 3$),
  and we may choose $a\in (k')^\times$,
  $a\ne \pm 1$. Then choosing $h'$ to have eigenvalues
  $\{1,\dots,1,a,-1/a\}$ gives a contradiction.
\end{proof}

\subsection{Representations whose image is generated by regular elements}  
\label{subsec:reg}
Recall that a square matrix with entries in $k$ is said to be {\em regular} if
its minimal polynomial and characteristic polynomial coincide.

\begin{lem}
  \label{lem:regular implies equal determinants}If $\rho:G\to\GL_n(k)$ and
  $\theta:G\to\GL_n(k)$ are representations such that the image
  $\theta(G)$ is generated by its subset of regular elements, and for
  every $g\in G$ the characteristic polynomial of $\rho(g)$
  annihilates $\theta(g)$, then $\det\rho=\det\theta$.
\end{lem}
\begin{proof}
  Let $g\in G$ be an element such that $\theta(g)$ is regular. Then
  the characteristic polynomials of $\rho(g)$ and $\theta(g)$ must be
  equal (since the minimal and characteristic polynomials of
  $\theta(g)$ coincide, and the characteristic polynomial of $\rho(g)$
  annihilates $\theta(g)$), so $\det\rho(g)=\det\theta(g)$. Since
  $\theta(G)$ is generated by its subset of regular elements, the
  result follows.
\end{proof}

In fact, we actually need a slight generalisation of this result,
where we simply have a collection of characteristic polynomials,
rather than a representation $\rho$. Suppose that for each $g\in G$ we
have a monic polynomial $\rho_g(X)=X^n-a_{1}(g)X^{n-1}+\dots+(-1)^na_{n}(g)\in
k[X]$ of degree $n$, with the property that for all $g$, $h\in G$, we
have $a_{n}(gh)=a_{n}(g)a_{n}(h)$.
\begin{lem}
  \label{lem:regular implies equal dets, starting with char polys
    rather than a repn}Suppose that $\theta:G\to\GL_n(k)$ is a
  representation with the property that $\theta(G)$ is generated by
  its subset of regular elements, and that for each $g\in G$ we have
  $\rho_g(\theta(g))=0$. Then for each $g\in G$ we have $\det\theta(g)=a_{n}(g)$.
\end{lem}
\begin{proof}
  This may be proved in exactly the same way as Lemma \ref{lem:regular
    implies equal determinants}.
\end{proof}

Let $\rho:G \to \GL_3(k)$ be irreducible. Our goal is to give
criteria for $\rho$ to satisfy the following hypothesis, in order to
apply the previous lemmas:
\begin{hyp}
\label{hyp:regular}
{\em The image $\rho(G)$ is generated by its subset of regular elements.}
\end{hyp}

\begin{lemma}
\label{lem:regular}
If $\rho: G \to \GL_3(k)$ is irreducible, and if either:
\begin{enumerate}
\item $\rho$ is isomorphic to an induction $\Ind_H^G \psi$, where $H$ is a normal
subgroup of index $3$ in $G$ and $\psi:H \to k^{\times}$ is a
character, or
\item  $\rho(G)$ contains a regular unipotent element,
\end{enumerate}
then $\rho$ satisfies Hypothesis~{\em \ref{hyp:regular}}.
\end{lemma}
\begin{proof}
Suppose firstly that $\rho$ is isomorphic to an induction $\Ind_H^G
\psi$. Since $H$ is a proper subgroup of $G$, the set of elements
$G-H$ generates $G$. If $g\in G-H$ then the characteristic polynomial
of $\rho(g)$ is of the form $X^3-\alpha$. If $p\ne 3$ then this has
distinct roots, so $\rho(g)$ is regular, and if $p=3$ then it is easy
to check that $\rho(g)$ is the product of a scalar matrix and a
unipotent matrix, and is regular.

Suppose now that $\rho(G)$ contains a regular unipotent element. For
ease of notation, we will refer to $\rho(G)$ as $G$ from now on. Let
$H$ be the subgroup of $G$ generated by the regular elements, and
assume for the sake of contradiction that $H$ is a proper subgroup of
$G$. We claim that $H$ contains every scalar matrix in
$G$; this is true because the product of a scalar matrix and a
regular matrix is again a regular matrix. Consider an element $g\in
G-H$; since it is not regular, and not scalar, it acts as a scalar on
some unique plane in $k^3$. We write $\ell_g$ for the corresponding line
in $\mathbb{P}^2(k)$.

Let $h$ be the given regular unipotent element in $H$. Then $h$
stabilises a unique line in $k^3$, so a unique point $P\in\mathbb{P}^2(k)$. As
$g\in G-H$, we also have $gh\in G-H$. Then $\ell_g\cap \ell_{gh}$ is
non-empty, so there is a point $Q\in\mathbb{P}^2(k)$ which is fixed by $g$ and
$gh$. It is thus also fixed by $h$, so in fact $Q=P$. Since $g$ was an
arbitrary element of $G-H$, we see that every element of $G-H$ fixes
$P$, and since $G$ is generated by $G-H$, this implies that every
element of $G$ fixes $P$. This contradicts the assumption that $\rho$
is irreducible.
\end{proof}

\appendix
\section{Cohomology of pairs}\label{appendix: cohomology results on
  pairs etc}
\renewcommand{\theequation}{A.\arabic{subsection}.\arabic{subsubsection}}

\subsection{\'Etale cohomology of a pair}

Let $X$ be a scheme, finite-type and separated over a field,
let $Z$ be a closed subscheme, and write $j: U \hookrightarrow X$ for the open immersion
of the complement $U:= X\setminus Z$ into $X$.  As in Section
\ref{sec:cohom}, we let $E$ be an algebraic extension of $\Q_p$,
where $p$ is invertible on~$X$, 
let $k_E$ denote the residue field of $E$, and we let 
$A$ be one of $E$ or~$k_E$.
We define the \'etale cohomology of the pair $(X,Z)$ with coefficients
in $A$
to be the \'etale cohomology of the sheaf $j_! A$ on~$X$,
i.e.\ we write
$$H^{\bullet}_{\et}(X,Z, A) := H^{\bullet}_{\et}(X,j_! A).$$
If $i: Z \hookrightarrow X$ is the closed immersion of~$Z$, then the short exact sequence
$$0 \to j_! A \to A \to i_* A \to 0$$
gives rise to a long exact sequence
$$\cdots \to H^m_{\et}(Z, A) \to H^{m+1}_{\et}(X,Z, A) \to H^{m+1}_{\et}(X,A) \to H^{m+1}_{\et}(Z,A) \to \cdots,$$
which is the long exact cohomology sequence of the pair~$(X,Z)$.

We are particularly interested in the case of a pair $(X\setminus Y,Z\setminus Y)$, where $X$ is a smooth projective
variety over a separably closed field, and $Y$ and $Z$ are smooth
divisors on $X$ which meet 
transversely.
In this case we have a Cartesian diagram of open immersions
\numequation
\label{eqn:diagram}
\xymatrix{ X\setminus (Y \cup Z)  \ar^j[r]\ar^{k'}[d] & X\setminus Y \ar^k[d] \\
X\setminus Z \ar^{j'}[r] & X .}
\end{equation}
According to the above definition, the cohomology of the pair $(X\setminus Y,Z\setminus Y)$
is computed
as the cohomology of the sheaf $j_! A$ on $X\setminus Y$, which is canonically isomorphic to the cohomology
of the complex $Rk_* j_! A$ on $X$. 
An important point is that there is a canonical isomorphism
\numequation
\label{eqn:Faltings}
j'_! R k'_* A\iso Rk_* j_! A.
\end{equation}
(See the discussion of \S III (b) on p.~44 of \cite{Fal}.)  

\subsection{Verdier duality.}
Verdier duality (\cite[Exp.\ XVIII]{MR0354654}, \cite{MR0230732}) states that if $f:X \to S$ is a
morphism of finite-type and separated schemes
over a separably closed field $k$, then for any constructible \'etale $A$-sheaves $\cF$ on $X$ and $\cG$ on $S$,
there is a canonical isomorphism (in the derived category) of complexes of \'etale sheaves on $S$
$$
R\!\CHom(Rf_! \mathcal F,\cG) \cong
Rf_* R\!\CHom(\mathcal F,f^!\mathcal G).$$ 
We recall some standard special cases of this isomorphism, in the context
of the diagram~(\ref{eqn:diagram}).

Taking $f$ to be $k'$ (and recalling that $k'$ is an open immersion), we obtain an isomorphism 
$$
R\!\CHom(k'_!A,A) \cong
 Rk'_* R\!\CHom(A,A) 
= Rk'_* A,$$
and hence by double duality,
an isomorphism
\numequation
\label{eqn:k prime case}
R\!\CHom(Rk'_*A,A)
\cong
k_!' A .
\end{equation}
Next, taking $f$ to be $j'$, and taking into account~(\ref{eqn:Faltings})
and~(\ref{eqn:k prime case}),
we obtain isomorphisms
$$R\!\CHom(Rk_* j_! A, A) \cong R\!\CHom(j'_! Rk'_*A,A) \cong Rj'_* R\!\CHom(Rk'_*A,A)
\cong Rj'_* k_!'A.
$$
Finally, taking $f$ to the natural map $X \to \Spec k$, and
recalling that in this case $f^! A = A[2d](d)$ (where $d$ is the
dimension of $X$; cf. \cite[Exp.\ XVIII Thm.\ 3.2.5]{MR0354654})
and that $Rf_*= Rf_!$ (since $f$ is proper, the variety $X$ being projective 
by assumption),
we find that
$$R\!\CHom(Rf_* Rk_* j_!A, A)
\cong
Rf_* R\!\CHom\bigl(Rk_* j_! A, A[2d](d)\bigr) \cong
Rf_*Rj'_* k_!'A[2d](d).$$Passing to cohomology,
we find that $H^m_{\et}(X\setminus Y,Z\setminus Y , A)$ is in natural duality
with $H^{2d -m}(X\setminus Z, Y\setminus Z, A)(d)$.

\subsection{Vanishing outside of, and torsion-freeness in, the middle degree}
\label{subsec:vanishing outside middle}
We continue to assume that $X$ is a smooth projective variety of dimension $d$
over the separably closed field $k$, and that $Y$ and $Z$ are smooth
divisors on $X$ which meet transversely,
but in addition, we now assume that the complements $X\setminus Y$ and $X\setminus Z$
are affine (and hence also that $Z\setminus Y$ and $Y\setminus Z$ are affine).
This latter assumption implies that $H^m_{\et}\bigl(X\setminus Y,
A\bigr)$ vanishes if $m>d$, and that $H^m_{\et}\bigl(Z\setminus Y, A\bigr)$ vanishes if
$m\ge d$
(\cite[Exp.\ XIV Cor.\ 3.3]{MR0354654}).
By the long exact cohomology sequence of the pair
$(X\setminus Y,Y\setminus Z)$, we see that $H^m_{\et}\bigl(X\setminus Y, Z\setminus Y, A\bigr)$
vanishes if $m>d$. Similarly, we see that $H^{2d -m}(X\setminus Z,
Y\setminus Z, A)(d)$ vanishes if $m<d$. Hence, by the duality between
$H^m_{\et}(X\setminus Y,Z\setminus Y , A)$ and $H^{2d -m}(X\setminus
Z, Y\setminus Z, A)(d)$, we find that both vanish unless $m=d$.

 Let $\mathcal O_E$ denote the ring of integers in $E$. Suppose
 momentarily that $E/\Qp$ is finite, 
and let $\unif$ be a uniformiser of $\cO_E$.  
From a consideration of the cohomology long exact sequence arising
from the short exact sequence of sheaves
$$0 \to \mathcal O_E/\unif^n \buildrel \unif \cdot \over \longrightarrow
\mathcal O_E/\unif^{n+1} \longrightarrow k_E \to 0,$$
and arguing inductively on $n$,
we find that 
$H^m_{\et}(X\setminus Y,Z\setminus Y , \mathcal O_E/\unif^n )$
vanishes in degrees other than $m = d$ for all $n$.
Passing to the projective limit over $n$, we see that the same is true of 
$H^m_{\et}(X\setminus Y,Z\setminus Y , \mathcal O_E).$
Finally, a consideration of the cohomology long exact sequence arising
from the short exact sequence $$0 \to \mathcal O_E \buildrel \unif \cdot \over
\longrightarrow \mathcal O_E \longrightarrow k_E \to 0$$
shows that 
$H^d_{\et}(X\setminus Y,Z\setminus Y , \mathcal O_E)$
is torsion-free.

By passage to the direct limit over subfields of $E$ which are finite
over $\Qp$, we see that these properties continue to hold for
arbitrary algebraic extensions $E/\Qp$.

\printbibliography

\end{document}